\documentclass[DIV=10]{scrartcl}% Math and Physical Sciences Reference Style
\usepackage{subcaption}
\usepackage{comment}
\usepackage{tabularx}
\usepackage{makecell}
\usepackage{amsfonts}
\usepackage{amsfonts, amsmath, amssymb}
\usepackage{amsthm}
\usepackage{csquotes}
\usepackage{enumitem}
\usepackage{mathtools}
\usepackage[%
    hyperfootnotes=false,
    colorlinks = true,
    final=true,
    plainpages=false,
    pdfstartview=FitV,
    pdftoolbar=true,
    pdfmenubar=true,
    pdfencoding=auto, % Allows symbols in toc and content oveview
    psdextra, % same as above
    bookmarksopen=true,
    bookmarksnumbered=true,
    breaklinks=true,
    linktocpage=true,
    citecolor=blue,
    linkcolor=blue]%
    {hyperref}
\usepackage[nameinlink]{cleveref}
\Crefname{figure}{}{}

\usepackage[style=numeric-comp,%
			defernumbers=true,
			useprefix=true,%
			giveninits=true,%
			hyperref=true,%
			uniquename=init,%
			sorting = none,
			sortcites=false,% sort citations when multiple entries are passed to one cite command
			doi=true,%
			isbn=false,%
			url=false,%
			backend=biber%
	]{biblatex}
\usepackage{graphicx}
\usepackage{titling}
\usepackage{authblk}

\newtheorem{theorem}{Theorem}%  meant for continuous numbers
\newtheorem{lemma}[theorem]{Lemma}%  meant for continuous numbers
\newtheorem{corollary}[theorem]{Corollary}

\newtheorem{example}{Example}%

\newtheorem{definition}{Definition}%

\newcommand{\R}{{\mathbb R}}
\newcommand{\C}{{\mathbb C}}

\newcommand{\N}{{\mathbb N}}
\newcommand{\T}{{\mathbb T}}

\newcommand{\Z}{{\mathbb Z}}

\newcommand{\op}{A}
\newcommand{\solu}{u}
\newcommand{\data}{f}

\newcommand{\opDiscrete}{\ensuremath{\tilde{A}}}
\newcommand{\solutionDiscrete}{\ensuremath{\tilde{u}}}

\newcommand{\noise}{n}
\newcommand{\dataNoisyDiscrete}{\ensuremath{\tilde{f}^\delta}}

\newcommand{\net}{\mathcal{G}}
\newcommand{\params}{\theta}
\DeclareMathOperator{\sgn}{sgn}

\addtokomafont{disposition}{\rmfamily}
\makeatletter
\let\blx@rerun@biber\relax
\makeatother

\title{A neural operator view on U-Nets for inverse
imaging problems}

\author[1]{Alexander Auras}
\affil[1]{Computer Vision Group, University of Siegen, H\"olderlinstraße 3, 57076 Siegen, Germany,}
\affil[ ]{\{alexander.auras, michael.moeller, michael.schopf\}@uni-siegen.de}
\author[2,3]{Martin Burger}
\affil[2]{Helmholtz Imaging, Deutsches Elektronen-Synchrotron DESY, Notkestr. 85, Hamburg, 22607,
Germany,}
\affil[ ]{\{martin.burger, samira.kabri\}@desy.de}
\affil[3]{Fachbereich Mathematik, Universit\"at Hamburg, Bundesstrasse 55, Hamburg, 20146, Germany}
\author[2]{Samira Kabri}
\author[1]{Michael Moeller}
\author[1]{Michael Schopf-Kuester}
\date{}

\begin{document}

\maketitle

\abstract{Deep neural networks have shown great empirical success in the solution of a wide variety of ill-posed inverse problems in imaging. Yet, very few works have studied their behavior in the limit that turns the discretized ill-conditioned problems into truly ill-posed ones, i.e., for an increasing resolution of the discretization. In this work, we review common approaches to \textit{neural operator} learning in architectures that resemble a U-Net, one of the most common classical architectures for inverse imaging problems. We discuss advantages and drawbacks of the respective approaches, consider a 1D toy example for improved interpretability, and present extensive numerical experiments on how different types of neural operator U-Nets can improve a first (crude) limited angle CT-reconstruction. In particular, we study how well networks trained for a certain resolution of the discretization generalize to other resolutions. Our finding is that while U-shaped neural operator architectures are by design resolution-invariant, the classical U-Net architecture seems to be more robust with respect to resolution changes than expected.}

\section{Introduction}
\label{aa_sec:intro}
In this work, we study the solution of ill-posed linear inverse problems in imaging with the help of deep learning. We consider an underlying (infinite-dimensional) inverse problem of recovering $\solu$ from $\data$ given by 
\begin{align}
\label{aa_eq:inverseProbInfiniteDim}
    \data = \op(\solu) 
\end{align}
for a compact linear operator $\op$ with infinite-dimensional range. Under these conditions, the pseudo-inverse of $\op$ becomes discontinuous (cf. \cite{engl1996regularization_AA}), making the problem \eqref{aa_eq:inverseProbInfiniteDim} ill-posed. \vspace*{2mm}\\
\textbf{Ill-posed inverse problems and regularization} \\
In any practical setting, one can only measure finitely many (discrete) data points, commonly interpreted as the integration of the continuous function $\op(\solu)$ over areas of individual sensors plus some measurement noise, e.g., 
\begin{align}
\label{aa_eq:inverseProbSemifiniteDim}
    \dataNoisyDiscrete_{i} = \int_{S_{i}} \op(\solu)(x) ~dx + \noise,
\end{align}
for a sensor area $S_{i}$ and measurement noise $\noise$, giving rise to discrete measurements $\dataNoisyDiscrete \in \mathbb{R}^{N_\data}$ for a certain resolution $N_\data$.\footnote{We remark that in imaging, the measured data often has the interpretation of a two-dimensional matrix $\mathbb{R}^{N_1 \times N_2}$, which, however, is isomorphic to $\mathbb{R}^{N_1N_2}$ such that we consider the vectorized version only.} Naturally, one chooses to reconstruct a discretization $\solutionDiscrete \in \mathbb{R}^{N_\solu}$, where a reasonable choice of $N_\solu$ depends on the number of measurements $N_\data$. Model-based reconstructions therefore determine a discrete approximation $\opDiscrete \in \mathbb{R}^{N_\data \times N_\solu}$ of the operator $\op$, and solve a regularized version of 
\begin{align}
\label{aa_eq:discreteInvProb}
\dataNoisyDiscrete \approx \opDiscrete \solutionDiscrete. 
\end{align}
While \eqref{aa_eq:discreteInvProb} is finite dimensional, such that the solution always depends on the data continuously, the ill-posedness of \eqref{aa_eq:inverseProbInfiniteDim} typically makes \eqref{aa_eq:discreteInvProb} ill-conditioned, with the operator norm $\|\opDiscrete\|$ increases with increasing $N_\data$, $N_\solu$.  Thus, model-based approaches have focused on developing suitable regularization strategies as well as regularization parameter choice rules for the resulting approaches to be provably convergent, even in the infinite-dimensional case \eqref{aa_eq:inverseProbInfiniteDim}.  

The difficulty of modeling suitable regularizers (or priors) has led to the rise of deep learning approaches for inverse problems. In the simplest case, such approaches use a simple (crude) estimate of $\solutionDiscrete$, e.g. $\solutionDiscrete^0 = \opDiscrete^T \dataNoisyDiscrete$, as an input to a neural network $\net$ and \textit{learn} to reconstruct a better approximation $\net(\solutionDiscrete^0;\params)$ from the initial one by training the parameters $\params$ on a large amount of (typically simulated) training examples for which the ground truth is known.
Another method is to train a denoiser as a so-called plug-and-play prior, and incorporate it into an iterative reconstruction approach (cf. \cite{hurault2023gradient_AA,venkatakrishnan2013pnp_AA}). Recently, generating samples from the posterior distribution with the help of score-based diffusion (cf. \cite{feng2023score_AA,hagemann2023multilevel_AA,Lim25ScoreBased_AA,song2019generative_AA,Welker2025Ptycho_AA,Xu24Provably_AA}) and Bayesian flow matching (cf. \cite{steidl2025flowmatching_AA}) has gained large interest. Many of these approaches rely on U-Nets: Originally designed for image segmentation in \cite{ronneberger2015unet_AA, shelhamer2017fully_AA}, these U-shaped neural networks appear in end-to-end approaches as proposed in \cite{Feng20Endtoend_AA}, in the DRUNET architecture \cite{Zhang2022dpirdrunet_AA}, a popular choice for plug-and-play priors, and in the NCSN architecture \cite{song2019generative_AA} used to approximate the score function in score-based diffusion, to name just a few.
While such approaches have been extremely successful in various applications (cf. \cite{auras2024overview}), most practical works consider the discretization $N_\data$ and $N_\solu$ to be fixed, neglecting the ill-posed nature of the underlying continuous formulation. 
\vspace*{2mm}\\\textbf{Neural Operators} \\
The field of \textit{Neural Operator Learning} \cite{berner2025principled_AA, kovachki2023neural_AA} has focused on the approximation of operators independent of the discretization of the input signal. It primarily deals with the question of a consistent continuum interpretation of linear layers, in particular convolutional layers, to allow for varying discretizations of the input. 
The original formulation of neural operators proposed in \cite{kovachki2023neural_AA} views convolutional layers as approximations of integral operators. A famous representative of this interpretation is the Fourier neural operator (FNO) proposed in \cite{li2021fourier_AA}. Here, convolution kernels are parametrized by their Fourier coefficients, allowing for the effortless approximation of a continuous convolution using trigonometric interpolation of the kernels. FNOs and FNO-derived architectures have been used successfully, especially in the context of partial differential equations (PDEs) (cf. \cite{Bonneville2025uafno_AA, li2023fourier_AA, li2021fourier_AA, liu2025uffno_AA, rahman2023uno_AA, tranfactorized_AA, wen2022ufno_AA}). In this chapter, we pay special attention to the U-shaped neural operator (UNO) that is based on the FNO and proposed in \cite{rahman2023uno_AA}. 

The parametrization by Fourier coefficients makes it more complicated to restrict the spatial support of convolution kernels. A natural approach to combine spatial locality with the approximation of an integral operator is to interpolate local kernels while keeping the size of their support fixed. The work \cite{liuschiaffini2024localno_AA} integrates this idea into neural operators and uses the discrete-continuous (DISCO) convolution established in \cite{ocampo2023scalable_AA} to implement so-called local convolutions.

Extending the integral operator view of \cite{kovachki2023neural_AA} to differential operators, it was further pointed out in \cite{liuschiaffini2024localno_AA} that classical convolutional layers can also be used to approximate spatial derivatives. This seems especially interesting in the context of finite difference approximations of partial differential equations that play an important role in model-based image processing (cf. \cite{Aubert2006_AA, Schoenlieb2015_AA}) and have further already motivated the neural network architectures proposed in \cite{ruthotto2020deep_AA}.

A more restrictive version of neural operators is given by  representation equivalent neural operators (ReNOs) described in \cite{bartolucci2023representation_AA}. In particular, ReNOs are required to respect a function's band-limit, which is usually not fulfilled with the usual activation functions. Motivated by the concept of representation equivalence, \cite{raonic20cno_AA} propose the convolutional neural operator (CNO). Based on a classical U-Net architecture, the CNO is made representation equivalent by additional upsampling and downsampling operations. In particular, the input resolution of a CNO has to be fixed, requiring a resizing operation for inputs of different resolutions.

While neural operators have also been applied in imaging applications, e.g., classification \cite{Johnny2022fno_AA, Kabri2023FNO_AA}, super-resolution \cite{Wei2023superres_AA} and image generation \cite{hagemann2023multilevel_AA}, the study of U-shaped neural operators is, to the best of the authors' knowledge, limited to the approximation of solution operators for partial differential equations.  With this chapter, we want to contribute to closing this gap. We focus on architectural design choices and their interpretation in the infinite dimensional setting. Of course, for a full understanding, quantities like the approximation error resulting from the discretization and the generalization error to unseen data have to be considered, as well as how training with data of different resolutions influences the training dynamics. While we do not cover these aspects in detail, we refer to \cite{kovachki2021fnoapprox_AA,kovachki2023neural_AA} for approximation theory, to \cite{deHoop2023operatorlearning_AA, Reinhardt2024Operator_AA} for statistical learning theory for neural operators and operator learning in general and to \cite{koshizuka2024express_AA,rowbottom2025multi_AA} for recent works on training aspects of FNOs.
\vspace*{2mm}\\\textbf{Structure of this work} \\ The remainder of this work is organized as follows: In Section \ref{aa_sec:discrete} we recall the classical U-Net architecture. In Section \ref{aa_sec:integral} we explore ways to transform the classical U-Net into a U-shaped neural operator following the integral operator interpretation. More precisely, in Subsection \ref{aa_subsec:fno} we re-derive the spectral convolution which is at the heart of FNO architectures and propose a simplistic U-shaped neural operator architecture. In Subsection \ref{aa_subsec:interpolating} we investigate how local convolutional layers could make the proposed architecture more efficient. In Section \ref{aa_sec:differential} we discuss possibilities to integrate differential layers into neural operators for image processing. In Section \ref{aa_sec:numexp} we illustrate different architectures with a $1$-D example of sparse signal deblurring and provide large scale numerical experiments for the post-processing of limited angle CT data. We summarize our findings in Section \ref{aa_sec:conclusion}. All code for reproducing our results will be made available publicly.\footnote{The code to reproduce our results is available at \url{https://github.com/AlexanderAuras/neural-operator-view-on-unets}}

\section{Discrete U-Net}\label{aa_sec:discrete}
In this section, we summarize common architectures for classical U-Nets. In most architectures, all linear layers are convolutional layers and the nonlinear activation function is usually the Rectified Linear Unit (ReLU). The ``U''-shape that characterizes a classical U-Net is caused by down-sampling  in the first half (contracting path) and up-sampling operations in the second half (expanding path) of the architecture. To make up for the information loss, it is common practice to increase the number of channels after down-sampling and to decrease them after up-sampling. The transmission of high frequency information is ensured by so-called skip-connections, which directly feed the outputs of the contracting path into the corresponding levels of the expanding path. We refer to Figure \ref{aa_fig:UNetArchi} for a visualization of the general structure. In the following, we give more detailed descriptions of the components mentioned above. We use the notation \begin{equation*}
    [n] := \left\{-\left\lfloor \frac{n}{2} \right \rfloor, \hdots, 0, \hdots,  \left\lceil \frac{n}{2} \right \rceil -1 \right\}.
\end{equation*} for $n \in \N$ and the corresponding zero-centered indexing, i.e., $u \in \R^{n}$ can be written as $(u_i)_{i \in [n]}$.

\begin{figure}
    \centering
    \includegraphics[width = 0.8\textwidth]{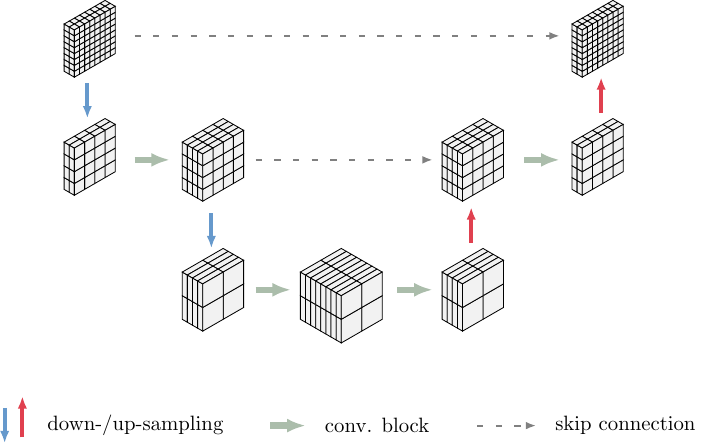}
    \caption{The general form of a U-Net.}
    \label{aa_fig:UNetArchi}
\end{figure}
\subsection{Convolutional blocks}\label{aa_sec:convblocks}
A convolutional block is a stack of convolutional layers and activation functions in alternating order.  For the sake of simplicity, we describe a convolutional block that only consists of one convolutional layer, followed by one application of the activation function.
The discrete convolution of an input $u \in \R^{n \times n}$ with a discrete convolution kernel $\kappa \in \R^{m \times m}$, $m \leq n$, can be formally defined as 
\begin{equation}\label{aa_eq:discconvn}
    (\kappa * u)_i := \sum_{j \in [m]^2} \kappa_{j} u_{i-j}, \qquad \text{for all } i \in [n]^2.
\end{equation}
This notion requires a definition of $u_i$ for $i \notin [n]^2$. In practice, filling up the values of $u$ for unknown indices is called padding and the so-called padding mode determines how the padding is performed. For completeness, we further note that it is also common to compute only $(\kappa * u)_i$ for $i$ for which there  don't appear any undefined indices in \eqref{aa_eq:discconvn}. In this work, we are interested in two different padding modes:
\begin{enumerate}
    \item Zero-padding: Unknown values are set to zero.
    \item Circular padding: Unknown values are determined by the periodic extension.
\end{enumerate}
While zero-padding is more popular than circular padding (images are usually not assumed to be periodic), circular padding allows for a change of basis to the Fourier domain, which will be useful in Section \ref{aa_sec:integral}. We further note that there is no need for the dimensions of $\kappa$ and $u$ to match, as the sum in \eqref{aa_eq:discconvn} is only taken over $[m]^2$, which corresponds to zero-padding of the kernel $\kappa$. This gives the key-insight how the convolutional layer behaves for inputs of different dimensions, or, from a function space perspective, of different resolution. We illustrate this in the following example.
\begin{example}\label{aa_ex:localizingkernel} Consider a function $u: \R^2 \rightarrow \R$ and its discretizations $u^n \in \R^{n\times n}$, $u^n := (u(i/n))_{i \in [n]^2}$. To get a well-defined definition of the discrete convolution, we can assume that the support of $u$ lies in $[-0.5,0.5]^2$ (zero-padding) or that $u$ is $1$-periodic (circular padding). For a fixed kernel $\kappa \in R^{m \times m}$, we get 
\begin{equation*}
    (\kappa*u^n)_i = \sum_{j \in [m]^2} \kappa_{j} \,u\left(\frac{i-j}{n}\right), \qquad \text{for all } i \in [n]^2,
\end{equation*}
meaning that the support of the kernel, or its receptive field, is given by $[0, m/n]^2$ and therefore shrinks as $n$ grows. In other words, the discrete convolution with a fixed kernel becomes more and more local the higher the resolution of the input is.
\end{example}
So far, we have omitted the so called channels, i.e., inputs $u \in \R^{c_{\text{in}}\times n \times n}$, that can be motivated by vector-valued functions. A convolutional layer parametrized by $\kappa \in \R^{c_{\text{out}}\times c_{\text{in}}\times m \times m}$ has $c_{\text{in}}$ input channels and $c_{\text{out}}$ output channels and is defined by
\begin{equation*}
    (\kappa * u)_{c} := \sum_{d \in [c_{\text{in}}]} \kappa_{c,d}*u_d, \qquad \text{for all } c \in [c_{\text{out}}].
\end{equation*}
Concatenating the convolutional layer with an activation function $\sigma: \R \times \R$, we obtain the simple convolutional block $C_\kappa: \R^{c_{\text{in}}\times n \times n} \rightarrow \R^{c_{\text{out}}\times n \times n}$,
\begin{equation}\label{aa_eq:cnnconvlayer}
    C_{\kappa}(u) := \sigma(\kappa * u),
\end{equation}
where the activation function is applied elementwise. In practice, the above form is often equipped with an additive bias $b \in \R^{c_{out}}$, which is added channel-wise to the output of the convolutional layer, leading to
\begin{equation*}
    C_{\kappa, b}(u) := \sigma(\kappa * u + b).
\end{equation*}
While the latter version is used in the numerical experiments, we use \eqref{aa_eq:cnnconvlayer} in the majority of our theoretical studies. However, all results can be extended to the biased operator, as $b \in \R^{c_{out}}$ can always be interpreted as a vector-valued function that is constant in each component.
We further note that in both cases, the only dimensions that are inherent to the parametrization by $\kappa$ and $b$ are $c_{\text{in}}$ and $c_{\text{out}}$. In particular, the image size $n$ is not determined by the size $m$ of the kernel.
\subsection{Downsampling, Upsampling and Skip Connections}
Different methods for decreasing the spatial resolution to transfer to a lower level of the U-Net in the encoding path have been proposed: The original paper \cite{ronneberger2015unet_AA} used max-pooling, i.e., took the largest value in every distinct $2\times 2$ patch. Alternatives are average pooling (=taking the mean instead of the maximum), strided convolutions (=learning weights for combining the entries in a patch), or other (non-linear) pooling layers. 

The reverse process of increasing the resolution in the decoder-path of the U-Net is typically done by direct upsampling (interpolation) or the use of transpose convolution, i.e., the adjoint operator to a convolution with stride. 

Skip connections in U-Nets are typically concatenating skip-connections, i.e., the operation of appending the channels of the output of a convolutional layer on the encoder side to the respective channels of the same spatial resolution on the decoder side. Alternatives are additive (residual) connections (for improved efficiency) or more sophisticated processing connections such as attention gates in \cite{erisen2024_AA}.

Typically, data-driven models based on U-Net architectures are optimized specifically for a single input resolution. The driving question of the remaining chapters will thus be ``How can we interpret a discrete U-Net architecture from the neural operator perspective?''. In detail, we consider infinite dimensional architectures that can be reduced to conventional U-Nets via discretization. The focus will be on the infinite dimensional interpretation of the convolution. As discussed in  \cite{liuschiaffini2024localno_AA}, the discrete convolution \eqref{aa_eq:discconvn} can result from the discretization of integral operators, but also differential operators. Therefore, there is no unique answer to the posed question and a suitable interpretation depends on the given application and in particular, on the specific ``role'' a convolutional layer fulfills within the U-Net. Our derivations aim to explain both interpretations of discrete convolutions as integral operators and as differential operators and why they could both be reasonable choices in the context of image processing.
\section{U-Nets with integral operators}
\label{aa_sec:integral}
In this section, we want to interpret the convolutional layers in a discrete U-Net as discretized integral operators. Based on the spectral convolution proposed in \cite{li2021fourier_AA} we derive a simple neural operator architecture that can be seen as a special case of the UFNO \cite{rahman2023uno_AA}, while still being very close to its discrete counter part. 

For the integral operator perspective, we interpret the discrete convolution \eqref{aa_eq:discconvn} to be an approximation of the convolution of two functions $\kappa: \Omega \rightarrow \R$ and $u: \Omega \rightarrow \R$,
\begin{equation}\label{aa_eq:contconv}
    (\kappa * u)(x) = \int_{\Omega} \kappa(y)\,u(x-y)\,dy,
\end{equation}
where we choose $\Omega = (-0.5,0.5)^2$. Again, one has to provide an extension of $u$ to make $(\kappa * u)$ well-defined on the entire domain $\Omega$. Assuming both functions to be smooth enough and $n \in \N$ large enough, using standard mid-point quadrature (cf. \cite[Chapter 9]{Quarteroni2007_AA}), we get
\begin{equation}
\label{aa_eq:discreteIntegralApproximation}
    (\kappa * u)(x_i) \approx \frac{1}{n^2} (\kappa^n * u^n)_i \qquad \text{for all }i \in [n]^2,
\end{equation}
where we use the notation $x_i := i/n$, $\kappa^n := (\kappa(x_i))_{i \in [n]^2}$ and $u^n = (u(x_i))_{i \in [n]^2}$. It follows directly from example \ref{aa_ex:localizingkernel} that the discrete convolution implemented by \eqref{aa_eq:discconvn} only allows for the above interpretation if $n$, i.e., the resolution of the inputs, is fixed. For two different resolutions, the same discrete kernel corresponds to the approximation of the convolution with two different kernel functions. The seminal work on neural operators \cite{kovachki2023neural_AA} proposes various options to parametrize convolutional layers such that they indeed approximate a convolution for varying input resolutions. In this chapter, we concentrate on the parametrization of kernels by their Fourier coefficients, as proposed for Fourier Neural Operators \cite{li2021fourier_AA}.
\subsection{Spectral architectures based on trigonometric interpolation}
\label{aa_subsec:fno} We first briefly recapitulate the basics of discrete Fourier analysis to motivate the implementation of the spectral convolution which is at the heart of Fourier neural operators. In a next step, we derive a U-shaped neural operator based on the spectral convolution. To keep the architecture as simple as possible, we omit the so-called linear transform that is typically used in the FNO-layer and only add residual connections.\footnote{In this context, the linear transform refers to a linear combination of the channels. Substituting it by residual connections is thus equivalent to fixing the linear transform to be the identity.} While in  \cite{rahman2023uno_AA} it is proposed to contract and expand the function domain to create a U-shape, we propose to first decrease and then increase the number of parameters used. For the sake of simpler notation, the results we provide here correspond to the general case that $u \in \C^{n\times n}$. Restricting $u$ to be real-valued requires some careful adaptations for even $n$. We refer the interested reader to \cite{Kabri2023FNO_AA}. All results we state in this section are well-known results from Fourier analysis (cf. \cite{Grafakos2014Fourier_AA}), but we restate them in our specific setting to improve readability.
\subsubsection{Basic definitions and results}
\label{aa_subsec:fourier_transformation}
We define the discrete Fourier transform of $u \in \C^{n\times n}$ as
\begin{equation}\label{aa_eq:dft}
    (\mathcal{F}u)_k := \frac{1}{n^2}\sum_{j \in [n]^2} u_j \,e^{-2\pi i\left\langle\frac{j}{n}, k\right\rangle} \qquad \text{for all } k \in [n]^2,
\end{equation}
where here and in the following, $i$ denotes the imaginary unit and $\langle \cdot, \cdot \rangle$ denotes the canonical inner product.
We restrict the support to $[n]^2$ due to the $1$-periodicity of $e^{-2\pi i\,x}$.  
The inverse discrete Fourier transform for $\hat{u} \in \C^{n \times n}$ is then given by 
\begin{equation}\label{aa_eq:idft}
    (\mathcal{F}^{-1}\hat{u})_j = \sum_{k \in [n]^2} \hat{u}_k\, e^{2\pi i\left\langle\frac{j}{n}, k\right\rangle} \qquad \text{for all } j \in [n]^2.
\end{equation}
Other variants of the discrete Fourier transform mainly differ from \eqref{aa_eq:dft} by the normalization factor $1/n^2$: Multiplying by $1/n$ in both, the ``forward'' and inverse discrete Fourier transform, yields a unitary operator, while only normalizing the inverse Fourier transform by $1/n^2$ is probably the most common choice. However, the normalization used in \eqref{aa_eq:dft} simplifies notation in our specific setting, where we are interested in Fourier coefficients of functions that are discretized with a varying number of sampling points. In this sense, the factor $1/n^2$ corresponds to approximating the Fourier coefficients of a periodic function defined on the two-dimensional torus $\T^2 = \R^2 \Big / \Z^2$. This interpretation becomes rigorous with the following lemma.
\begin{lemma}[cf. {\cite[Definition 3.1.4]{Grafakos2014Fourier_AA}}]\label{aa_lem:exactdft}
    For $\hat{u} \in \C^{m\times m}$ consider $u: \mathbb{T}^2 \rightarrow \C$,
    \begin{equation}\label{aa_eq:trigonoPoly}
        u(x) = \sum_{k \in [m]^2} \hat{u}_k\,  e^{2\pi i\,\langle x, k\rangle}.
    \end{equation}
Then it holds for any $n \geq m$ that
\begin{equation*}
    (\mathcal{F}u^n)_k = \begin{cases}
        \hat{u}_k & \text{if } k \in [m]^2,\\
        0 & \text{if } k \in [n]^2\setminus [m]^2.
    \end{cases}
\end{equation*} 
where $u^n = (u(j/n))_{j \in [n]^2}$ denotes the discretization of $u$ using $n^2$ equidistant sampling points.
\end{lemma}
\begin{proof}
The proof follows from inserting the inverse Fourier transform into \eqref{aa_eq:trigointerp}.
\end{proof}
Lemma \ref{aa_lem:exactdft} shows that for suitable functions, we can compute all non-trivial Fourier coefficients with the discrete Fourier transform of a discretization with equidistant samples, provided that the number of samples is large enough. The functions that fulfill \eqref{aa_eq:trigonoPoly} are called trigonometric polynomials and can be seen as the periodic version of band-limited functions. In this sense, Lemma \ref{aa_lem:exactdft} is the trigonometric polynomial counterpart of the sampling theorem for band-limited functions (cf. \cite[Theorem 6.6.9]{Grafakos2014Fourier_AA}). The trigonometric interpolation formula that is analogous to $\operatorname{sinc}$-interpolation for band-limited functions would take the form
\begin{equation*}
    u(x) = \sum_{j \in [n]^2} u(j/n)\, \psi_n(x-j/n), \qquad \text{with }\;\psi_n(x) := 1/n^2 \sum_{k \in [n]^2} e^{2\pi i\langle x, k \rangle },
\end{equation*}
which is equivalent to zero-padding the Fourier coefficients, i.e.,
\begin{equation}\label{aa_eq:trigointerp}
    u(x) = \sum_{k \in [n]^2} \hat{u}_k\,  e^{2\pi i\,\langle x, k \rangle} \qquad \text{with }  \hat{u}_k = (\mathcal{F}u^n)_k,\, k\in [n]^2.
\end{equation}
Having established the definition of the discrete Fourier transform, we restate the discrete convolution theorem, which motivates the implementation of the convolution of two functions using their Fourier coefficients.
\begin{theorem}[Discrete Convolution Theorem, cf. {\cite[Proposition 3.1.2, (9)]{Grafakos2014Fourier_AA}}]\label{aa_thm:discconvthm}
    Let $u, \kappa \in \C^{n\times n}$. Then it holds for the discrete convolution \eqref{aa_eq:discconvn} with circular padding of $u$ that
    \begin{equation*}
        \frac{1}{n^2}(\kappa * u) = \mathcal{F}^{-1}(\mathcal{F}\kappa \cdot \mathcal{F}u),
    \end{equation*}
    where we define the point-wise multiplication operator $\cdot$ for $\hat{u}, \hat{\kappa} \in \C^{n\times n}$ as \begin{equation*}
        \hat{\kappa} \cdot \hat{u} := (\hat{\kappa}_k\hat{u}_k)_{k \in [n]^2}.
    \end{equation*}
\end{theorem}
\begin{proof}
    The proof follows from inserting the definitions of the discrete convolution \eqref{aa_eq:discconvn} and the discrete Fourier transform and its inverse into the claimed formula and reordering the sums.
\end{proof}
\subsubsection{Spectral convolutional layers}
Up to a rescaling with the factor $1/n^2$ the convolution theorem allows to parametrize the discrete convolution \eqref{aa_eq:discconvn} restricted to inputs $u \in \C^{n \times n}$ by $\hat{\kappa} \in \C^{n\times n}$ that correspond to $\mathcal{F}\kappa$. 
This derivation however, requires a rule how to perform convolutions of inputs $u \in \C^{n \times n}$ with ``smaller'' kernels $\hat{\kappa} \in \C^{m\times m}$, $m\leq n$. We recall that in conventional implementations, this is done as in \eqref{aa_eq:discconvn}, i.e., by zero-padding of the ``spatial'' parameters $\kappa$. The authors of \cite{li2021fourier_AA} also propose to handle dimension mismatches by zero-padding, but this time in the domain of Fourier coefficients. This means to define
\begin{equation*}
    \hat{\kappa}_k = 0\qquad \text{for all } k\in \Z^2 \setminus [m]^2
\end{equation*}
for $\hat{\kappa} \in \C^{m\times m}$. The resulting operation is called the spectral convolutional layer. We denote it by $\operatorname{\hat{\ast}}$, i.e., for $u \in \C^{n \times n}$ and $\hat{\kappa} \in \C^{m\times m}$ we define
\begin{equation}
\label{aa_eq:spectralConvolution}
    \hat{\kappa} \operatorname{\hat{\ast}} u:= \mathcal{F}^{-1}\left(\left(\hat{\kappa}_k\,(\mathcal{F}u)_k\right)_{k \in [n]^2}\right).
\end{equation}
We note that the above definition also works for kernels that are larger than the input as the excess entries of $\hat{\kappa}$ will be ignored.

A direct advantage of the spectral convolutional layer is that it indeed approximates the continuous convolution \eqref{aa_eq:contconv} for a fixed kernel function, even for varying input resolutions. For trigonometric polynomials, it even follows from Lemma \ref{aa_lem:exactdft} and Theorem \ref{aa_thm:discconvthm} that the spectral convolutional layer is equivariant with respect to trigonometric interpolation. 
\begin{corollary}[cf. {\cite[Corollary 1]{Kabri2023FNO_AA}}]\label{aa_cor:exacttrigoconv}
    Consider trigonometric polynomials $\kappa: \T^2 \rightarrow \C$ and $u: \T^2\rightarrow \C$ that are determined by \eqref{aa_eq:trigonoPoly} with coefficients $\hat{\kappa}, \hat{u} \in \C^{m\times m}$. Then it holds for any $n \geq m$ that
    \begin{equation*}
        (\kappa * u)^n = \mathcal{F}^{-1}\left(\hat{\kappa}\operatorname{\hat{\ast}}u^n\right),
    \end{equation*}
    where $(\kappa * u)^n = ((\kappa * u)(j/n))_{j_\in[n]^2}$ and $u^n = (u(j/n))_{j_\in[n]^2}$ are discretizations of the functions $\kappa * u$ and $u$ using $n^2$ equidistant sample points.
\end{corollary}
The generalization to multiple input and output channels is done exactly as for the discrete convolution presented in Section \ref{aa_sec:convblocks}. For an activation function $\sigma: \R \rightarrow \R$, the spectral convolutional block parametrized by $\hat{\kappa} \in \C^{c_{\text{out}} \times c_{\text{in}}\times m \times m}$ is then defined for $u \in \C^{n \times n}$ by
\begin{equation}\label{aa_eq:fnoconvlayer}
    S_{\hat{\kappa}}(u) = \sigma\left(\hat{\kappa} \operatorname{\hat{\ast}} u\right).
\end{equation}
At this point we would like to stress that although we can show interpolation equivariance for the spectral convolutional layer, this result does in general not transfer to the spectral convolutional block.
For example, it is pointed out in \cite{bartolucci2023representation_AA,Fanaskov2023specneurop_AA} that applying the popular $\operatorname{ReLU}$ activation function,
    $\operatorname{ReLU}(t) = \max\{0,t\}$,
pointwise in the sense of Nemytskii-operators (cf. \cite[Chapter 4.3]{Troeltzsch2010_AA}) 
does not keep the band-limit of a function, or, to relate it to our derivations, does not map trigonometric polynomials to trigonometric polynomials. This means that in general, even if we use two fine enough but different discretizations of a trigonometric polynomial, we can not expect the outputs of the spectral convolutional layer to be related by trigonometric interpolation anymore. This issue is addressed by respresentation equivalent neural operators (ReNOs) proposed in \cite{bartolucci2023representation_AA}. Building on this concept, the work \cite{raonic20cno_AA} proposes a modified activation function that first upsamples the discrete data to a high resolution before applying the $\operatorname{ReLU}$. 

\subsubsection{Derivation of a spectral U-Net architecture}
\label{aa_subsec:spectralUNet}
Using the spectral convolutional layer, we now derive a spectral U-Net architecture. 
To do so, we first recall that in a classical U-Net architecture the image that is passed through the U-Net is downsampled in the contraction path and then upsampled in the expansion path. We now want to inspect this behavior from the perspective of spectral convolutional layers.

\begin{corollary}[cf. {\cite[Lemma 3]{Kabri2023FNO_AA}}]\label{aa_cor:conversion}
    Let $m, n \in \N$, $m \leq n$ be fixed.
    For any $\kappa \in \C^{m \times m}$ and $u \in \C^{n \times n}$ it holds that
    \begin{equation*}
        \kappa * u = n^2\,(\mathcal{F}{\kappa}^n)  \operatorname{\hat{\ast}} u,
    \end{equation*}
    with $\kappa^n_j := \kappa_j$ if $j \in [m]^2$ and $\kappa^n_j := 0$ if $j \in [n]^2\setminus [m]^2$.
    Vice versa, for any $\hat{\kappa}\in \C^{m \times m}$ and $u \in \C^{n \times n}$ it holds that
    \begin{equation*}
        \hat{\kappa} \operatorname{\hat{\ast}} u = \frac{1}{n^2}\,(\mathcal{F}^{-1}{\hat{\kappa}}^n) * u,
    \end{equation*}
     with $ \hat{\kappa}_k^n:= \hat{\kappa}_k$ if $k \in [m]^2$ and  $\hat{\kappa}_k^n:= \hat{\kappa}_k$ if $k \in [n]^2\setminus [m]^2$.
\end{corollary}
\begin{figure}
    \centering
        \includegraphics[height=0.3\textwidth]{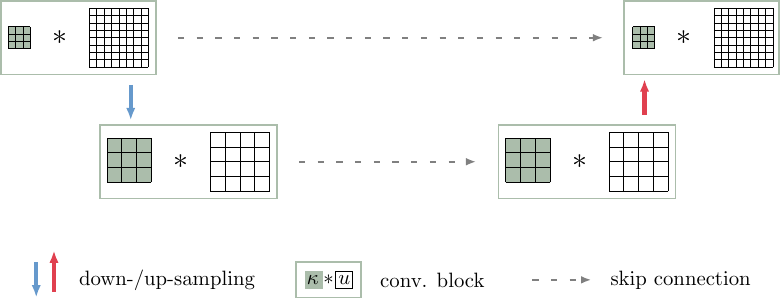}
        \caption{A more detailed visualization of the convolution in classical U-Net architectures. Typically, the size of the input changes by down- and upsampling, while the size of the convolutional kernels is fixed.}
        \label{fig:classicalunetdetail}
\end{figure}
Up to a scaling, the parameter transform from spatial parameters to spectral parameters is performed by first zero-padding the spatial kernel to the size of the input, followed by the application of the discrete Fourier transform, which is essentially trigonometric interpolation. The resulting set of spectral parameters has then the same size as the input. In other words: The kernel is resized to match the size of the input by trigonometric interpolation \eqref{aa_eq:trigointerp}. We now motivate our spectral U-Net architecture as follows:
\begin{enumerate}
    \item We start with a classical U-Net and assume that it is tied to input data of a fixed size.
    \item Using Corollary \ref{aa_cor:conversion}, we transform each convolutional block of the U-Net into a spectral convolutional block.
\end{enumerate}
\begin{figure}[!t]
        \centering
        \includegraphics[height=0.3\textwidth]{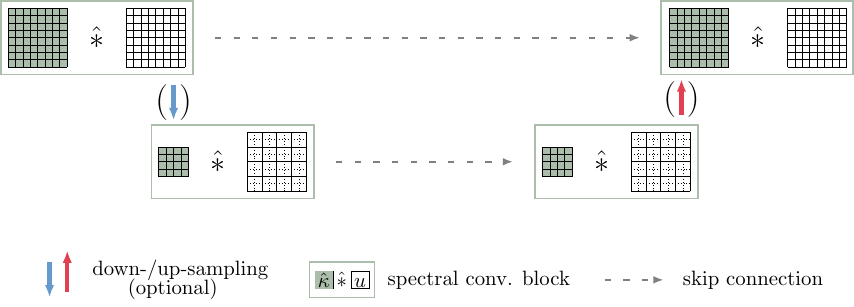}
        \caption{A more detailed visualization of the spectral convolution in the considered spectral U-Net architecture. While down- and upsampling is optional, the size of the kernels shrinks in the contracting path and grows in the expanding path. The overall structure is analogue to classical U-Net architectures.}
        \label{fig:spectralunetdetail}
\end{figure}
Passing through a classical U-Net, the size of the input first shrinks and then grows. This corresponds to a decreasing and then increasing number of spectral parameters. This again means that the convolutional kernels get coarser in the contraction path and then finer in the expansion path (cf. Figure \ref{fig:classicalunetdetail}). Since the spectral convolution is not tied to a fixed input size, we omit the resizing operation. The resulting architecture does not have a U-shape in the original sense anymore, but it has a U-shape with respect to the number of its parameters (cf. Figure \ref{fig:spectralunetdetail}). Since the number of required parameters is significantly larger than usually used in the classical case, we refrain from increasing the number of channels in the contraction path.
A drawback of spectral convolutional layers is that they cannot approximate the identity very well, since for any $u \in \C^{n\times n}$ it holds that 
\begin{equation*}
    u = ((1)_{k \in [n]^2}\operatorname{\hat{\ast}} u),
\end{equation*}
which cannot be obtained by zero-padding any fixed set of parameters $\hat{\kappa} \in \C^{m\times m}$ with $m < n$. For this reason, we supply our architecture with residual connections in each layer.
We summarize our spectral architecture in the following definition.
\begin{definition}[Spectral U-Net]\label{aa_def:spectralunet}
   Given a number of input and output channels $c_{\text{in}} = c_{\text{out}}$, a number of hidden channels $c$ and a depth $d$. Further let $0 < m_0 \leq m_1 \leq \hdots \leq m_{d-1} \leq m_d \leq m_{\text{in}}$ be the kernel sizes.
    The blocks $B_{\text{in}},B_{-d},\dots,B_0,\dots,B_{d},B_{\text{out}}$ of the spectral U-Net consists of the following components:
    \begin{itemize}
        \item The \textbf{first block} $B_{\text{in}}$ is parametrized by $\hat{\kappa}_{\text{in}} \in \C^{c \times c_{\text{in}} \times m_{\text{in}} \times m_{\text{in}}}$. For an input $u_{\text{in}} \in \C^{c_{\text{in}} \times n \times n}$ it is defined by $B_{\text{in}}: u_{\text{in}} \mapsto u_{-d}$,
        \begin{equation*}
        u_{-d} := S_{\hat{\kappa}_{\text{in}}}(u_{\text{in}}). 
    \end{equation*}
    \item \textbf{Contracting path:} The block $B_\ell$ with $\ell = -d, \hdots, -1$ is parametrized by $\hat{\kappa}_{\ell} \in \C^{c \times c \times m_{-\ell} \times m_{-\ell}}$. For an input $u_{\ell - 1} \in \C^{c \times n \times n}$ it is defined by $B_\ell: u_{\ell - 1} \mapsto u_{\ell}$,
     \begin{equation*}
        u_{\ell} := S_{\hat{\kappa}_{\ell}}(u_{\ell -1}) + u_{\ell -1}.
    \end{equation*}
    \item The \textbf{central block} $B_0$ is parametrized by $\hat{\kappa}_0 \in  \C^{c \times c \times m_{0} \times m_{0}}$. For an input $u_{-1} \in \C^{c \times n \times n}$ it is defined by $B_0: u_{-1} \mapsto u_0$,
     \begin{equation*}
        u_0 := S_{\hat{\kappa}_0}(u_{-1}) + u_{-1}.
    \end{equation*}
    \item \textbf{Expanding path:} The block $B_\ell$ with $\ell = 1, \hdots, d$ is parametrized by $\hat{\kappa}_{\ell} \in \C^{c \times c \times m_{\ell} \times m_{\ell}}$. For inputs $u_{-\ell},u_{\ell - 1} \in \C^{c \times n \times n}$ it is defined by $B_{\ell}:(u_{-\ell},u_{\ell - 1}) \mapsto u_{\ell}$,
    \begin{equation*}
        u_{\ell} := S_{\hat{\kappa}_{\ell}}(u_{-\ell} + u_{\ell-1}) + u_{\ell - 1}.
    \end{equation*}
    \item The \textbf{last block} $B_{\text{out}}$ is parametrized by $\kappa_{out} \in \C^{c_{\text{out}}\times c \times 1 \times 1}$. For an input $u_{d} \in \C^{c \times n \times n}$ it is defined by $B_{\text{out}}: u_{d} \mapsto u_{\text{out}}$,
    \begin{equation*}
        u_{\text{out}} = \kappa_{out} * u_{d}.
    \end{equation*}
    \end{itemize}
    Stacking all components together, the spectral U-Net is given by
    \begin{equation*}
        B_{\text{out}} \circ U^+_{d} \circ B_{\text{in}},
    \end{equation*}
    where $U^+_{d}$ is defined recursively for $u \in \C^{c\times n\times n}$ by
    \begin{eqnarray*}
        U^+_\ell(u) & = &\begin{cases}
            B_\ell(U^-_\ell(u), U^+_{\ell -1}(u)) & \text{if } \ell \in \{1,\hdots, d,\}\\
            U^-_0(u) & \text{if } \ell = 0,
        \end{cases}\\
        \\
        U^-_\ell(u) &= & (B_{-\ell} \circ B_{-\ell + 1} \circ \hdots \circ B_{-d})(u), \; \text{for }\ell \in {0,\hdots,d}.
    \end{eqnarray*}
\end{definition}
As mentioned, the above architecture does not include resizing operations. Depending on the size of the input, this can cause a high memory demand. In our numerical experiments in Section \ref{aa_sec:numexp}, we therefore also consider a spectral U-Net architecture with resizing operations. More precisely, in the contraction path and the central block, any input is downsampled via trigonometric interpolation to the size of the kernel $\hat{\kappa}_{\ell}$ before passing through the block $B_\ell$ for $\ell = -d, \hdots, 0$. Inputs that are already smaller keep their size. In the expanding path we then perform the corresponding upsampling via trigonometric interpolation. Although it follows from Corollary \ref{aa_cor:exacttrigoconv} the resizing does not cause any information loss in the spectral convolutional layers, a different behavior of both architectures is possible after the first application of the activation function.

We note that with or without resizing, the spectral U-Net architecture typically needs much larger kernels in order to catch high frequency details than a classical U-Net with the same depth. In the classical version, each convolutional layer is usually parametrized by kernels of the same height and width, e.g. $\kappa_\ell \in \C^{c_{\text{out},\ell}\times c_{\text{in},\ell} \times 3 \times 3}$ for each layer. This additionally means that the size of the kernels first grows with respect to the size of the input and then shrinks again. The kernels parametrized in the spectral convolutional layer however are not restricted to a local field of view. Therefore, although the parametrization used in the spectral convolutional layer is consistent with the integral operator intepretation and comes with the advantage of fast computations due to the fast Fourier transform, it is probably not optimal. Another option we want to explore in the next section, is to use a locally restricted parametrization, but to keep approximating an integral operator by interpolating the kernel to match the size of the input.
\subsection{Interpolating local kernels}\label{aa_subsec:interpolating}
As discussed in Corollary \ref{aa_cor:conversion}, the interpretation of spectral convolutions is to resize the kernel to match the size of the input by trigonometric interpolation. Yet, spectral convolutional layers typically need significantly larger parametrizations than usual fixed-size (spatial) kernels to capture some high-frequency information, leading to a significant increase in the number of learnable parameters and memory requirements. As an alternative and motivated by \cite{liuschiaffini2024localno_AA} and \cite{continuous_conv_AA}, we study locally restricted kernels modeled as piecewise bi-linear functions parameterized by $\tilde{k}^m \in \mathbb{R}^{m\times m}$ (for choices of $m$ such as $m=3$ or $m=5$). In other words, we interpret discrete values of a learnable kernel $\tilde{k}^m$ as evaluations of an underlying continuous function on a fixed, (resolution-independent) regular grid $x^l \in [0,1]$, $l \in [m]$, $x^l < x^{l+1}$, and approximate the underlying continuous function via bi-linear interpolation, i.e., we set
\begin{align}
    k(x,y) =& \left(1-\frac{x-x^i}{x^{i+1}-x^i}\right)\left(1-\frac{y-x^j}{x^{j+1}-x^j}\right) \tilde{k}^m_{i,j} \nonumber\\
    &+ \left(\frac{x-x^i}{x^{i+1}-x^i}\right)\left(1-\frac{y-x^j}{x^{j+1}-x^j}\right) \tilde{k}^m_{i+1,j} \nonumber\\
    &+ \left(1-\frac{x-x^i}{x^{i+1}-x^i}\right)\left(\frac{y-x^j}{x^{j+1}-x^j} \right)\tilde{k}^m_{i,j+1} \nonumber\\ 
    &+ \left(\frac{x-x^i}{x^{i+1}-x^i}\right) \left(\frac{y-x^j}{x^{j+1}-x^j}\right) \tilde{k}^m_{i+1,j+1}
    \label{aa_eq:bilinInterp}
\end{align}
for $x ^i\leq x \leq x^{i+1} \text{ and } x ^j\leq y \leq x^{j+1} $, and set $k(x,y) = 0$ if $x\leq x^1$ or $y\leq x^1$ or $x\geq x^m$ or $y\geq x^m$. Consequently, we approximate the continuous convolution \eqref{aa_eq:contconv} for a given discretization $u^n \in \mathbb{R}^{n \times n}$ via $1/n^2$ times the discrete convolution between $u^n$ and $k^n$ (see \eqref{aa_eq:discreteIntegralApproximation}), where we use the above bi-linear interpolation to create  $k^n$ from $\tilde{k}^m$. Because the above bilinear interpolation is zero if $x\leq x^1$, $y\leq x^1$, $x\geq x^m$ or $y\geq x^m$, the kernel $k(x,y)$ has a local support, and most values of $k^n$ do not have to be computed for evaluating the discrete convolution. We therefore largely retain the computational efficiency of the usual fixed-size discrete convolution. 
This approach permits an exact approximation of a certain set of kernels, namely piecewise, bilinear polynomials:
\begin{equation}
k(x,y) = ax + by + cxy + d.
\label{aa_eq:bilinPoly}
\end{equation}
Even though different sets might be reasonable from a practical point of view, the most natural choice for a set of functions preserved by bilinear interpolation are (piecewise) bilinear polynomials. That this holds can be seen by plugging in \eqref{aa_eq:bilinPoly} into \eqref{aa_eq:bilinInterp}, where it is easily visible that the four parameters are always uniquely determined by the four neighboring vertices of the sampling grid.
An important counterexample in the context of this work is the identity function, which we define here for the discretized case as follows:
$$
k^n(x,y) = \begin{cases}n^2&\text{ if } x=y=0,\\0&\text{ else.}\end{cases}
$$
The term $n^2$ is necessary as we intend to approximate \eqref{aa_eq:discreteIntegralApproximation} as close as possible. 
By bilinearly interpolating such an identity-kernel of size $m$ to size $n>m$, we observe two effects: First, the interpolation of a singular Dirac peak of height $m^2$ results in a convex combination of the peak value and a neighboring zero value, thus either keeping the value or decreasing it, contrary to the required increase to a value of $n^2$. Furthermore the just described effect also spreads the central, ``singular'' peak onto multiple neighbors, losing one of the central properties of a Dirac measure. Intuitively speaking, bilinearly interpolating a Dirac measure results in a wider and flatter function.

Similarly to Dirac measures, differential operators cannot be preserved very well by interpolation. In the following section, we therefore take a closer look at possibilities to express differential operators by convolutional layers.

\section{U-Nets with differential operators}
\label{aa_sec:differential}
In contrast to the interpretation as integral operators, we now turn to the interpretation of convolutional layers in discrete U-Nets as discretized differential operators. Differential operators are a valuable ingredient in image processing for various reasons: On the one hand, a prominent tool from classical feature extraction is edge detection, which can be related to identifying areas with a large image gradient (cf. \cite[Chapter 4]{Nixon2020_AA}). On the other hand, differential operators naturally appear in techniques from PDE-based image processing (cf. \cite{Aubert2006_AA}). In this section, we discuss possibilities to connect discrete U-Net architectures to these two paradigms.

As pointed out in \cite{liuschiaffini2024localno_AA}, the convolution \eqref{aa_eq:discconvn} can be generalized to variable input resolutions by interpreting it as the finite differences approximation of a differential operator (cf. \cite{Strikwerda2004_AA}) by scaling with the inverse spatial step-size. To motivate this, we restate Proposition 3.1 from \cite{liuschiaffini2024localno_AA} for our equispaced grid setting. In general, the proposition states that the scaled convolution converges pointwise to a directional derivative. In our specific case, we consider the limit of the scaled convolution with a fixed kernel and discretizations with increasing resolution, i.e.,
\begin{equation*}
    u^{l\cdot n} = \left(u\left(\frac{i}{l\cdot n}\right)\right)_{i \in [l\cdot n]^2} \qquad \text{with } l, n  \in \N.
\end{equation*}
Here, $n$ denotes an arbitrary but fixed base resolution and $l$ denotes the refinement. Scaling the convolution with the inverse spatial step-size corresponds to a multiplication with $l \cdot n$. We state the pointwise convergence in points $i/n = (l \cdot i)/(l\cdot n)$ for any $i \in [n]^2$. 
We further note that throughout this section, we perform the discrete convolution via zero-padding of undefined values of the input $u$. In the case of continuous functions, this can be interpreted as zero boundary conditions. To match the zero-centered indexing, we again use the domain $\Omega = (-0.5,0.5)^2$.
\begin{lemma}[{\cite[Prop. 3.1]{liuschiaffini2024localno_AA}}]\label{aa_lem:diffconvergence}
Consider $u \in C_0^1(\Omega, \R)$ and its discretizations $u^n = (u(i/n)))_{i \in [n]^2}$ using $n^2
$ equidistant sampling points for any $n \in  \N$.
Further let $\kappa \in \R^{m \times m}$, $m\in \N$, be a fixed convolution kernel that fulfills
\begin{equation}\label{aa_eq:zeromeankernel}
    \sum_{i \in [m]^2} \kappa_i = 0.
\end{equation}
Then it holds for any $i \in [n]^2$ that 
\begin{equation*}
    \lim_{l \rightarrow \infty}  (l\cdot n)\,(\kappa * u^{l\cdot n})_{l\cdot i}= \langle b, \nabla u\rangle  (i/n),
\end{equation*}
for any $i \in [n]^2$ and
\begin{equation}\label{aa_eq:bdirecderiv}
    b = -\sum_{i \in [m]^2} i \, \kappa_i.
\end{equation}
\end{lemma}
\begin{proof}
    The proof follows from the first order Taylor expansion of $u$ and can be found in \cite[Appendix A.1]{liuschiaffini2024localno_AA}. For the derivation of the explicit form of $b$ we note that the result in \cite{liuschiaffini2024localno_AA} is stated for the cross-correlation (denoted by $\star$) instead of the convolution (denoted by $*$), which causes opposite signs of $b$.
\end{proof}
A kernel that fulfills \eqref{aa_eq:zeromeankernel} is called a differential kernel.
In the implementation provided by the authors of \cite{liuschiaffini2024localno_AA}, it is proposed to transform any kernel $\kappa$ into a differential kernel by subtracting the sum over the entries of $\kappa$ from the element $\kappa_{0,0}$.\footnote{We would like to thank Miguel Liu-Schiaffini for sharing this insight. The code is available at \url{https://github.com/neuraloperator/neuraloperator/blob/main/neuralop/layers/differential_conv.py}} This approach has the additional advantage that it does not alter the vector $b$, since $\kappa_{0,0}$ is always multiplied by $(0,0)$ in \eqref{aa_eq:bdirecderiv}.

A consequence of condition \eqref{aa_eq:zeromeankernel} one should keep in mind is that convolutions with differential kernels cannot approximate the identity anymore. In the architecture proposed in \cite{liuschiaffini2024localno_AA} this issue is resolved by adding residual connections. To stay as close to the classical convolutional layer as possible, we instead adapt the definition of our differential convolutional layer: For a kernel $\kappa \in \R^{m\times m}$ and inputs $u \in \R^{n \times n}$ we define the differential convolutional layer $\operatorname{Diff}^n$ as
\begin{equation*}
    \operatorname{Diff}^n(\kappa, u) := n \,\kappa * u + (1-n)\,\tilde{\kappa}\,u, \qquad \text{with } \tilde{\kappa}= \sum_{i \in [m]^2} \kappa_i.
\end{equation*}It follows directly from Lemma \ref{aa_lem:diffconvergence} that the differential convolutional layer still approximates a directional derivative but is also capable to reproduce the identity.\begin{corollary}
    Consider $u \in C_0^1(\Omega, \R)$ and its discretizations $u^n = (u(i/n))_{i \in [n]^2}$ using $n^2$ equidistant sampling points for any $n \in  \N$.
Further let $\kappa \in \R^{m \times m}$, $m\in \N$, be any fixed convolution kernel. Then it holds that
\begin{equation*}
\lim_{l \rightarrow \infty} \operatorname{Diff}^{l\cdot n}(\kappa,u^{l\cdot n})_{l\cdot i} = \left(\tilde{\kappa}\, u + \langle b, \nabla u \rangle\right) (i/n),
\end{equation*}
for any $i \in [n]^2$ and $b$ as in \eqref{aa_eq:bdirecderiv}.
\end{corollary}
The following example visualizes how we can use the layer $\operatorname{Diff}^n$ to approximate any directional derivative and multiple of the identity.
\begin{example}\label{aa_ex:dirderivative}
    Consider the kernel
    \begin{equation*}
        \kappa = \begin{pmatrix}
            0 & 0 & 0 \\
            0&\kappa_1+\kappa_2+c & \hphantom{..}-\kappa_1\\
            0 & -\kappa_2& 0
        \end{pmatrix},
    \end{equation*}
    with $\kappa_1, \kappa_2, c \in \R$. We compute $\tilde{\kappa} = c$ and thus for any $u \in C_0^1(\Omega, \R)$, $u^n$ as above, it follows for any $i \in [n]^2$ that
    \begin{equation*}
        \lim_{l \rightarrow \infty} \operatorname{Diff}^{l \cdot n}(\kappa, u^{l\cdot n})_{l\cdot i} = \left(c\,u + \kappa_2 \,\partial_1 u + \kappa_1\, \partial_2 u\right) (i/n).
    \end{equation*}
    The reason that $\kappa_2$ controls the horizontal derivative $\partial_1 u$ and $\kappa_1$ controls the vertical derivative $\partial_2 u$ lies in the discretization of $u$ as a matrix, which rotates the coordinate system by $-\pi /2$. In practice one would rather consider discretizations $(u((i_2, -i_1)/n))_{i\in \N}$ which does not cause the described switch of axes.
\end{example}
In the remainder of this section we want to investigate how the idea of differential convolutions can be incorporated into U-shaped neural network architectures for image processing. Here, we distinguish between the role as a feature extractor on the one hand and its integration into the approximation of a partial differential equation on the other hand. 
\subsection{Feature extraction perspective}
We first want to subsume the discussed differential convolutions in the context of feature extraction. This means that we understand the purpose of the differential convolutional layer as returning a directional derivative as a feature that characterizes a given function. 
It is important to note that in the case of image processing, the assumption of differentiability is typically only fulfilled piecewise. Instead, images are assumed to lie in the space of functions of bounded variation, $BV$. In particular, the space $BV(\Omega)$ consists of all function $u \in L^1(\Omega)$, for which there exists a finite, vector-valued Radon measure $Du$ in $\Omega$ such that
\begin{equation*}
    \int_{\Omega} u\,\operatorname{div}\phi \,dx = -\int_{\Omega} \langle \phi, dDu\rangle \quad \text{ for any test function }\phi \in C^1_0(\Omega, \R^2).
\end{equation*}
The measure $Du$ is called the distributional derivative of $u$. We refer to \cite[Chapter 3]{Ambrosio2000_AA} for a profound introduction. If a function $u$ is weakly differentiable, $Du$ is absolutely continuous with respect to the Lebesgue measure and its density is given by the weak derivative of $u$. In this case, the application of a differential convolutional layer should not cause any problems. The more interesting case however, is to consider input images with sharp edges, or, in other words, piecewise differentiable functions that contain jumps.  For such inputs, a clarification of the purpose of the feature extraction layer is required.

If the aim is to approximate $Du$ in regions where $u$ is differentiable, the differential convolutional layer will return the desired output in these regions. At jumps however, the scaling with the inverse spatial step-size will cause divergence of the outputs as the number of sampling points tends to infinity. This issue could be approached for example by adding a truncation layer that cuts off high gradients. 

In contrast to this, an important part of classical feature extraction is edge detection (cf. \cite[Chapter 4]{Nixon2020_AA}), which is concerned with the jumps of the input function. Therefore, it could be more suitable to have a layer that extracts the location and the height of jumps, which would then correspond to the classical convolutional layer again. Here, the output would consist of the sparse set of edges, while in differentiable regions it would vanish with an increasing number of sampling points. While the gradient information contained in the differential part might be negligible (for example in the case of piece-wise constant input functions), only extracting the edges of an image might not be sufficient to solve a given problem. In many cases the architecture would still need a mechanism to fill in, or, ``color in'' the extracted edges, for example by a learned upsampling method. Additionally, the function of deeper layers in the contracting path is not clear from these derivations. 

We thus leave this direction for future work and summarize our considerations in Figure \ref{aa_fig:jump_approximation} and the following example.
\begin{figure}
    \centering
    \includegraphics[width=0.8\linewidth, trim = {4cm 0cm 5cm 0cm}, clip]{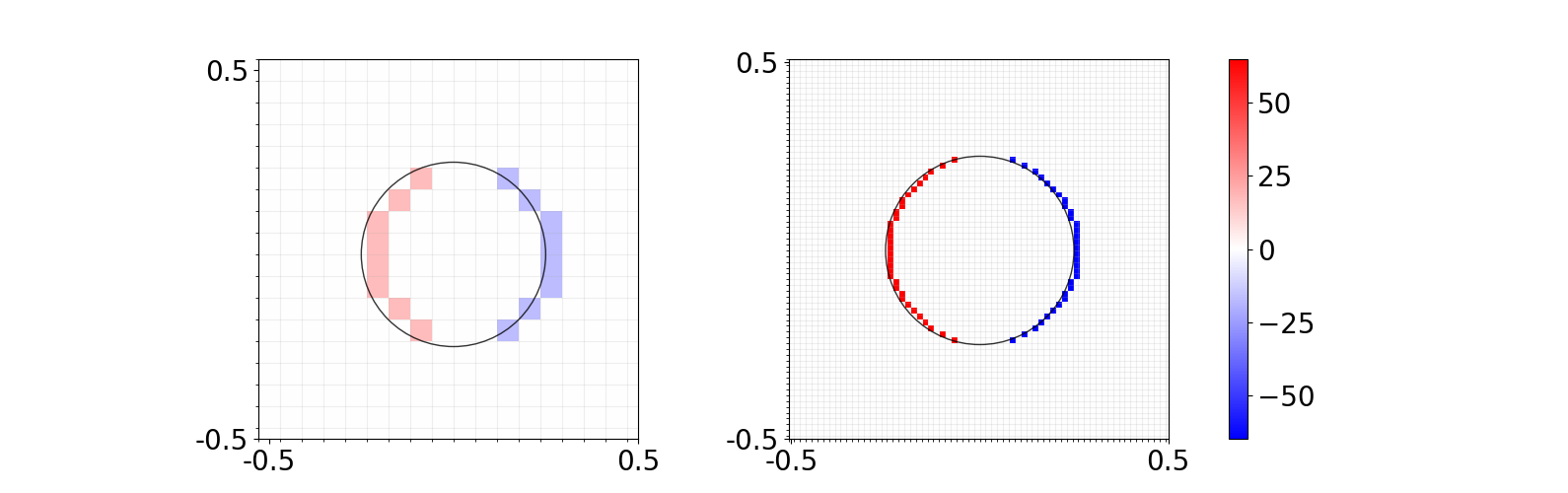}
    \caption{Outputs of the differential layer for two different input resolutions. Left: $\operatorname{Diff}^{17}(\kappa, u^{17})$, right: $\operatorname{Diff}^{65}(\kappa, u^{65})$, with $u$ and $\kappa$ as described in Example \ref{aa_ex:jumpexample}. The black circle corresponds to the boundary $\partial B_{1/4}(0)$. The maximum magnitude of the outputs grows with the number of samples. The output in the boundary regions does not contain the full information about the corresponding normal vector in $\partial B_{1/4}(0)$.}
    \label{aa_fig:jump_approximation}
\end{figure}
\begin{example}\label{aa_ex:jumpexample}
    On the zero-centered domain $\Omega = (-0.5, 0.5)^2$, with the corresponding Borel $\sigma$-algebra $\mathcal{B}$, we consider the function $u \in BV(\Omega)$, defined for a constant $c > 0$ as 
    \begin{equation*}
        u(x) = \begin{cases}
         c & \text{if } x\in \overline{B_{1/4}\left(0\right)},\\
            0 & \text{otherwise.}
        \end{cases}
    \end{equation*}
    In the interior of $B_{1/4}(0)$ and $\Omega \setminus B_{1/4}(0)$ the function $u$ is differentiable (with $\nabla u = 0$) and the jump set is given by $\partial B_{1/4}(0)$. It follows from Gauß's Theorem (see e.g, \cite[A8.8]{Alt2016_AA}) that the distributional derivative $D u$ can be computed for sets $A \in \mathcal{B}$ as
    \begin{align*}
        Du(A) = \int_{A \cap B_{1/4}(0)} \nabla u \, dx - &c \int_{A \cap \partial B_{1/4}(0)} x/\|x\| \, d\mathcal{H}^{1}(x) \\= \,-&c \int_{A \cap \partial B_{1/4}(0)}  x/\|x\|  \, d\mathcal{H}^{1}(x),
    \end{align*}
    where $x/\|x\|$ corresponds to the normal vector at $x \in \partial B_{1/4}(0)$ and $\mathcal{H}^1$ denotes the $1$-dimensional Hausdorff measure (cf. \cite{Alt2016_AA}).
    For the sake of simplicity, we only consider the first component of the derivative, $Du_1$. For a differentiable function, we know from Lemma \ref{aa_lem:diffconvergence} and Example \ref{aa_ex:dirderivative} that $\partial_1 u $ can be approximated by differential layers $\operatorname{Diff}^{l\cdot n}(\kappa,\cdot)$, $l \rightarrow \infty$, with kernel $\kappa = (0,1, -1)^T$. It directly follows that for any $n \in \N$, $i \in [n]^2$ that
    \begin{equation*}
       \lim_{l\rightarrow \infty} \operatorname{Diff}^{l\cdot n}(\kappa, u^{l \cdot n})_{l \cdot i} = \partial_1u(i/n),
    \end{equation*}
    as long as $i > 0$ or $\|i/n\| \neq 1/4$. However, if $i \leq 0$ and $\|i/n\|= 4$, the differential layer approximates the jump and the above limit is infinity. Moreover, we can show that 
    \begin{equation*}
        \|\operatorname{Diff}^{n}(\kappa, u^{n})\|_{\infty} \longrightarrow \infty,  \text{ as }n \rightarrow \infty,
    \end{equation*}
    since for any $n \in \N$ there exists $i \in [n]^2$ such that $\|i/n\| \geq 1/4$, but $\|(i_1,i_2-1)/n\| < 1/4$ and thus $\operatorname{Diff}^{n}(\kappa, u^{n})_i = n\, c$. The diverging entries could be bounded by applying a thresholding operation, but $c$, the height of the jump, cannot be recovered with this approach. 

    Next, we consider the behavior of the unscaled, classical convolutional layer. Due to the missing scaling, we get for any $n \in \N$, $i \in [n]^2$ 
    \begin{equation*}
       \lim_{l\rightarrow \infty} (\kappa * u^{l \cdot n})_{l \cdot i} = 0,
    \end{equation*}
    as long as $i > 0$ or $\|i/n\| \neq 1/4$. Although in our example, this coincides with the actual gradient $\partial_1(i/n)$, the limit would also be zero in the case of a non-trivial gradient, for example for a function that is not constant in $\overline{B_{1/4}\left(0\right)}$. In contrast at any positive jump, i.e., $i \in [n]^2$ such that $\|i/n\| \geq 1/4$ and $\|(i_1,i_2-1)/n\| < 1/4$, we get $(\kappa * u^n)_i = c$, and analogously, we get $(\kappa * u^n)_i = -c$ at negative jumps. This means, that we can approximate the position and the height of jumps. However, the directional information $x/\|x\|$ contained in $Du$ is reduced to the elementwise sign $\sgn(x)$.
    \end{example}

\subsection{PDE perspective}
Another way to interpret differential convolutional layers is as a component of finite difference schemes to approximate the solution of a partial differential equation (PDE). The essential ingredient here are residual connections. These have been studied in a range of works \cite{chen2018neural_AA, e2017dynamical_AA, haber2017stable_AA, thorpe2023deep_AA} in the context of finite difference schemes to approximate solutions of ordinary differential equations. Further incorporating spatial derivatives, \cite{ruthotto2020deep_AA} proposes to design residual convolutional neural networks based on finite difference schemes for partial differential equation. In this section, we want to investigate how this approach transfers to the setting of variable input resolutions. Similar to the approach described in \cite{ruthotto2020deep_AA}, we now consider kernels $\kappa$ that implement finite differences of order $\alpha \in \N$, for example  
\begin{equation*}
    \kappa = \begin{pmatrix}
        0 & 0 & 0 \\
        0 &2 & -1 \\
        0 &-1 & 0
    \end{pmatrix},\;\text{or,}\; \kappa =\begin{pmatrix}
        0 & 1 & 0 \\
        1 & -4 & 1\\
        0 & 1 & 0
    \end{pmatrix}.  
\end{equation*}
which (after appropriate scaling with the pixel size) approximate first derivates or a Laplacian, respectively.
We rewrite the convolutional layer with a residual connection in iterative form as
\begin{equation}\label{aa_eq:residual_cnn}
    u^n(t_{k+1}) = u^n(t_k) + \operatorname{ReLU}\left(\kappa* u^n(t_k)\right).
\end{equation}
In the following, we want to show that for suitable kernels, this iterative formulation can be seen as a stable finite difference scheme to approximate the solution of a partial differential equation.
To do so, we introduce a time discretization determined by the time step-size $\Delta t$, $t_k = k\Delta t$ for any $k \in \N$ and denote the spatial step-size by $\Delta x$. We recall that for discretizations $u^n$ sampled on a regular $n\times n$-grid, it holds that $\Delta x = 1/n$. 
Further introducing a discretization independent constant $c \in \R$, we consider the finite difference scheme 
\begin{equation}\label{aa_eq:findiffscheme}
    u^n(t_{k+1}) = u^n(t_k) + c\,\Delta t \,\operatorname{ReLU}\left(\frac{\kappa* u^n(t_k)}{{\Delta x}^\alpha}\right),
\end{equation}
where $\alpha \in \N$ determines the order of the derivative.\footnote{We note that to obtain a approximation of a derivative of order $\alpha$, the kernel $\kappa$ has to fulfill conditions similar to \eqref{aa_eq:zeromeankernel}.} Such finite difference schemes are called conditionally stable if under a suitable relation between $\Delta t$ and $\Delta x$ for any $T > 0$, there exists a constant $C^T$ independent of $\Delta t$, $\Delta x$ such that
\begin{equation*}
    \|u^n(t_k)\| \leq C^T \|u^n(t_0)\|\qquad \text{for any } k \text{ such that }t_k \leq T.
\end{equation*}
We refer to \cite[Definition 1.5.1]{Strikwerda2004_AA} for a more general definition. Comparing \eqref{aa_eq:residual_cnn} to \eqref{aa_eq:findiffscheme}, we see that the classical convolution without rescaling of $\kappa$ corresponds to  
$$ c \Delta t/(\Delta x)^{\alpha} = 1.$$
Such a choice is reminiscent of the Courant-Friedrichs-Levy (CFL) condition needed for stable explicit discretizations of a partial differential equation (cf. \cite[Theorem 1.6.1]{Strikwerda2004_AA}). We note that for the above condition to be sufficient to obtain conditional stability, it requires further assumptions on $\kappa$ that control the scaling and the direction of the finite difference approximation and refer to \cite{Strikwerda2004_AA} for further details.

On the other hand, the above explicit Euler scheme is not stable if the CFL condition is violated in the sense that $\Delta t/(\Delta x)^{\alpha} \rightarrow \infty$.\footnote{This can be seen as follows: For an arbitrary 
kernel $\kappa$ choose a non-zero entry $\kappa_j$ and $u^n_{-j}(t_0) = \kappa_j$, zero everywhere else. Then $\|u^n(t_0)\| = \kappa_j$ and $(\kappa * u^n(t_0))_{0} = \kappa_j^2$, thus $\|u^n(t_1)\| \geq ||\kappa_j| - |c| \,\Delta t/(\Delta x)^{\alpha} \kappa_j^2| \rightarrow \infty$ as $\Delta t/(\Delta x)^{\alpha} \rightarrow \infty$.} 
However this divergence of $\Delta t/(\Delta x)^{\alpha}$ does occur if the  convolutional layer is substituted by a differential layer that approximates a derivative of order $\alpha$. More precisely, the iterative scheme \begin{equation*}
    u^n(t_{k+1}) = u^n(t_k) + \,\operatorname{ReLU}\left(\operatorname{Diff}_\alpha^n(\kappa, u^n(t_k))\right),
\end{equation*}
with
$$  \operatorname{Diff}_\alpha^n(\kappa, u^n(t_k)) = \frac{\kappa* u^n(t_k)}{{\Delta x}^\alpha},$$
corresponds to choosing the constant time step-size $\Delta t =  1/c$.

From the above derivations we conclude that the classical convolutional layer seems more suitable to implement the described PDE-perspective.
However, we also see that choosing $\Delta t = c\,\Delta x^\alpha$ implicitly means that approximating the solution at a fixed time $T$ requires more and more time steps (respectively layers) as the resolution increases. In contrast, the classical U-Net architecture consists of a fixed number of layers. Hence, we expect a decreasing  effect on the reconstruction quality of such a network architecture for increasing resolution, which can partly be observed from the computational results below.   
We leave a more detailed study of this direction for future work.

\section{Numerical experiments}\label{aa_sec:numexp}
\subsection{A toy example}
As a starting point for understanding different ways to handle varying discretizations of the same inverse problem, consider deblurring a signal that is a superposition of finitely many dirac deltas, i.e., consider
\begin{align}
    f = Au + n, \qquad u(x) = \sum_{j=1}^N \alpha_j \delta(x-x_j), \qquad A u = k*u
\end{align}
for a Gaussian kernel $k$, and non-negative coefficients $\alpha_j >0$. For any discretization, i.e., sampling of $f$ at discrete positions $i$, leading to $\tilde{f} =(f(x_i))_{i=1,\hdots,n} \in \mathbb{R}^n$, we reconstruct an approximation $\tilde{u} \in \mathbb{R}^n$ of $u$ on the same spatial grid. 

As we expect $\tilde{u}$ to be non-negative and sparse, we use an $\ell^1$ regularization as well as a non-negativity constraint along with a quadratic data fidelity term (assuming measurements of $\tilde{f}$ with Gaussian noise) to compute the classical model-based reconstruction
\begin{align}
\tilde{u} = \arg \min_{u\geq 0} \frac{1}{2}\|\tilde{A}u - \tilde{f}\|^2 + \gamma \|u\|_1, 
\end{align}
for $\tilde{A}u = \tilde{k}*u$ being the convolution with a discrete Gaussian kernel. Using a projected gradient descent for the above minimization leads to iterations
\begin{align}
\label{aa_eq:gradProj}
    \tilde{u}^{k+1} =& \text{proj}_{ \cdot \geq 0 }\left(\tilde{u}^k - \tau \tilde{A}^*(\tilde{A}\tilde{u}^k - \tilde{f}) - \tau \gamma  \right) \\
\label{aa_eq:gradProj_net1}
    =& \text{ReLU}\left( \underbrace{\tilde{g}}_{\text{kernel}} * \ \tilde{u}^k \  + \  \underbrace{(- \tau \gamma)}_{\text{bias}} \  +\   \underbrace{\tau\tilde{u}^0}_{\text{skip con. from input}} \right) \\
\label{aa_eq:gradProj_net2}
     =& \text{ReLU}\left( \underbrace{\tilde{u}^k}_{\text{skip con.}} \   +\   \underbrace{\tilde{h}}_{\text{kernel}} * \ \tilde{u}^k \  + \  \underbrace{(- \tau \gamma)}_{\text{bias}} \  +\  \underbrace{\tau\tilde{u}^0}_{\text{skip con. from input}} \right).
\end{align}
We have interpreted the (provably convergent) gradient projection iterations as two slightly different but equivalent convolutional neural networks: For an input $u^0 = \tilde{A}^*f$ (being an approximate reconstruction obtained from applying the adjoint operator to the measurements), we use skip connections (weighted with $\tau$) from the input to the output of every convolutional layer. We obtain two different interpretations of the networks by either using a kernel $\tilde{g}:= I - \tau (\tilde{k}*\tilde{k})$ that represents $(I - \tau \tilde{A}^*\tilde{A})$,\footnote{Please note that the convolution with a Gaussian kernel is self-adjoint.} or using a kernel $\tilde{h}:= -\tau (k*k)$ and making the skip connection from the previous layer explicit.

While the two formulations, \eqref{aa_eq:gradProj_net1} and \eqref{aa_eq:gradProj_net2}, are of course identical, they start to differ if one considers them to be architectures that ought to handle different resolutions (discretizations) with the same parameters. In this case, we have several options:

\begin{figure}
    \centering
\begin{subfigure}{.43\linewidth}
    \includegraphics[width=\textwidth]{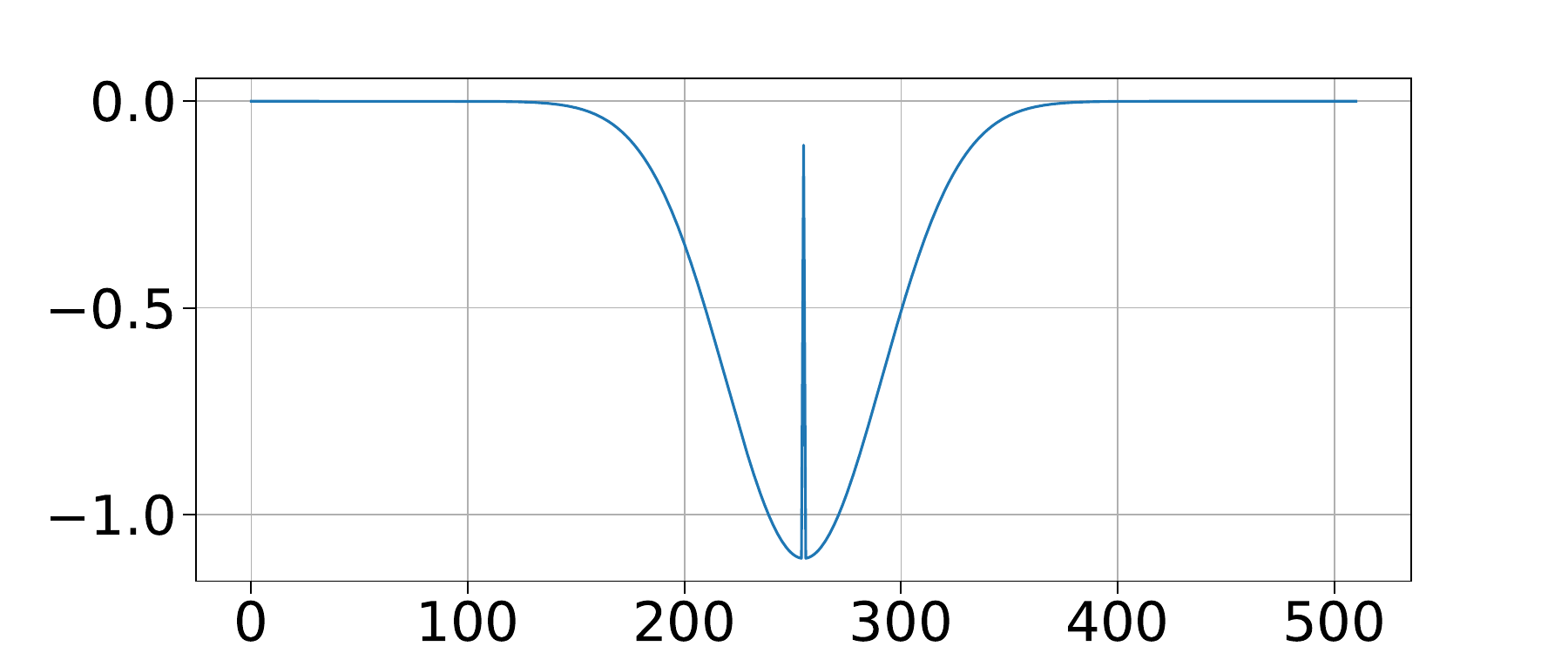} 
    \subcaption{True Kernel.}
\end{subfigure}
\begin{subfigure}{.43\linewidth}
    \includegraphics[width=\linewidth]{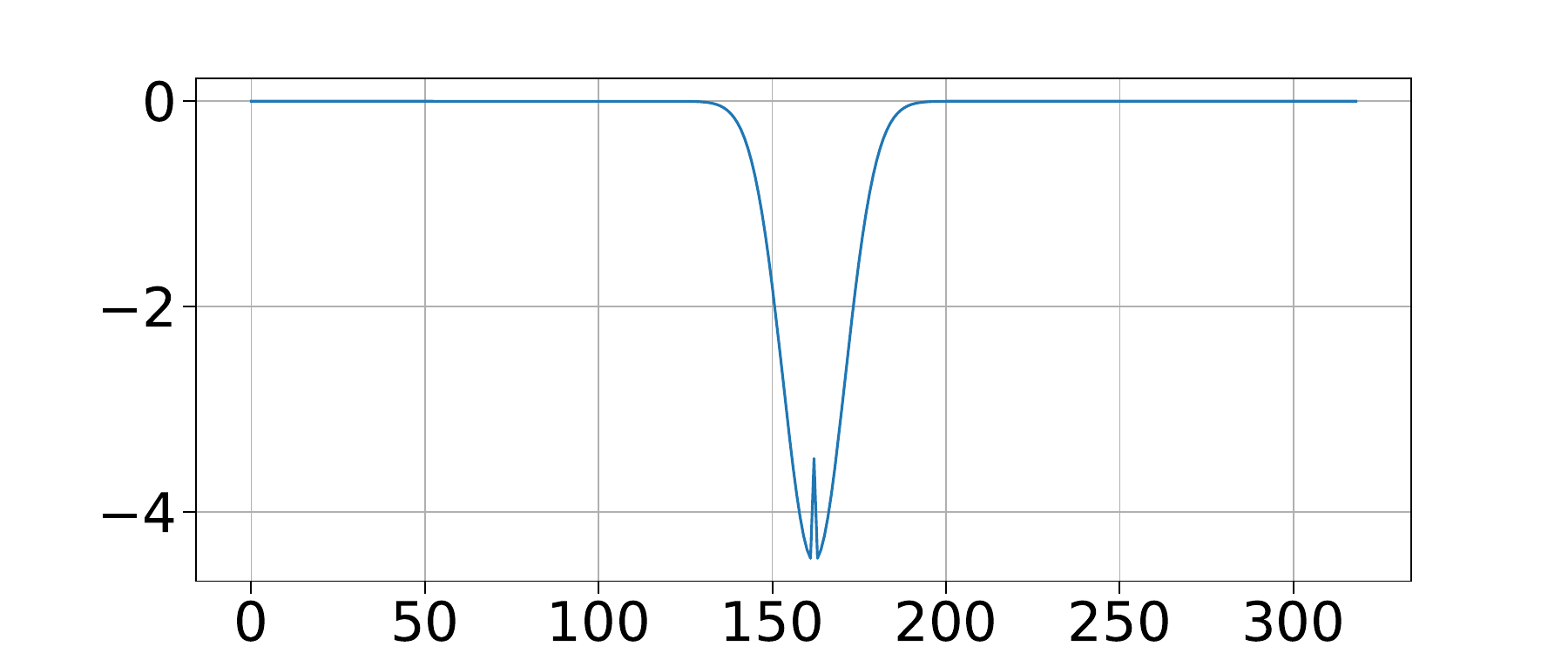}
    \subcaption{Small kernel $\tilde{h}$.}
\end{subfigure}
\begin{subfigure}{0.43\linewidth}
    \includegraphics[width=\linewidth]{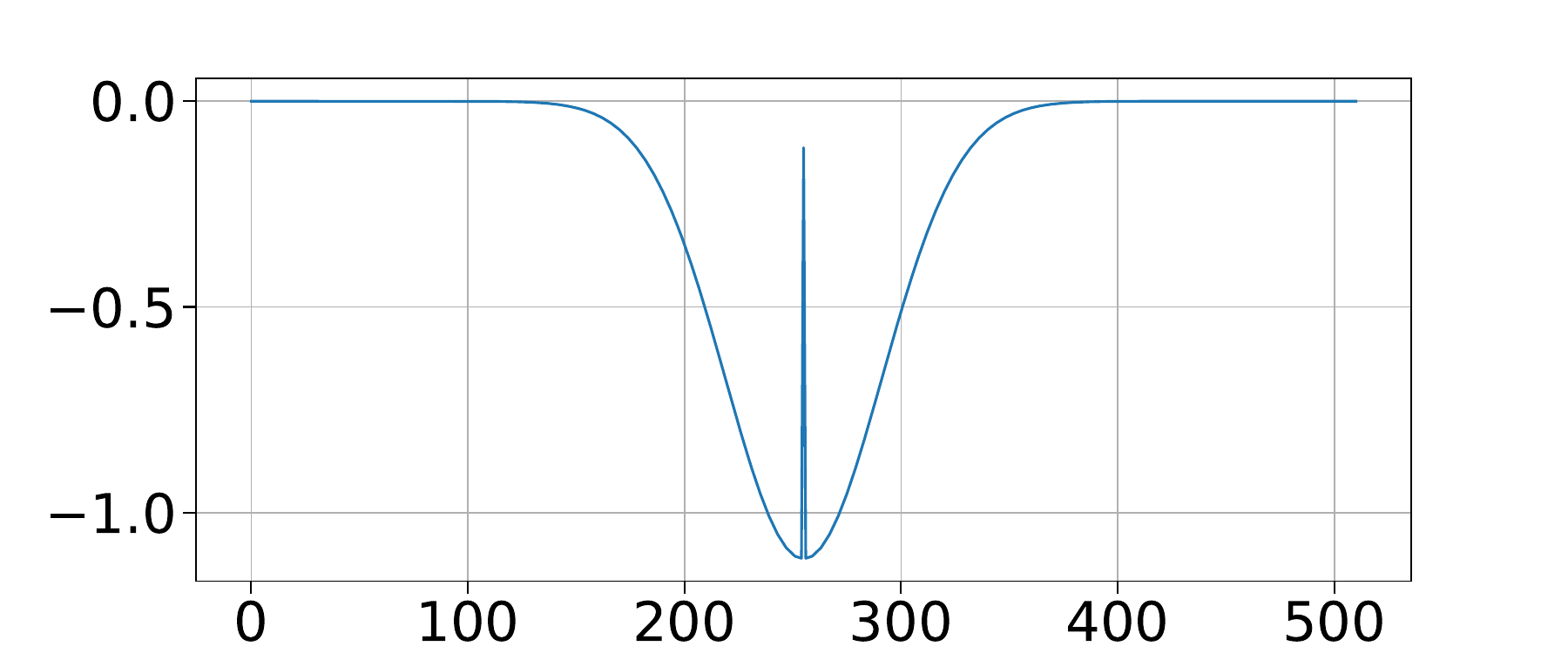}
    \subcaption{Interpolated Kernel $\tilde{h}$.}
\end{subfigure}
\begin{subfigure}{.43\linewidth}
    \includegraphics[width=\linewidth]{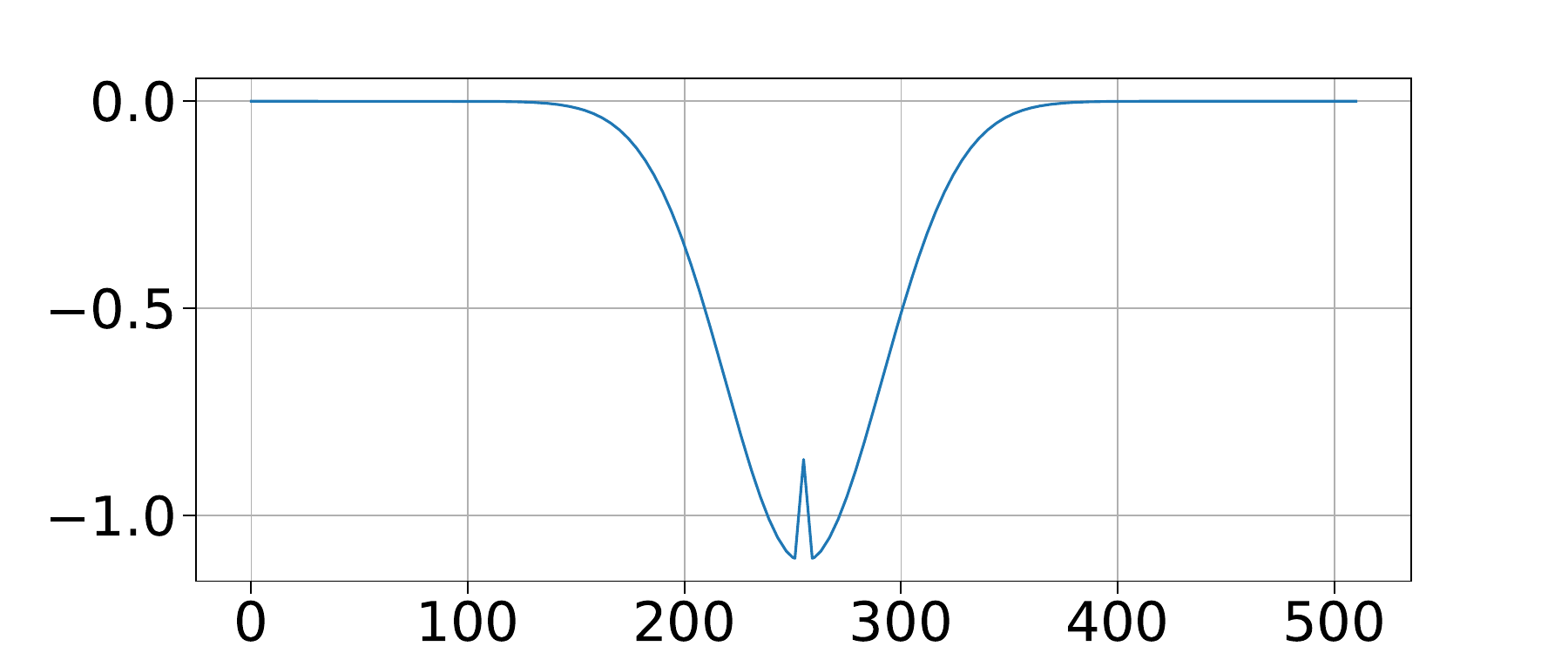}
    \subcaption{Interpolated Kernel $\tilde{g}=I-\tilde{h}$.}
\end{subfigure}
\begin{subfigure}{.43\linewidth}
    \includegraphics[width=\linewidth]{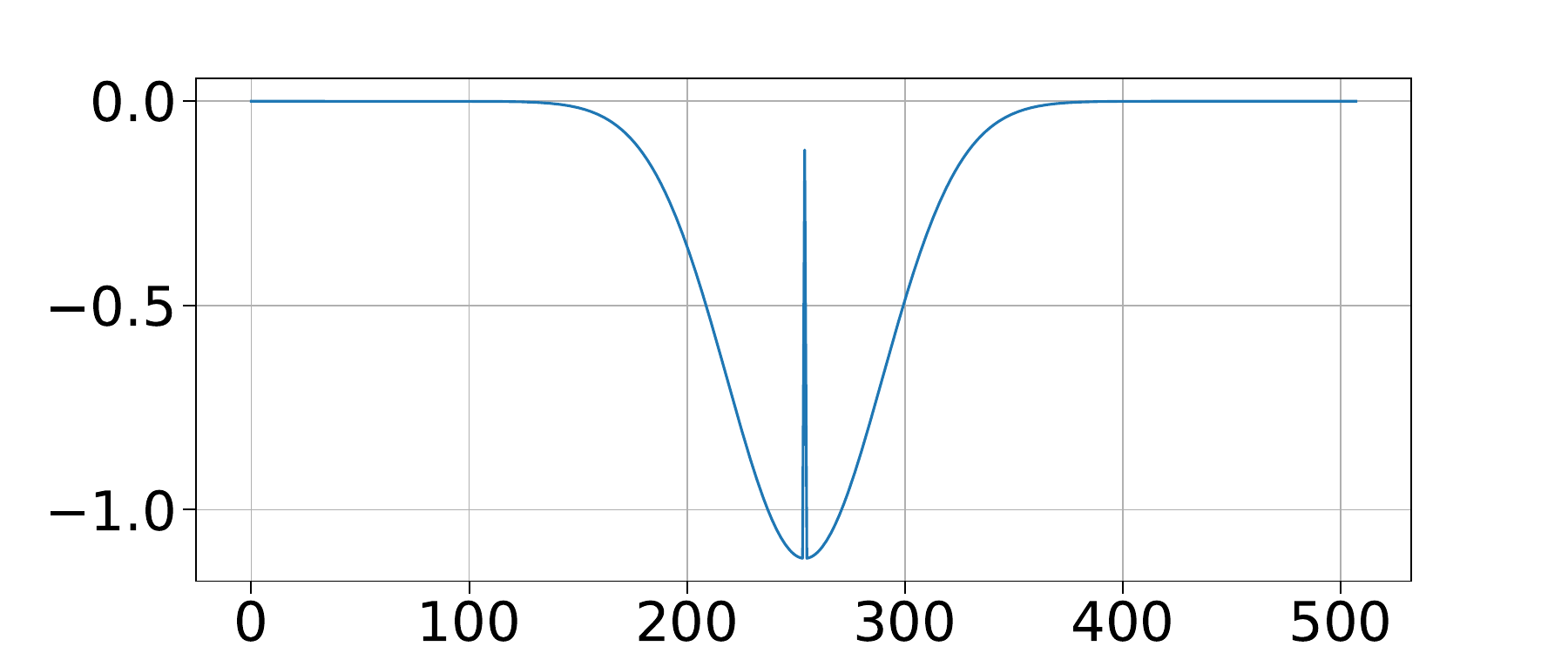}
    \subcaption{Spectral kernel $\tilde{h}$.}
    
\end{subfigure}
\begin{subfigure}{.43\linewidth}
    \includegraphics[width=\linewidth]{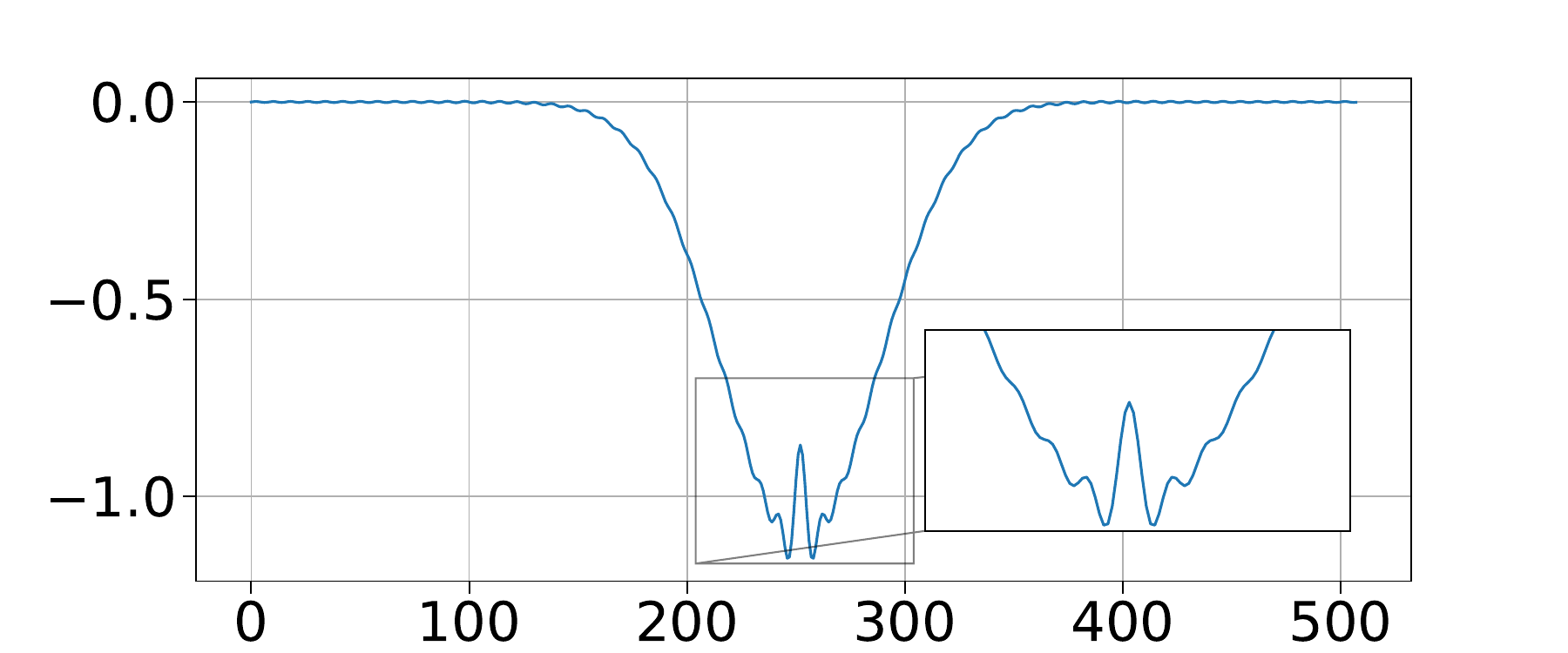}
    \subcaption{Spectral kernel $\tilde{g}=I-\tilde{h}$.}
    
\end{subfigure}
    
    \caption{The different kernels. Interpolating as well as padding the kernel  $\tilde{h}$ in the Fourier space produces a result that is indistinguishable from the true kernel. When interpolating $\tilde{g}=I-\tilde{h}$, the skip connection is no longer accurately preserved.The same issue arises when $\tilde{g}$ is padded in Fourier space, which leads to a sinc-like artifact in the resulting kernel.}
    \label{aa_fig:kernelIllustration}
\end{figure}
\begin{itemize}
    \item Use the same kernel ($\tilde{g}$ or $\tilde{h}$) for different resolutions. In this case the two network architectures remain identical. Yet, they will not generalize across different resolutions at all, since a finer discretization (larger $n$) requires a wider kernel.
    \item Parameterize the kernel with a few parameters (suitable for a small $n$) and use linear interpolation to handle different discretizations. In this case, the two approaches, \eqref{aa_eq:gradProj_net1} and \eqref{aa_eq:gradProj_net2}, start to differ as \eqref{aa_eq:gradProj_net1} includes the skip connection into the interpolation while \eqref{aa_eq:gradProj_net2} does not, see Fig.~\ref{aa_fig:kernelIllustration}. Based on how the direct application of the gradient projection algorithm at the correct resolution would look like, we expect \eqref{aa_eq:gradProj_net2} to work significantly better. 
    \item Write the convolutions in \eqref{aa_eq:gradProj_net1} and \eqref{aa_eq:gradProj_net2} as multiplications in Fourier space. Consider the Fourier coefficients for a certain resolution (small $n$) to be the parameters of an FNO-inspired network and apply the convolutions with zero padding at higher resolution. Note that it is potentially necessary to rescale the kernels after interpolation in order to preserve the integral value (see also chapter \ref{aa_subsec:fourier_transformation} Transforming the Fourier padded kernel back to the spatial domain, Fig.~\ref{aa_fig:kernelIllustration} illustrates their shape. As we can see, the interpolated kernel $\tilde{h}$ remains unaffected. However, this is not the case for the kernel $\tilde{g}$: adding the skip connection and padding modifies the kernel such that, in the Fourier domain, it behaves like a rectangular window function. Consequently, this results in a $\text{sinc}$-function in the time domain (see Figure ~\ref{aa_fig:kernelIllustration}).
      
\end{itemize}

\begin{figure}
    \centering
    \begin{subfigure}{0.43\linewidth}
     \includegraphics[width=\linewidth]{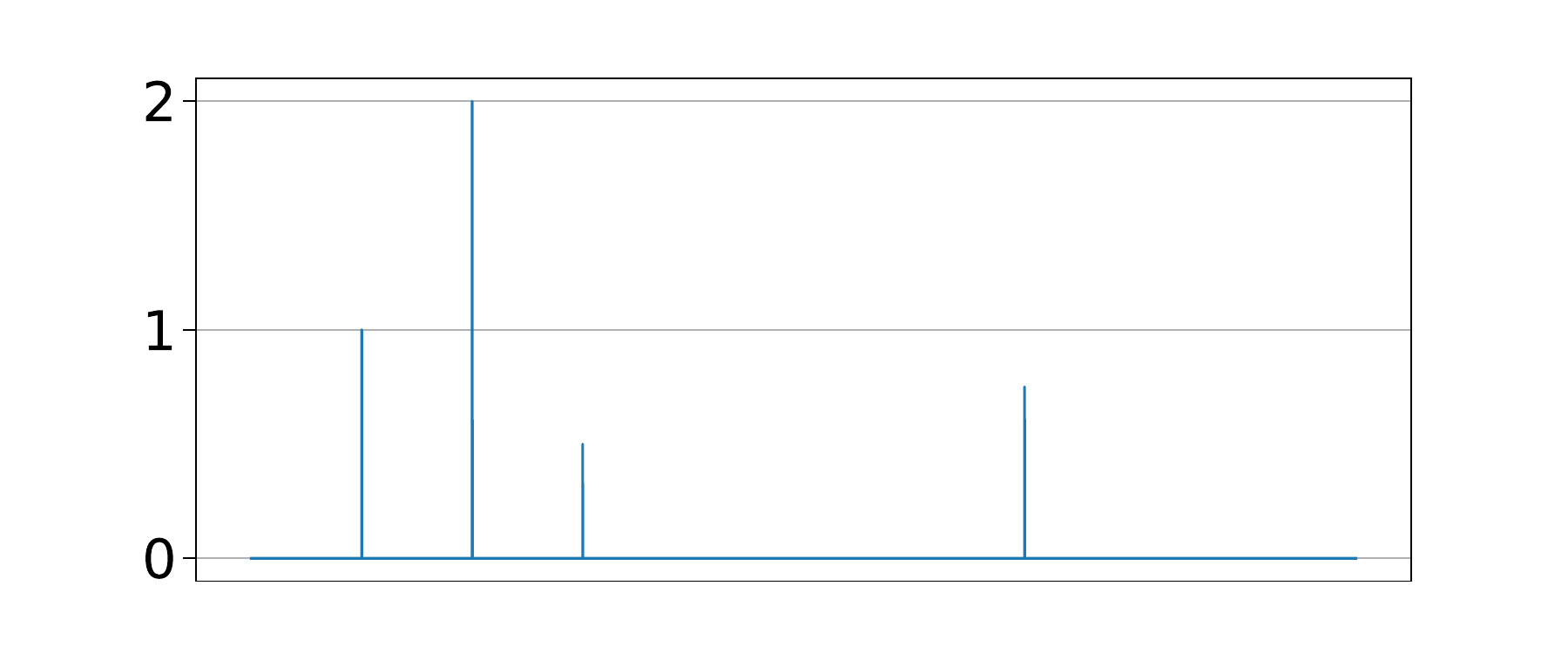}
     \subcaption{Ground truth signal}
\end{subfigure}
\begin{subfigure}{0.43\linewidth}
     \includegraphics[width=\linewidth]{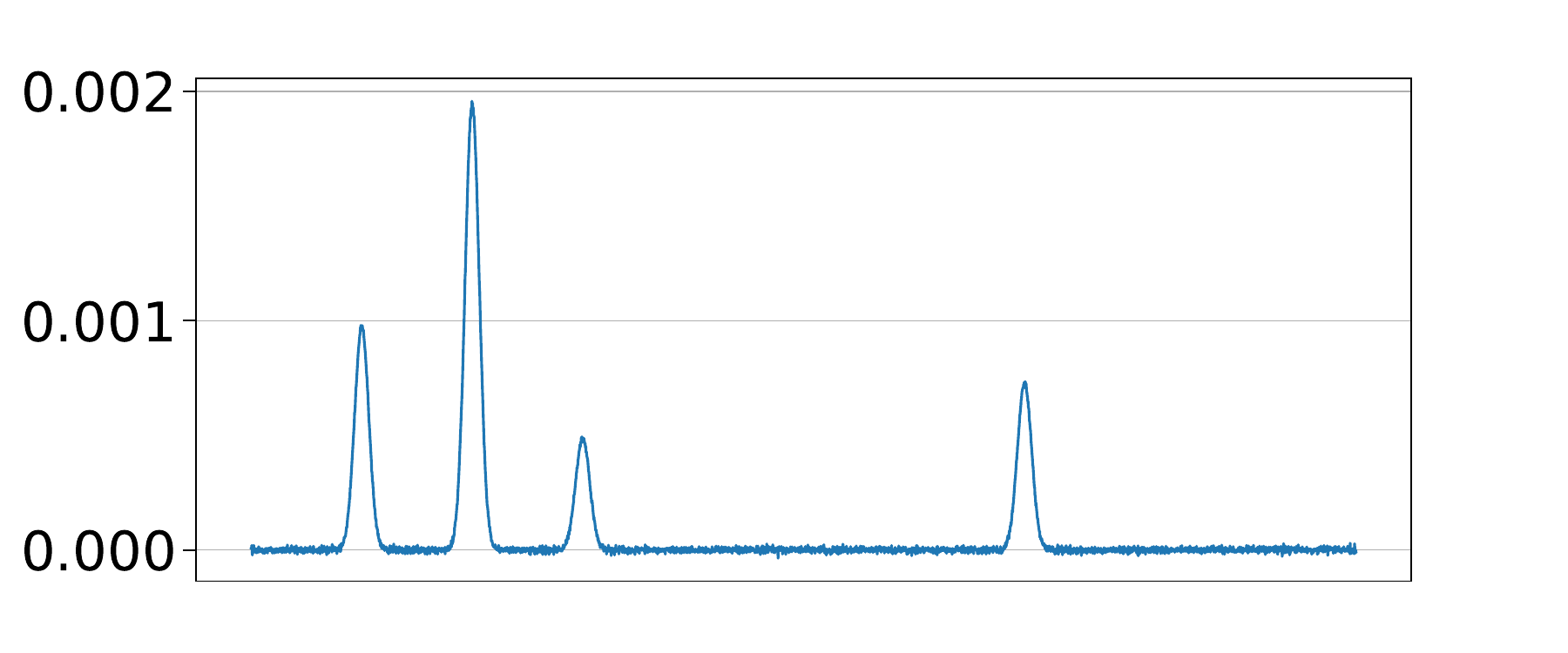}
     \subcaption{Noisy measurement.}
\end{subfigure}
\begin{subfigure}{0.43\linewidth}
     \includegraphics[width=\linewidth]{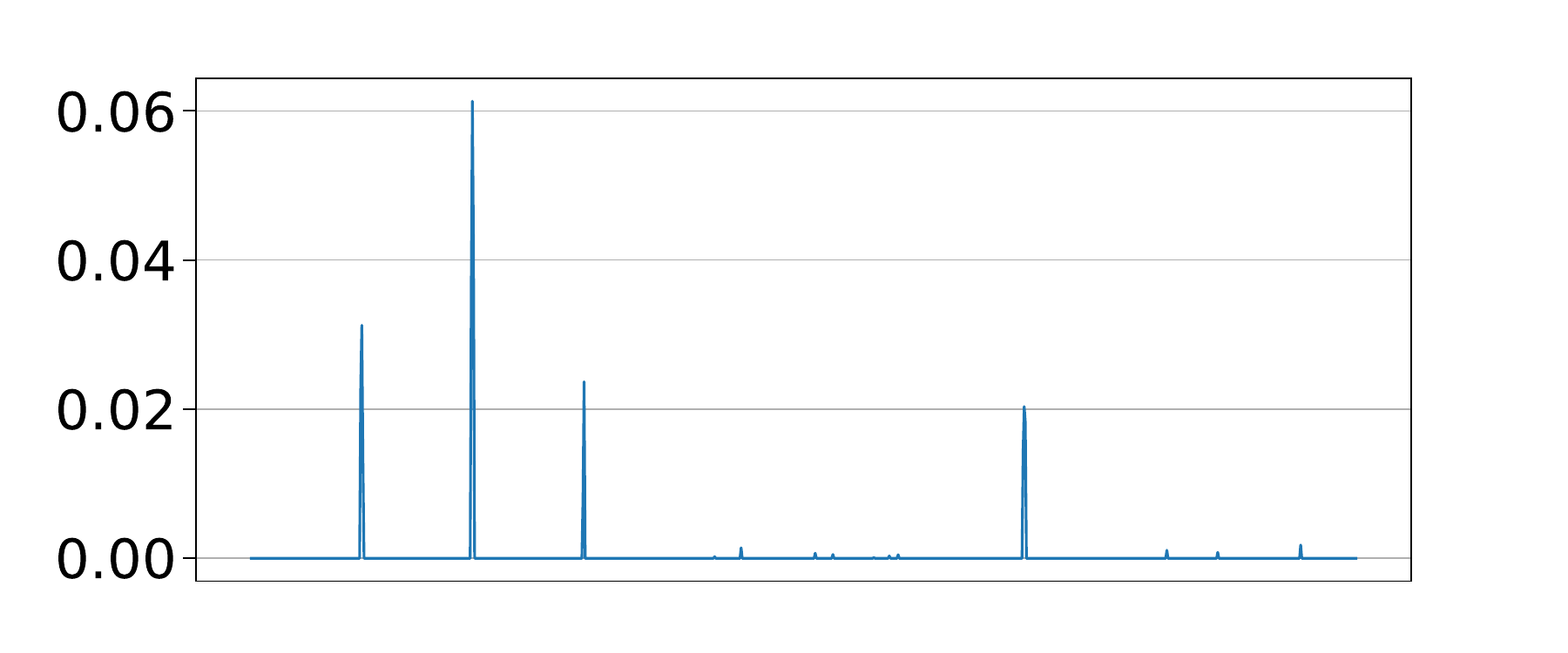}
     \subcaption{Reconstruction with large kernel}
\end{subfigure}
\begin{subfigure}{0.43\linewidth}
     \includegraphics[width=\linewidth]{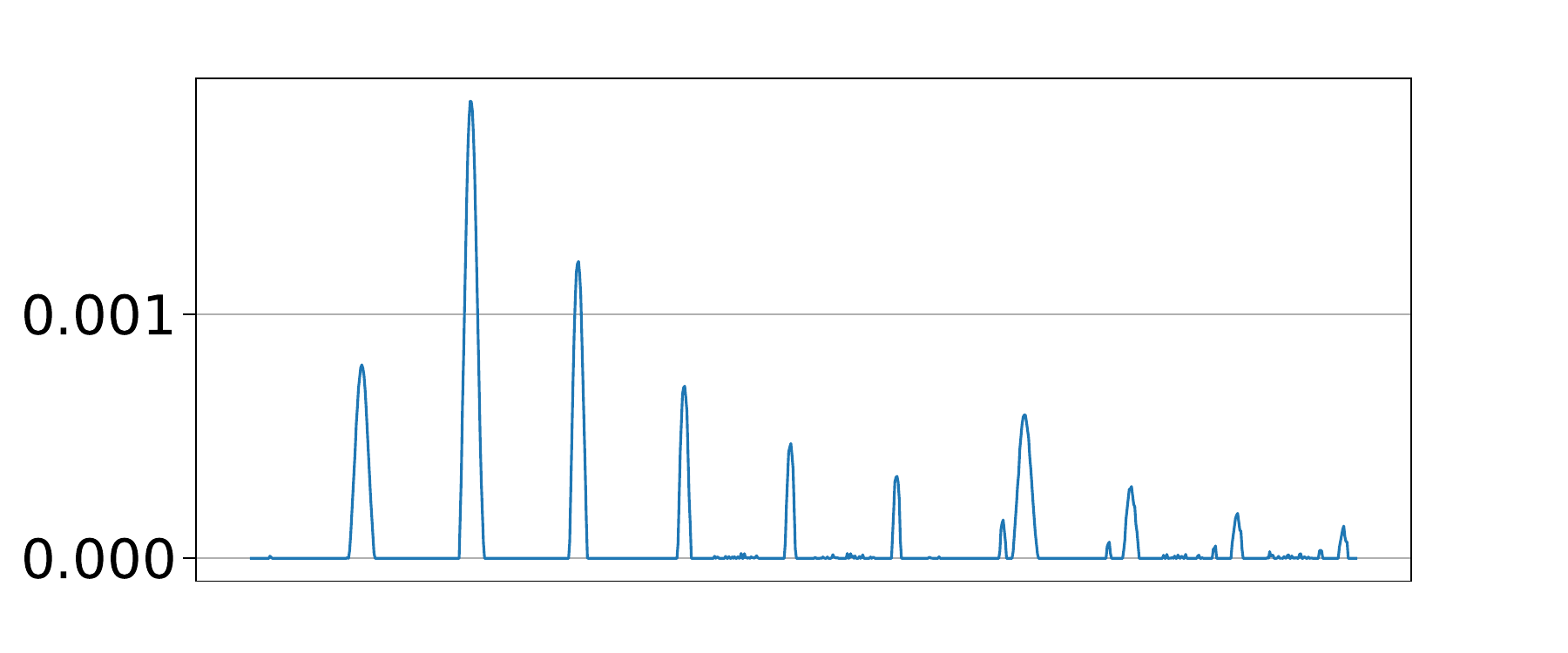}
     \subcaption{Reconstruction with small kernel.}
\end{subfigure}
\begin{subfigure}{0.43\linewidth}
     \includegraphics[width=\linewidth]{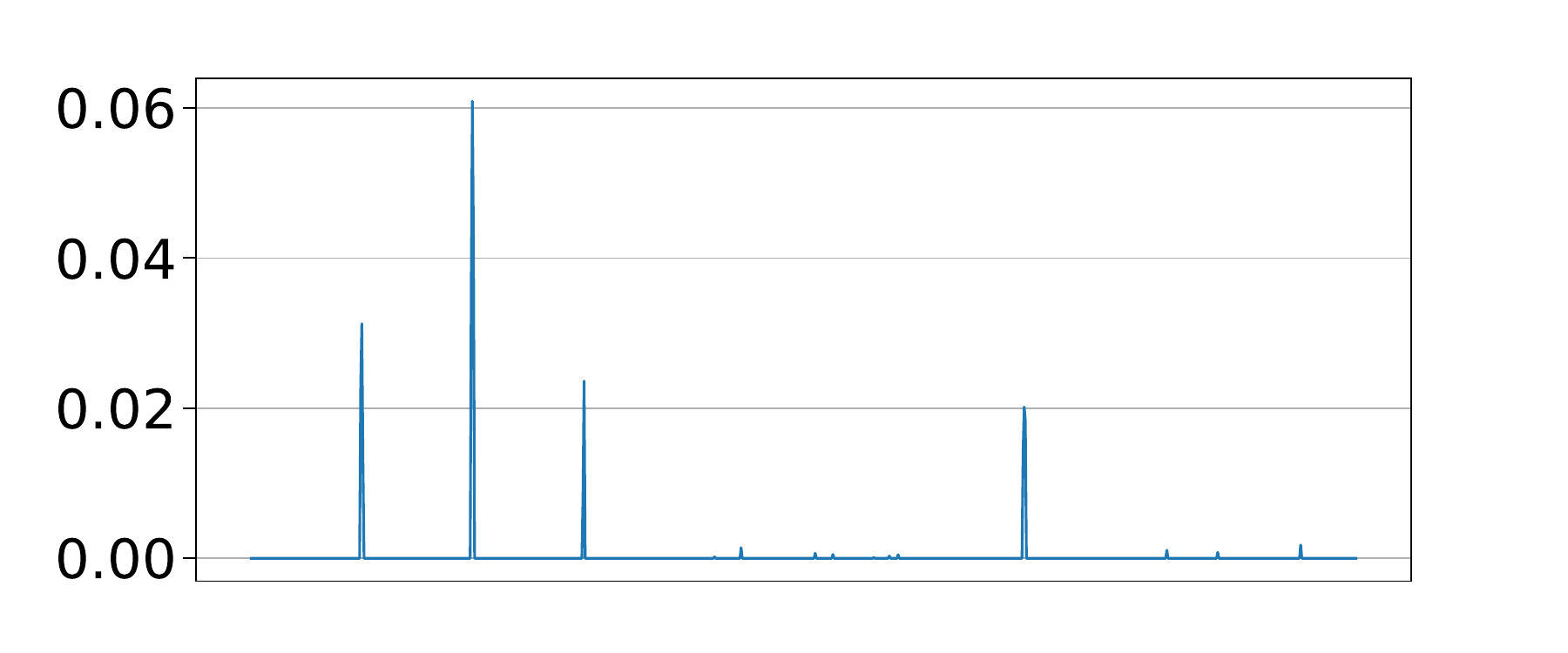}
     \subcaption{Reconstruction with interpolation of $\tilde{h}$.}
\end{subfigure}
\begin{subfigure}{0.43\linewidth}
     \includegraphics[width=\linewidth]{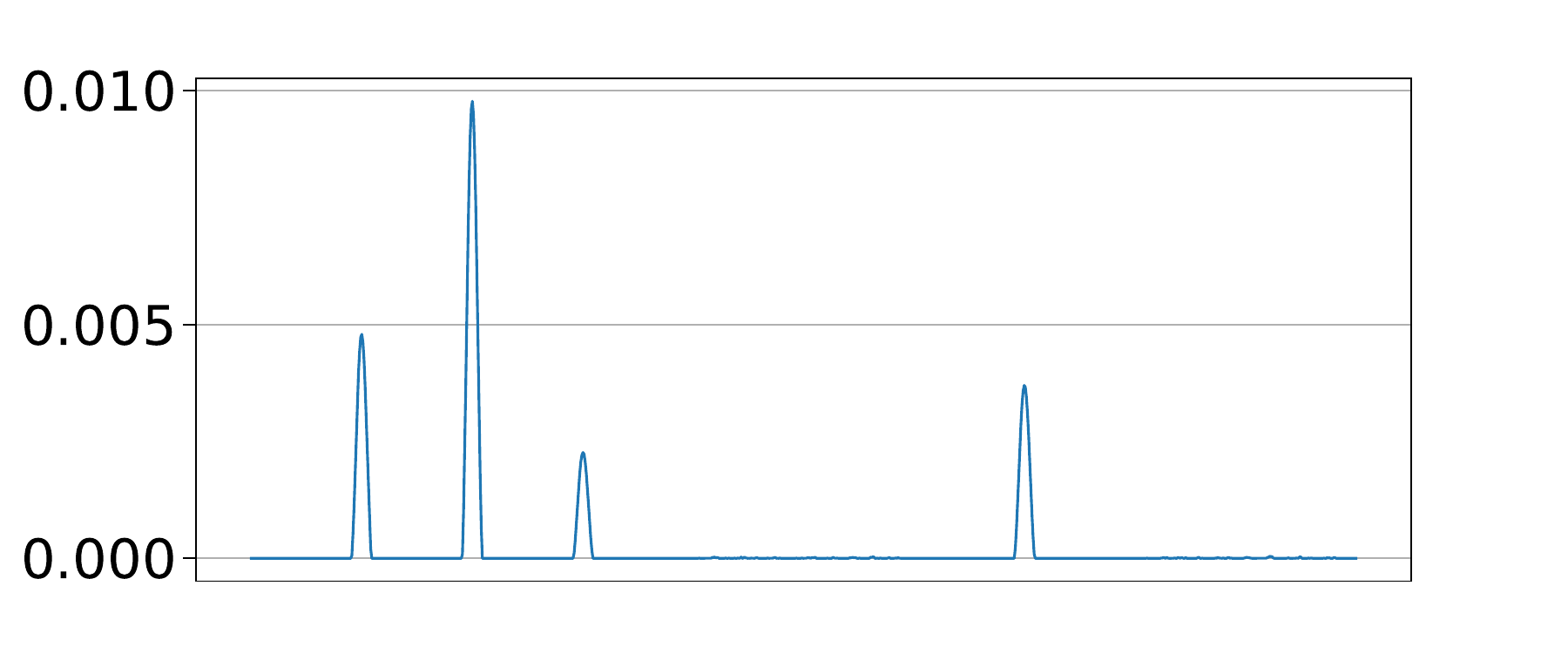}
          \subcaption{Reconstruction with interpolation of $\tilde{g}$.}
\end{subfigure}

\begin{subfigure}{0.43\linewidth}
     \includegraphics[width=\linewidth]{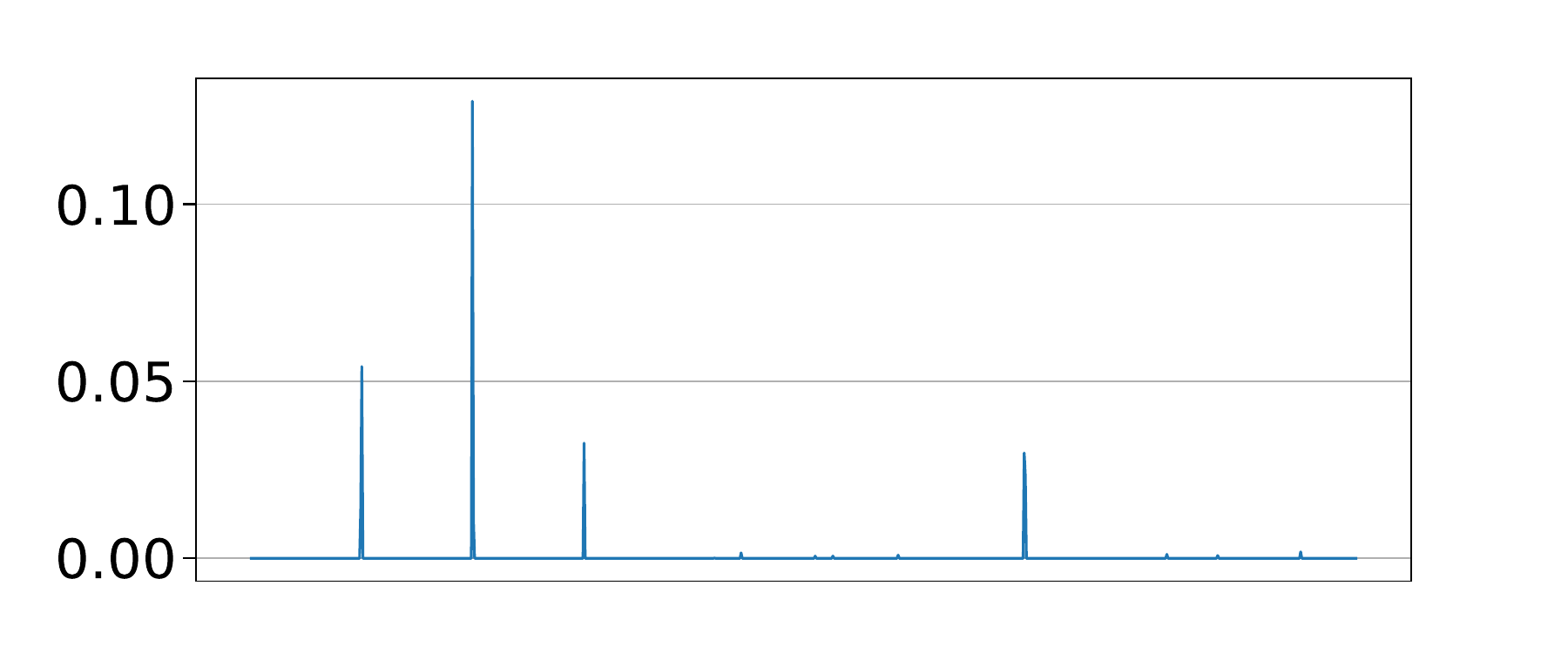}
     \subcaption{Reconstruction with spectral interpolation of $\tilde{h}$.}
\end{subfigure}
\begin{subfigure}{0.43\linewidth}
     \includegraphics[width=\linewidth]{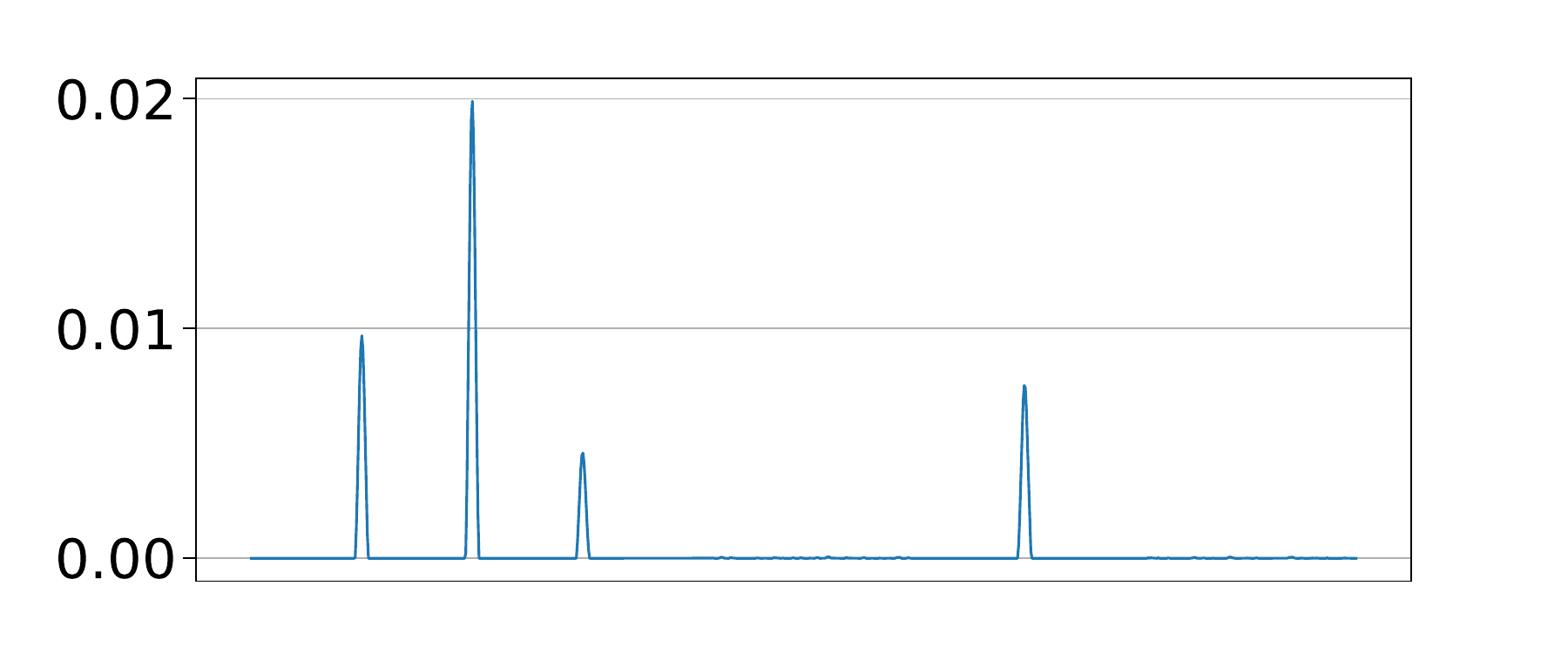}
     \subcaption{Reconstruction with spectral interpolation of $\tilde{g}$.}
\end{subfigure}
    \caption{The results with different kernel.  The interpolation of $\tilde{h}$ reconstructs the sparsity of the ground truth signal, while the interpolation of $\tilde{g}$ leads to a more spread-out signal.}
    \label{aa_fig:toy-results}
\end{figure}
To demonstrate the five different choice of handling different resolutions within a CNN motivated by our projected gradient descent algorithm, we simulate the continuous $f$, $u$ and $k$ using a discretization of $n=64000$ points. We then subsample $f$ by a factor of $16$ and add noise to obtain $\tilde{f} \in \mathbb{R}^{4000}$. We create Gaussian kernels with $\sigma = 300$ at a resolution suitable for reconstructing ($n=1000$) and use the five different approaches described above (use low-resolution kernel, use linear interpolation of $\tilde{g}$, use linear interpolation of $\tilde{h}$, use spectral variant of \eqref{aa_eq:gradProj_net1}, use spectral variant of \eqref{aa_eq:gradProj_net2}) and show the reconstructions in Fig.~\ref{aa_fig:toy-results}. Interpolating the smaller kernel, defined as $\tilde{h} = -\tau (k * k)$, yields results comparable to those obtained with the larger kernel. In this case, the sparsity of the ground truth signal can be effectively reconstructed. In contrast, interpolation using $\tilde{g}$ fails to recover this sparsity; the output primarily reflects the smooth intensity profile of the Gaussian kernel. Similar results were observed for the spectral variants. These observations suggest that incorporating a skip connection from the input to the corresponding layer, as implemented in Equation~\ref{aa_eq:gradProj_net2}, is a more effective modeling strategy.

Overall, the approach of interpolating kernels appears to be a promising strategy. In the following, we shift our focus to a more practical setting, where we aim to solve limited-angle CT reconstruction problems by \textit{learning} suitable convolution kernels for the reconstruction, rather than deriving them analytically through a model-based approach. 

\subsection{Limited angle CT post-processing}
Inspired by the results on the one-dimensional toy example we took a closer look into the behavior of different neural network architectures on imaging data based on a common real world application, namely limited angle computerized X-ray tomography.
This imaging modality is commonly encountered in medical and scientific applications. While for full-angle tomography data, there exists a unique analytical solution (at least in the theoretical case of perfect measurements), limited angle tomography is an underdetermined inverse problem, leading to so-called streaking artifacts in naive reconstructions. To remove these, common approaches use various combinations of variational or data-driven approaches to pre-process, reconstruct and post-process limited angle CT measurements.

As we intend to stay close to a real-world application while also staying as simple as possible regarding theoretical analysis and practical implementation we restrict our experiments to post-processing. We therefore consider naive reconstructions generated from measurements of different resolutions as input data and the corresponding artifact free reconstructions as target data. Since data collection for such a task is costly we use simulated data.
\subsubsection{Data generation}
\begin{figure}
    \centering
    \includegraphics[width=0.75\linewidth, trim = {0cm 0cm 0cm 0.5cm}, clip]{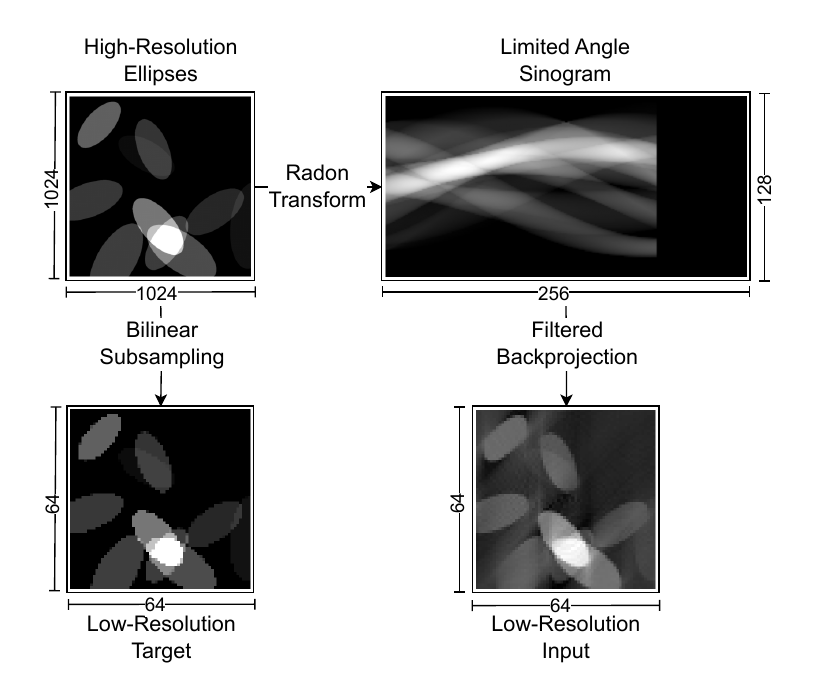}
    \caption{An overview over the data generation process to simulate the limited angle CT data used in our experiments, here exemplary for images of size $64\times64$.}
    \label{aa_fig:dataGeneration}
\end{figure}
Figure \ref{aa_fig:dataGeneration} visualizes our data generation process. We generate high-resolution ($1024\times1024$ pixels) groundtruth images consisting of ellipses with random sizes, positions, rotations and densities. These high-resolution images were then transformed using the Radon transform with angles in $[0,\frac{3}{4}\pi]$, resulting in an incomplete but high-resolution sinogram missing approximately $25\%$ of angles. As a sinogram size (determined by the amount of angles and offsets) with the (vectorized) dimension as the groundtruth proved insufficient for reconstructing images in our framework we empirically determined a sufficient size as $2N$ angles and $4N$ offsets, where N is the width/height of the final image. Finally, we generated each final data sample by constructing the target by bilinear subsampling of the high-resolution groundtruth image while constructing the input by performing a ram-lak filtered backprojection onto the desired input size. As the focus of our work lies on the behavior across different resolutions we chose to not include any noise, resulting in an easier reconstruction problem and thus hopefully easier training of the different models.

We generated 3 datasets with different resolutions ($64\times64$, $128\times128$ or $256\times256$), each having $10000$ samples, split between training-, validation- and test-data ($64\%/16\%/20\%$). To guarantee fairness we used identical training data across all training runs.

\subsubsection{Architectures}
We analyze six different neural network architectures:
\par
\textbf{Classical U-Net} Visualized in Figure \ref{aa_fig:UNetArchi} we implemented this popular image-to-image architecture as described in section \ref{aa_sec:discrete} and \cite{ronneberger2015unet_AA}, with a minor modification: The convolutional layers in the original U-Net do not perform zero-padding, which causes the outputs to be slightly smaller than the inputs. In our implementation we use zero-padding as this generates outputs that have the same size as the inputs and thus better fits the idea of post-processing. As in the original work, we use a depth of 4, meaning we operate on 5 different resolutions. Every downsampling step doubles the amount of channels while halving the image resolution, which is exactly inverted in the upsampling part. An initial convolution extends the amount of channels to 64, which is inverted at the end by a final convolution. This U-Net structure is used as reference for all other architectures, with the goal of staying as close as possible.
\par
\textbf{Differential U-Net} The differential U-Net has the same structure as our U-Net implementation, with the only difference being the use of differential convolution filters, as described in section \ref{aa_sec:differential}. Based on our discussions in Section \ref{aa_sec:differential}, we do not expect a resolution-invariant behavior for such an architecture, as we neither include a thresholding layer to cut off exploding gradients, nor adapt the number of layers to the input resolution.
\par
\textbf{Spectral U-Net} A U-Net inspired by Fourier neural operators, as introduced in \cite{li2021fourier_AA}, described in section \ref{aa_subsec:spectralUNet}. Structurally our implementation again follows the U-Net as described above, as we also use a depth of $4$ and extend the number of channels to $64$ in the first layer. As described in Definition \ref{aa_def:spectralunet}, we do not further increase the number of channels in the contracting path and implement the skip connections as additions instead of concatenations. This makes it possible to add residual connections in each layer, as we have argued for in section \ref{aa_subsec:spectralUNet}. We further choose the kernel sizes $m_{\text{in}}= m_4 = 256$, $m_3 = 128$, $m_2 = 64$, $m_1 = 32$, $m_0 = 16$. 
\par
\textbf{Spectral resizing U-Net} The spectral resizing U-Net is similar to the basic spectral U-Net described above. The main difference is the inclusion of up- and downsampling operations in the form of trigonometric interpolation.
\par
\textbf{CNO} We use the original CNO implementation\footnote{The code for the CNO architecture is available at \\
\url{https://github.com/camlab-ethz/ConvolutionalNeuralOperator/tree/main/CNO2d_original_version}.} as described in \cite{raonic20cno_AA}, with hyperparameters such as depth, amount of initial channels etc. chosen in a way to most closely resemble the classical U-Net.
\par
\textbf{U-NO} We use the original U-NO implementation\footnote{The code for the U-NO architecture is available at\\ \url{https://github.com/neuraloperator/neuraloperator/blob/main/neuralop/models/uno.py}.} as described in \cite{rahman2023uno_AA}, with hyperparameters (number of channels, skip connections, kernel sizes) chosen such that it matches the spectral resizing U-Net.

\subsubsection{Results}
Even if a network contains resolution-agnostic building blocks, we can only assess empirically if it is a suitable model to solve our problem. We consider the following questions:
 \begin{enumerate}
     \item[(Q1)] Performance for fixed input resolution: How well does the network perform on a single resolution when it was trained on the same resolution?
     \item[(Q2)] Performance for varying input resolutions: How well does the network perform on varying resolutions when it was trained on the same varying resolutions?
     \item[(Q3)] Generalization to unseen input resolutions: How well does the network perform on resolutions that were not seen during training? Does the performance improve if the range of training resolutions increases?
     \item[(Q4)] Performance for resized input resolutions: How well does a network trained on a single resolution perform on images of varying resolutions resized to the resolution the network was trained on?
\end{enumerate}
We evaluate each architecture by training 10 instances and reporting the mean performance. The following figures show the performance of each architecture in terms of MSE, trained either on 6400 images of $64\times64$, $128\times128$ or $256\times256$ pixels, or an equal mixture of all three resolutions with the same total amount of images.

We would like to emphasize that we did not perform an extensive search for the various hyperparameters, which could be a source for further improvement. Since a careful selection of hyperparameters will most probably lead to a good performance on a fixed input resolution, we do not put much weight to the first question (Q1). Instead, we use the performance obtained on a fixed resolution as a reference value to answer the remaining questions. We further note that since the input resolution of the CNO is tied to the fixed training resolution, we only consider questions (Q1) and (Q4) for this architecture. 
\begin{figure}
    \centering
    \begin{subfigure}{0.47\linewidth}
        \includegraphics[width=\linewidth]{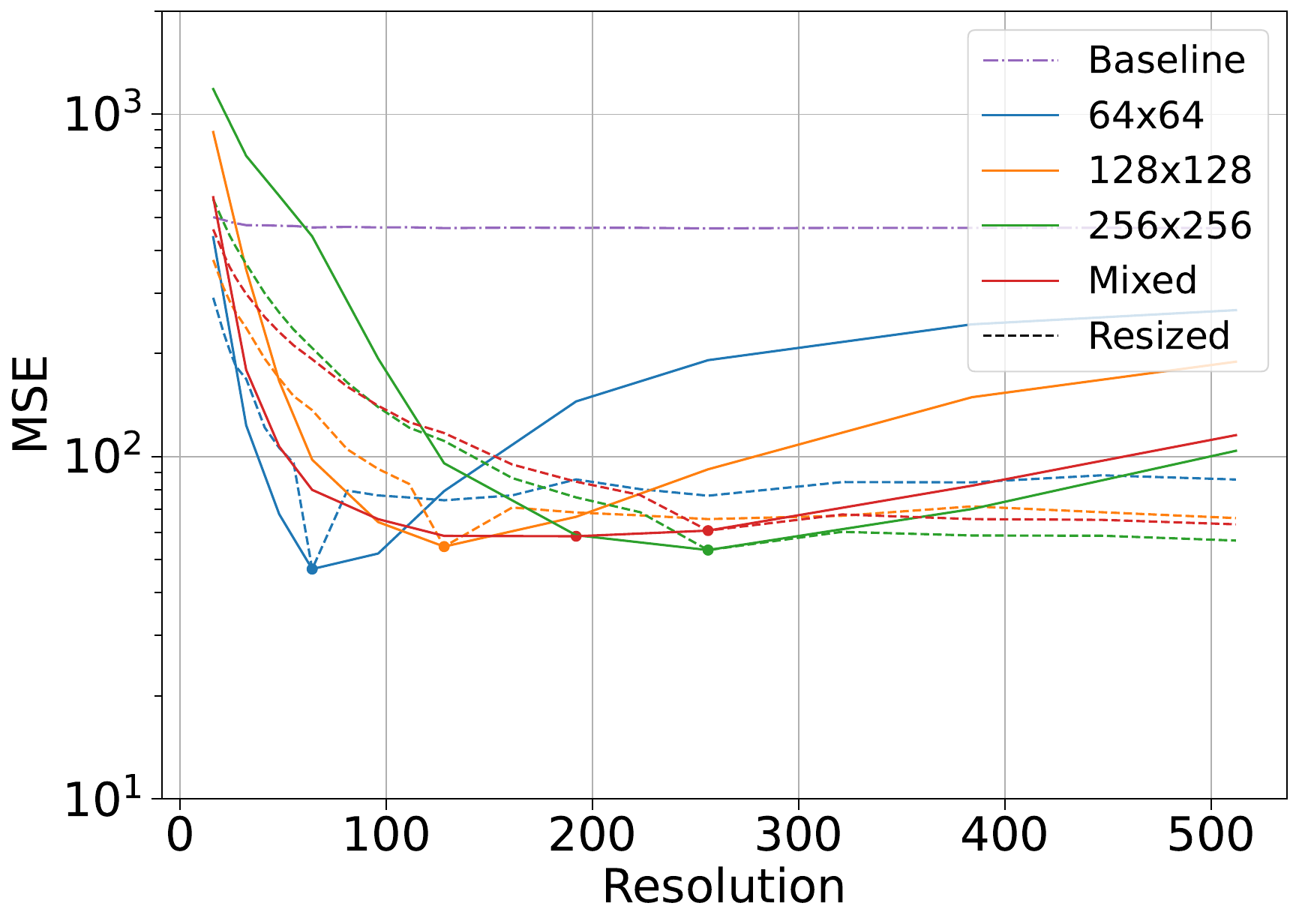}
        \subcaption{Classical U-Net\\\quad}
    \end{subfigure}
    \begin{subfigure}{0.47\linewidth}
         \includegraphics[width=\linewidth]{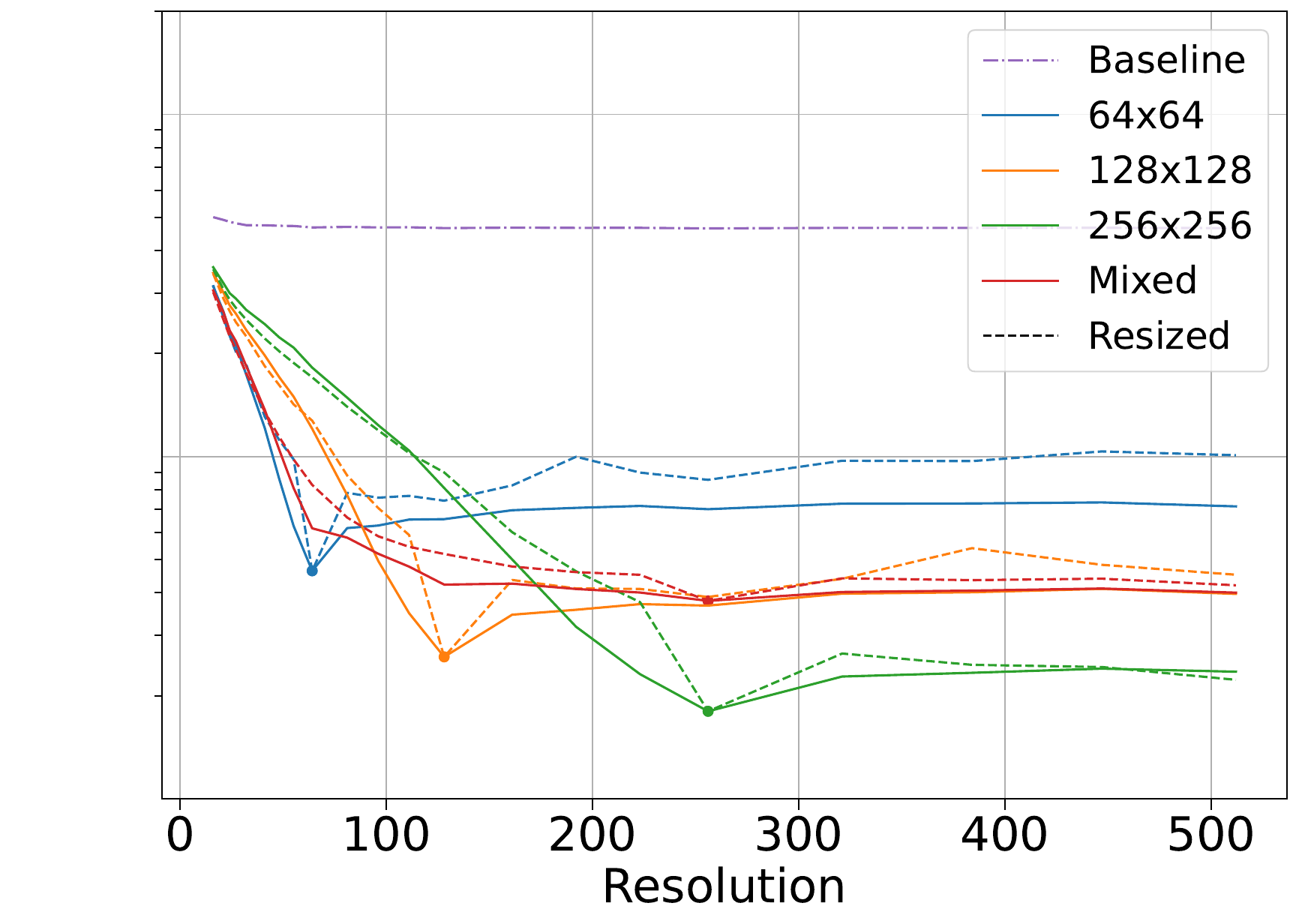}
         \subcaption{Spectral U-Net\\\quad}
    \end{subfigure}
    \begin{subfigure}{0.47\linewidth}
         \includegraphics[width=\linewidth]{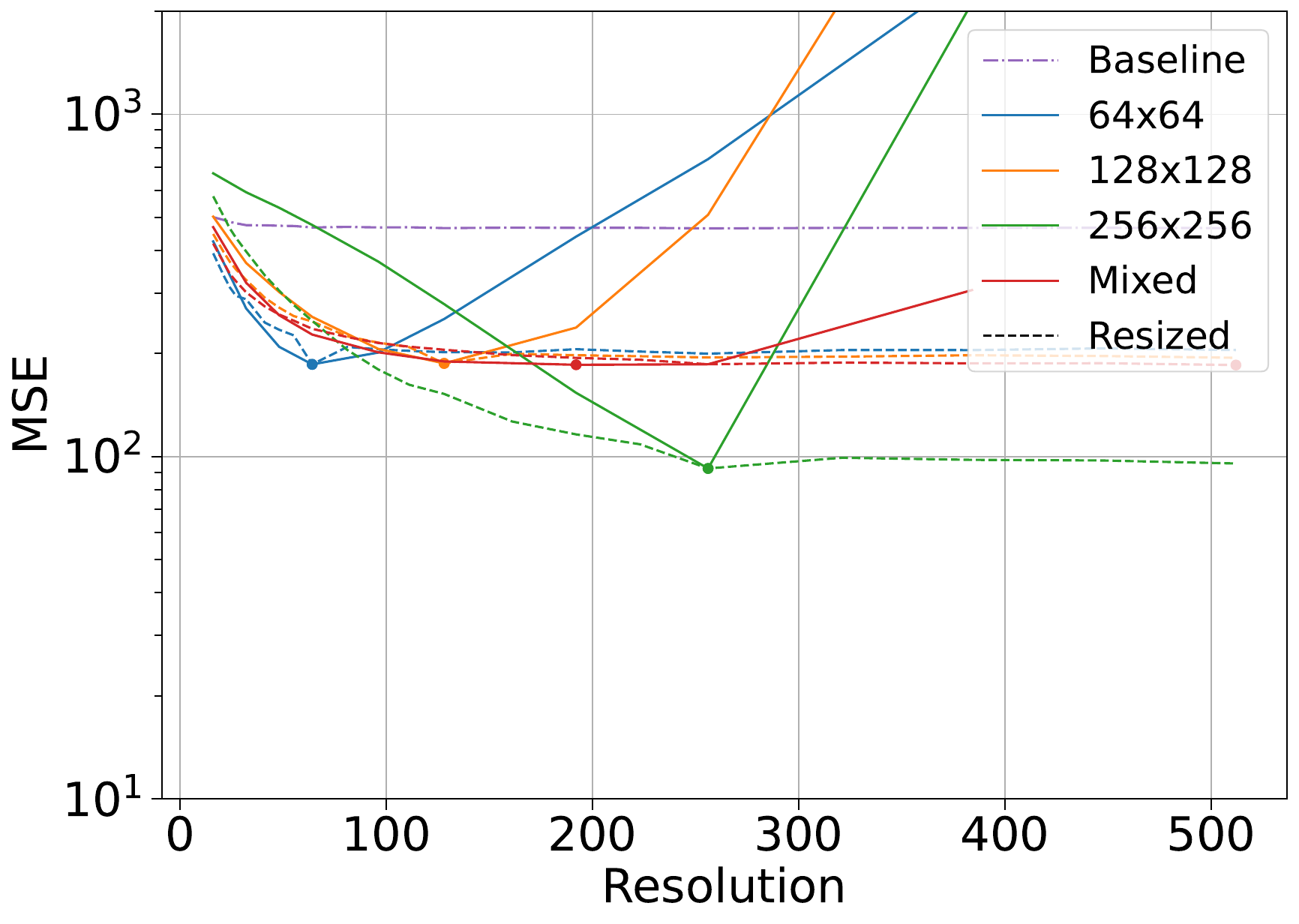}
         \subcaption{Differential U-Net\\\quad}
    \end{subfigure}
    \begin{subfigure}{0.47\linewidth}
        \includegraphics[width=\linewidth]{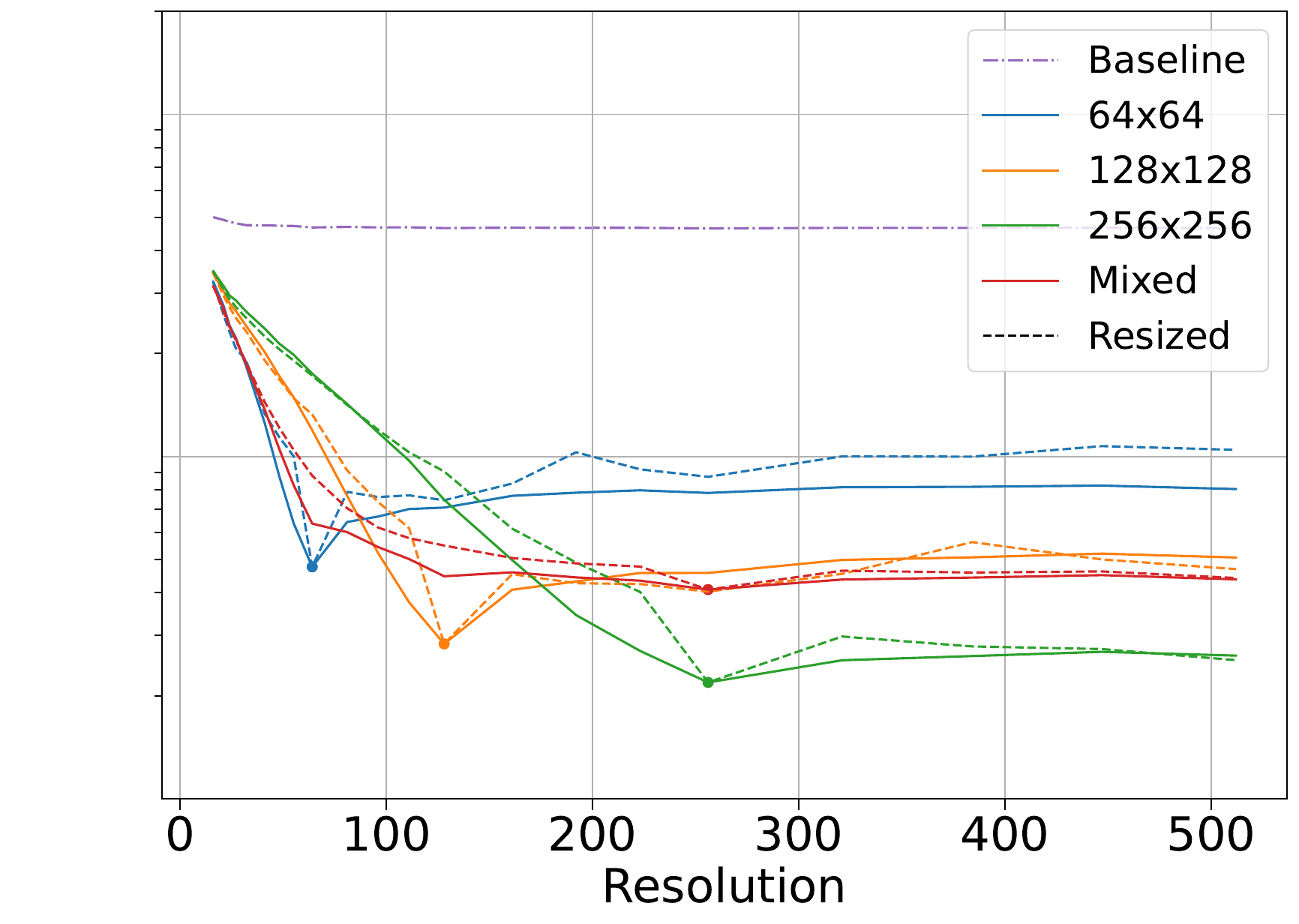}
        \subcaption{Spectral resizing U-Net\\\quad}
    \end{subfigure}
    \begin{subfigure}{0.47\linewidth}
         \includegraphics[width=\linewidth]{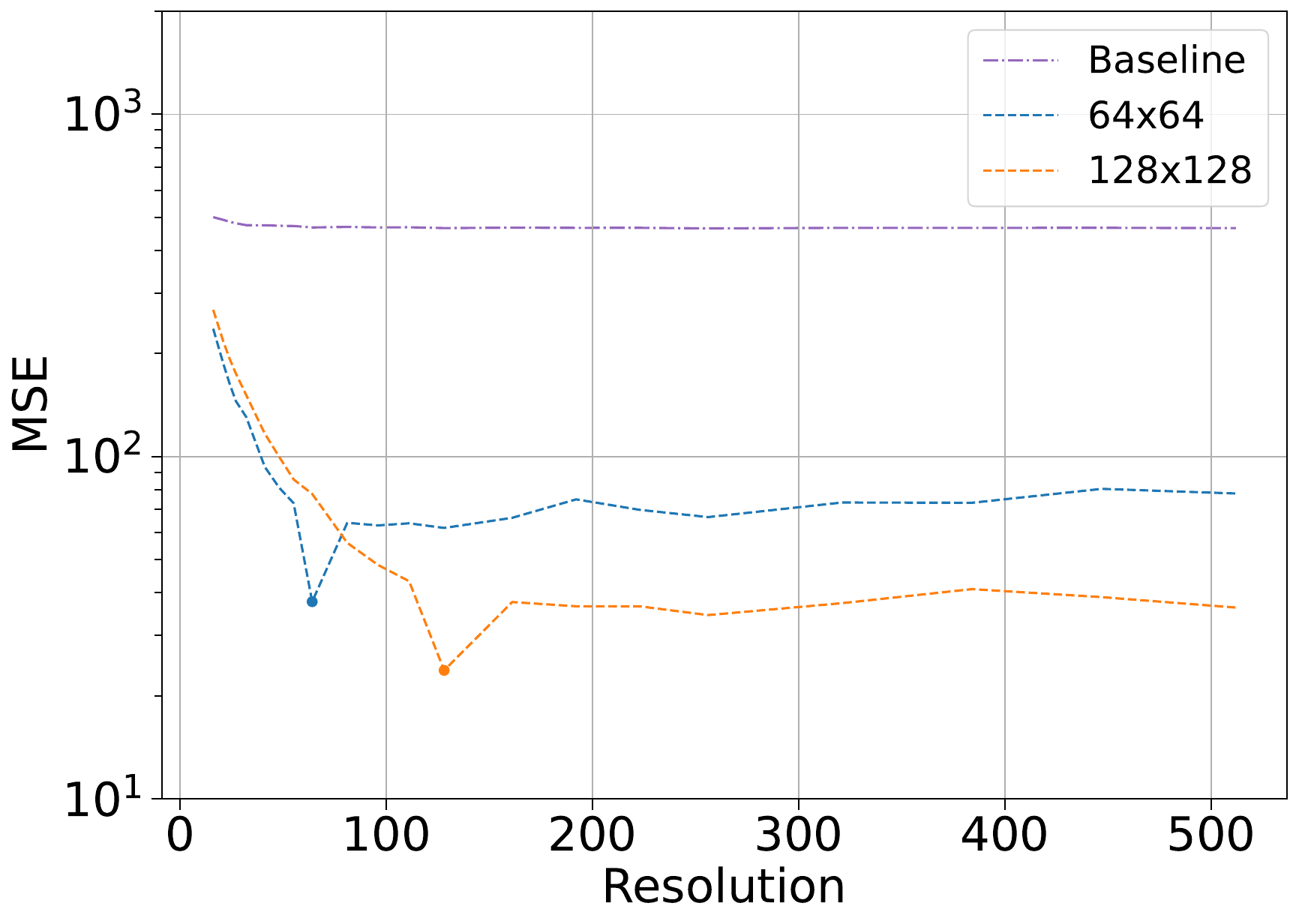}
         \subcaption{CNO}
    \end{subfigure}
    \begin{subfigure}{0.47\linewidth}
         \includegraphics[width=\linewidth]{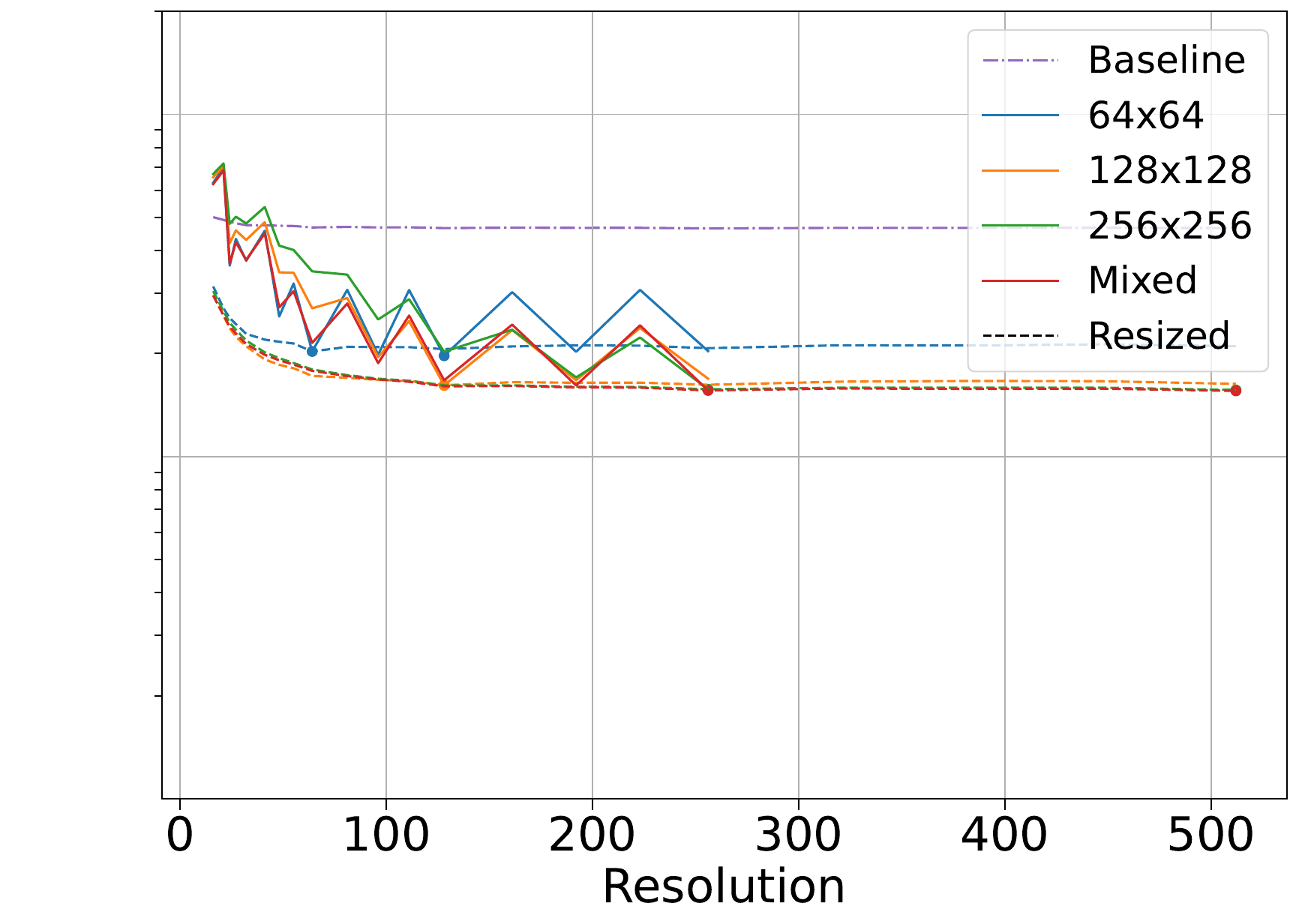}
         \subcaption{U-NO}
    \end{subfigure}
    \caption{Performance comparison for different architectures on different resolutions, with dots marking the resolution of minimal MSE. Dashed lines indicate the performance if the test-inputs are resized to the size of the training samples. ``Baseline'' indicates the MSE of the initial reconstruction.}
    \label{aa_fig:results2dMeans}
\end{figure}

Our findings are visualized in Figures \ref{aa_fig:results2dMeans} to \ref{aa_fig:results2dVarsR}, where Figure \ref{aa_fig:results2dMeans} shows the mean performance and Figures \ref{aa_fig:results2dVars} and \ref{aa_fig:results2dVarsR} give further insights on the performance range.
While the curves showing the performance for models trained on a fixed input resolution (``$64\times64$'', ``$128\times128$'', ``$256\times256$'') give answers to (Q1), the models trained on mixed input resolutions (``mixed'') are related to (Q2). While for these first two questions, we only consider the performance at the respective training resolutions, the plots further shows the performance on unseen resolutions. Comparing these results for models trained on single fixed resolutions to the ones trained on mixed resolutions gives insights about (Q3). Lastly, (Q4) is taken into account with the dashed curves, which show the performance on inputs that have been resized to the training resolution.
In the following we describe our findings for each architecture.
\par\textbf{Classical U-Net} The classical U-Net performs quite well, however we see a large deviation between different runs in Figure \ref{aa_fig:results2dVars}.  Training on mixed data slightly worsens the maximum performance but increases the generalization capabilities of the network. While the models perform best on the resolution they were trained on, they also generalize to a broader range of resolutions than expected. Training on mixed resolutions further increases the performance on unseen low resolutions with a slight performance trade-off on high resolutions. 
Resizing the input data to the training input size (dashed lines) drastically improves generalization. This indicates that the architecture itself is not resolution-invariant.

\par\textbf{Differential U-Net} As clearly depicted in Figure \ref{aa_fig:results2dMeans}, simply substituting classical convolutional layers by differential ones without further adaptations affects the networks' performance very negatively. We attribute the lower performance on the training resolutions to the scaling introduced in differential layers, that seems to make training less stable. Taking into account this in general worse performance, we observe a similar change in performance as for the classical U-Net when training on mixed resolutions. The generalization to unseen resolutions is substantially worse than for all other tested architectures.  

Resizing the inputs to the respective training resolution yields consistent results across nearly all resolutions. From these observations we conclude that the chosen differential architecture is not resolution-invariant.
\par\textbf{Spectral U-Net} The performance obtained with the spectral U-Net is comparable to the best runs of the classical U-Net, while the deviation observed between different runs is much lower for the spectral version (Figure \ref{aa_fig:results2dVars}). All models trained on a single fixed resolution generalize well to data of high resolutions, while performing worse on lower resolutions.
Training on mixed data does not seem to improve the results at all, neither on a single resolution nor regarding the generalization capabilities. The same applies to resizing the input data to the network’s training resolution: In nearly all cases, resizing strictly decreases network performance. This indicates the strong resolution independence of spectrally parameterized networks and resizing only introduces interpolation errors.

\par\textbf{Spectral resizing U-Net} The resizing version of the spectral U-Net behaves nearly identically to the non-resizing version, except for its negligibly worse performance in general. 
The advantage of the resizing version is the decrease in memory consumption, see figure \ref{aa_fig:memReq}.

\par\textbf{CNO} Due to the high memory requirements of the CNO model (see Figure \ref{aa_fig:memReq}, mostly caused by intermediate values in the computation) and our available hardware, we had to restrict our numerical experiments to resolutions smaller than $512\times512$. A further constraint was given by the architecture, which does not allow for training (or testing) with mixed resolutions, we therefore only compare the resizing variants. For the remaining cases the networks performed similar to the spectral U-Net variants, producing slightly better results.

\par\textbf{U-NO} As this model also requires high memory (Figure \ref{aa_fig:memReq}), we needed to restrict our tests to resolutions smaller or equal to $256\times256$. Compared to other models, U-NO did not achieve competitive results on any resolution, but we attribute this to our choice of hyperparameters. From Figure \ref{aa_fig:results2dVars} we see that the  performance is almost identical in all 10 runs. Training on mixed resolutions does not seem to change the models performance. For unseen resolutions, we observe a strong zig-zag pattern between good peformance on even resolutions and bad performance on odd resolutions. The reason for this could lie in the used implementation of the spectral convolution, that in the case of real tensors of even size differs from the spectral convolution we use for the spectral (resizing) U-Net (cf. \cite{Kabri2023FNO_AA}). 

\begin{figure}
    \centering
    \begin{subfigure}{0.47\linewidth}
        \includegraphics[width=\linewidth]{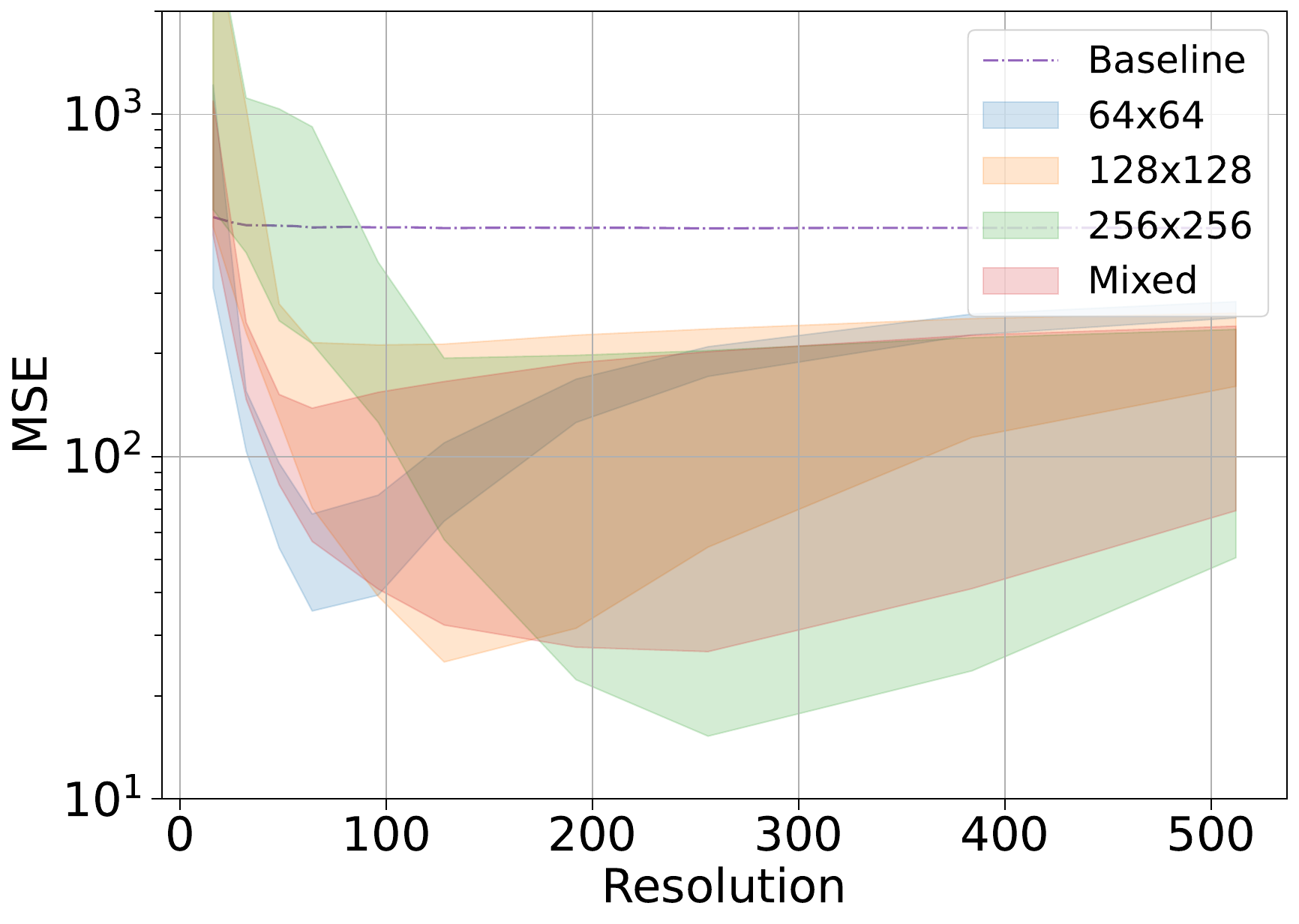}
        \subcaption{U-Net\\\quad}
    \end{subfigure}
    \begin{subfigure}{0.47\linewidth}
         \includegraphics[width=\linewidth]{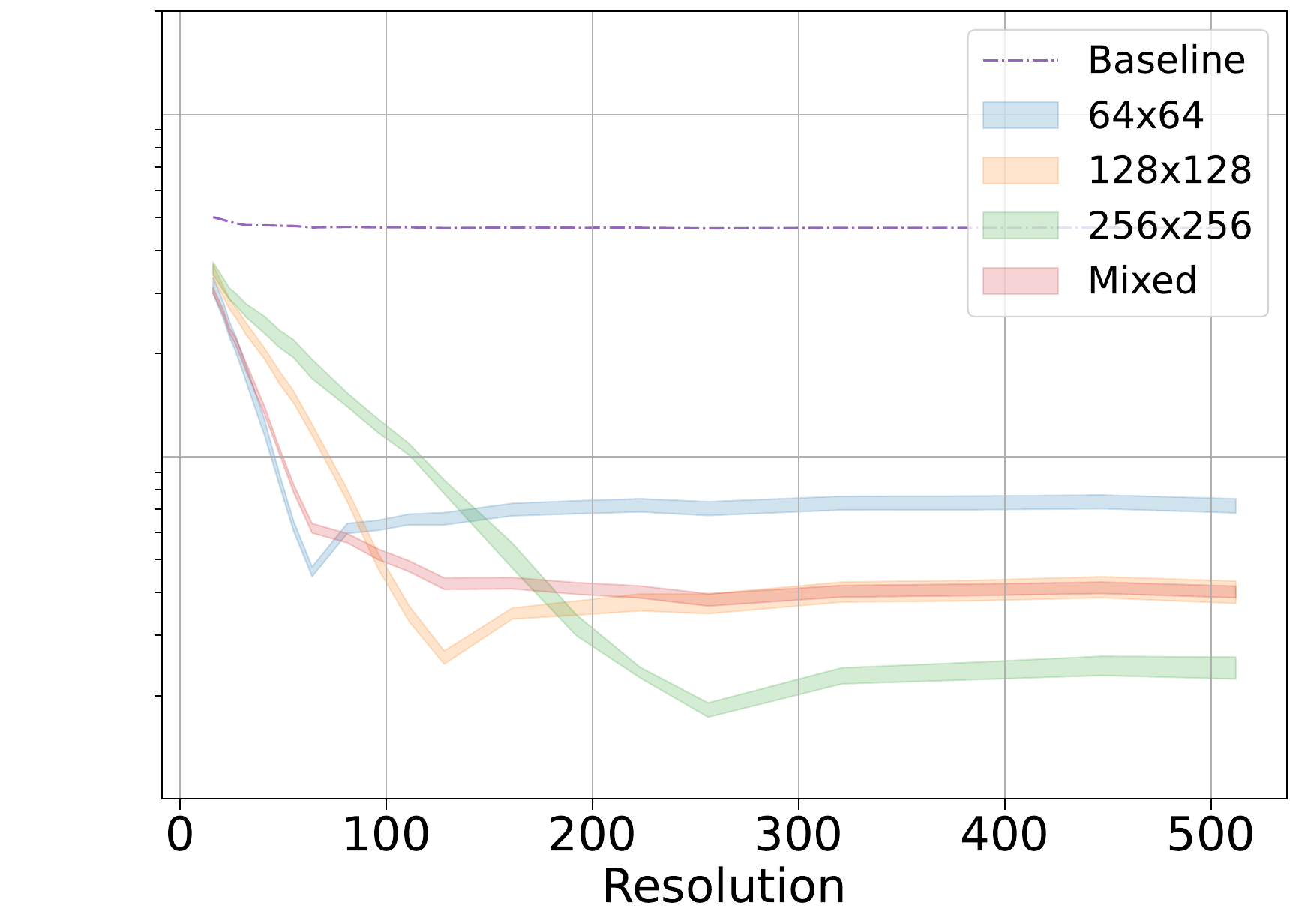}
         \subcaption{Spectral U-Net\\\quad}
    \end{subfigure}
    \begin{subfigure}{0.47\linewidth}
         \includegraphics[width=\linewidth]{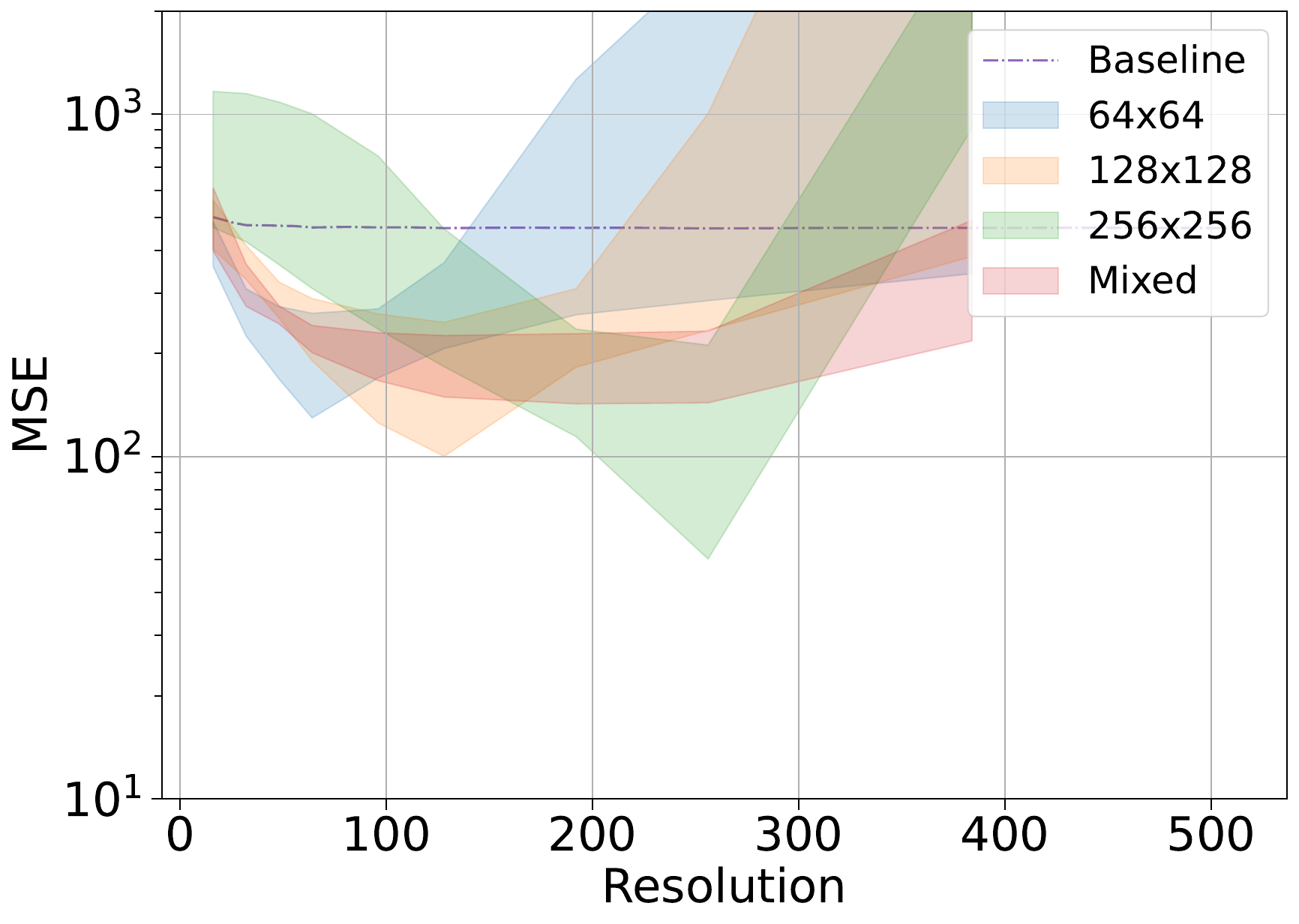}
         \subcaption{Differential U-Net\\\quad}
    \end{subfigure}
    \begin{subfigure}{0.47\linewidth}
        \includegraphics[width=\linewidth]{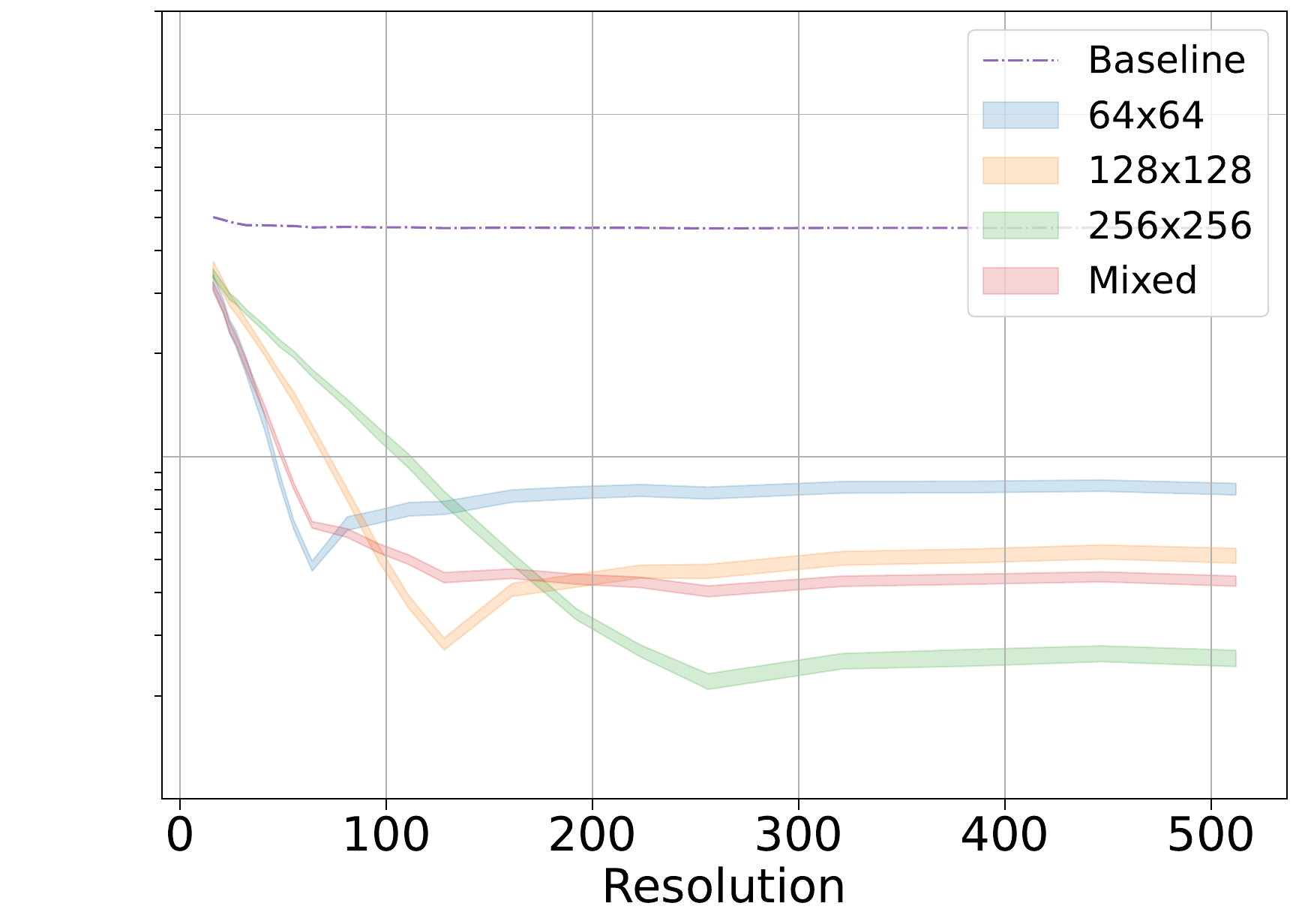}
        \subcaption{Spectral resizing U-Net\\\quad}
    \end{subfigure}
    \begin{subfigure}{0.47\linewidth}
         \includegraphics[width=\linewidth]{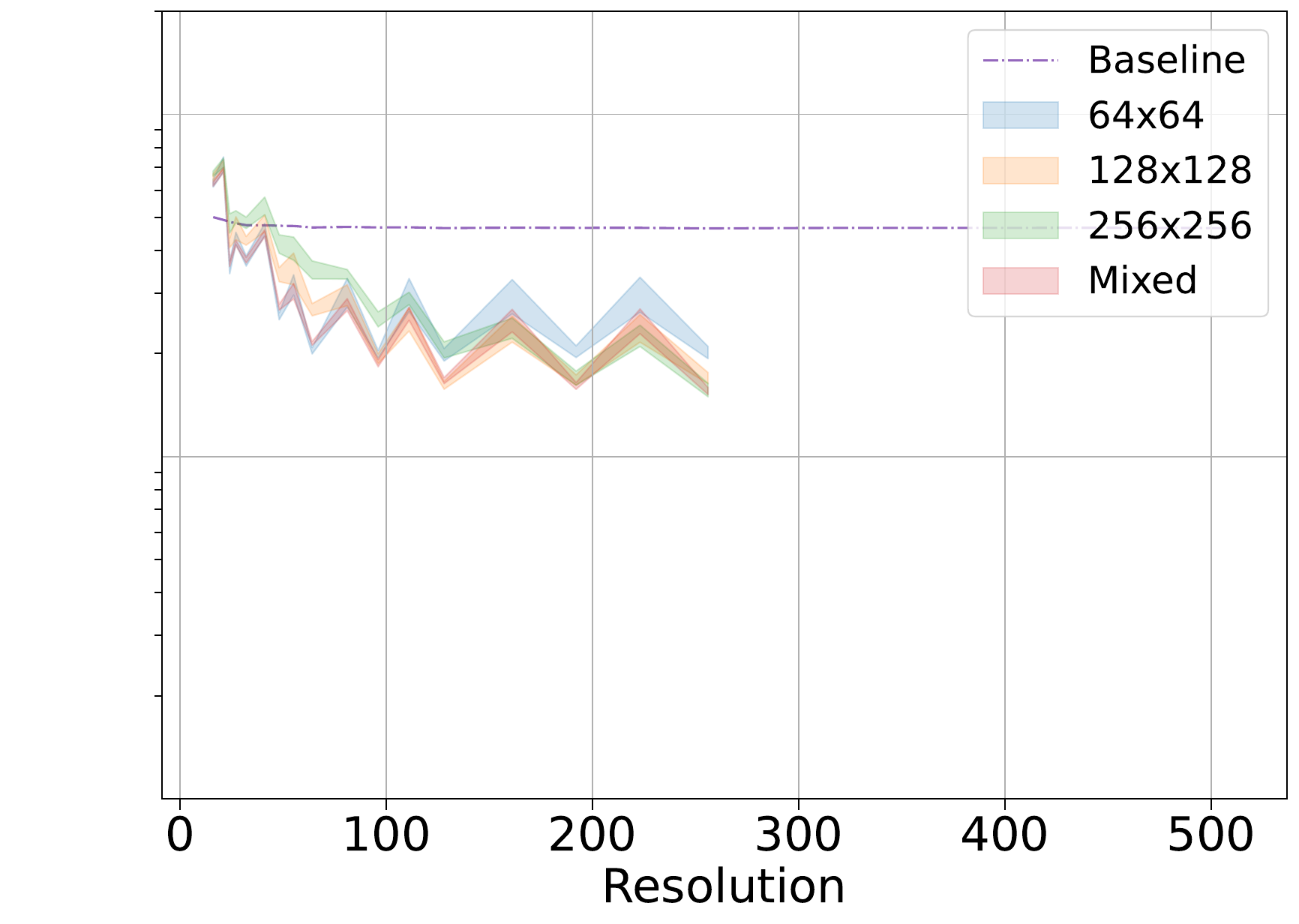}
         \subcaption{U-NO}
    \end{subfigure}
    \caption{Ranges in performance of different architectures on different resolutions. ``Baseline'' indicates the MSE of the initial reconstruction.}
    \label{aa_fig:results2dVars}
\end{figure}

\begin{figure}
    \centering
    \begin{subfigure}{0.47\linewidth}
        \includegraphics[width=\linewidth]{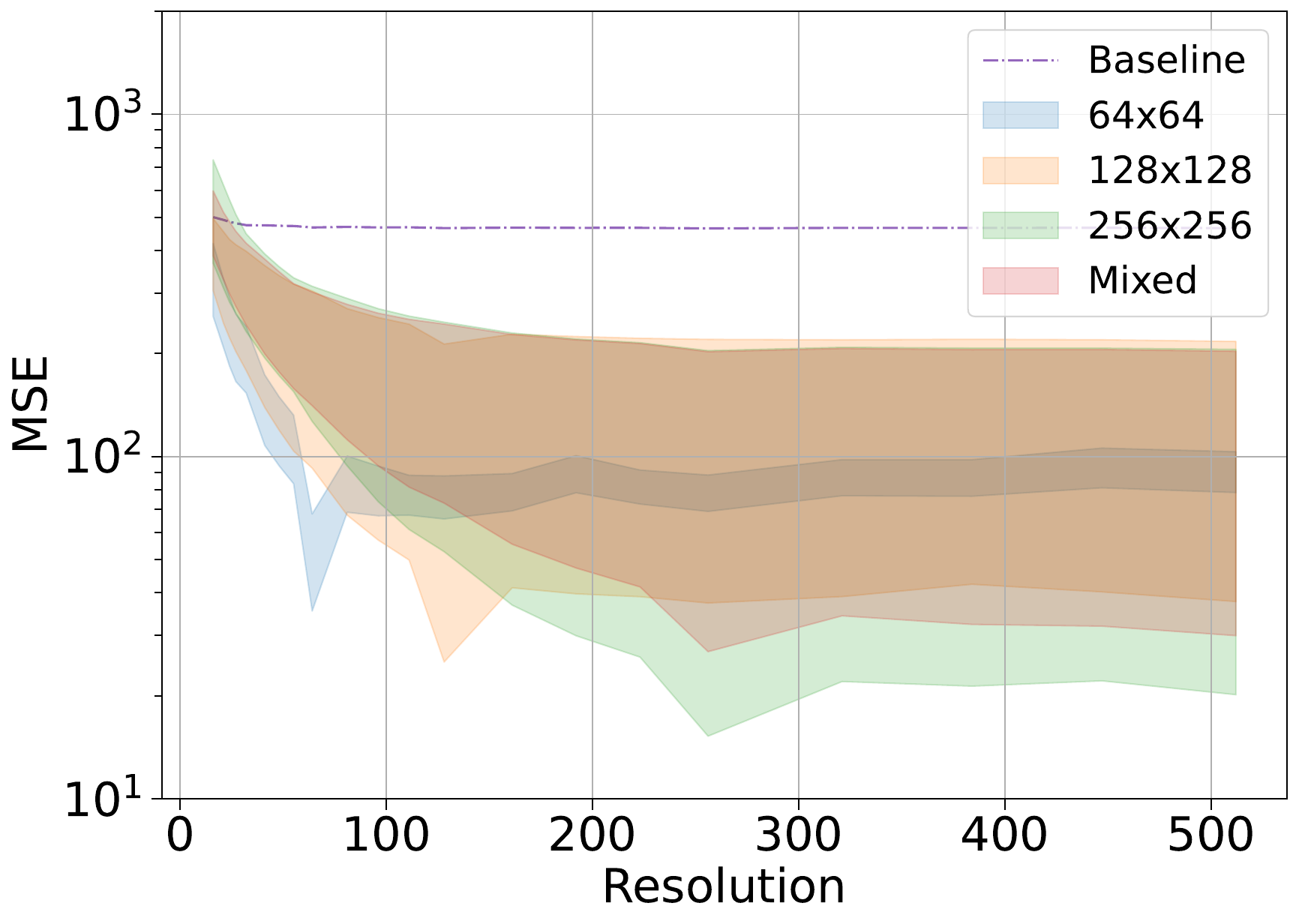}
        \subcaption{U-Net\\\quad}
    \end{subfigure}
    \begin{subfigure}{0.47\linewidth}
         \includegraphics[width=\linewidth]{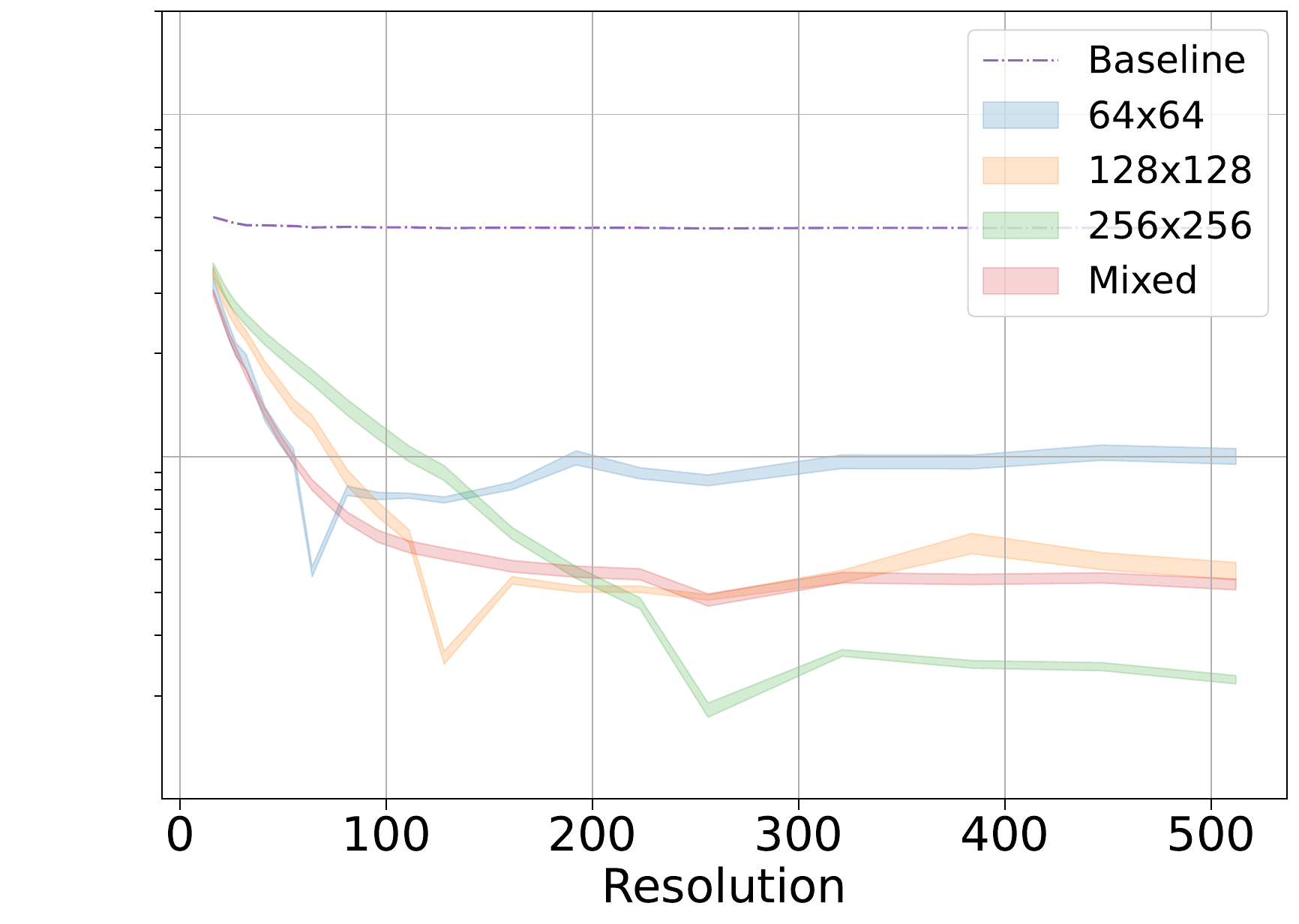}
         \subcaption{Spectral U-Net\\\quad}
    \end{subfigure}
    \begin{subfigure}{0.47\linewidth}
         \includegraphics[width=\linewidth]{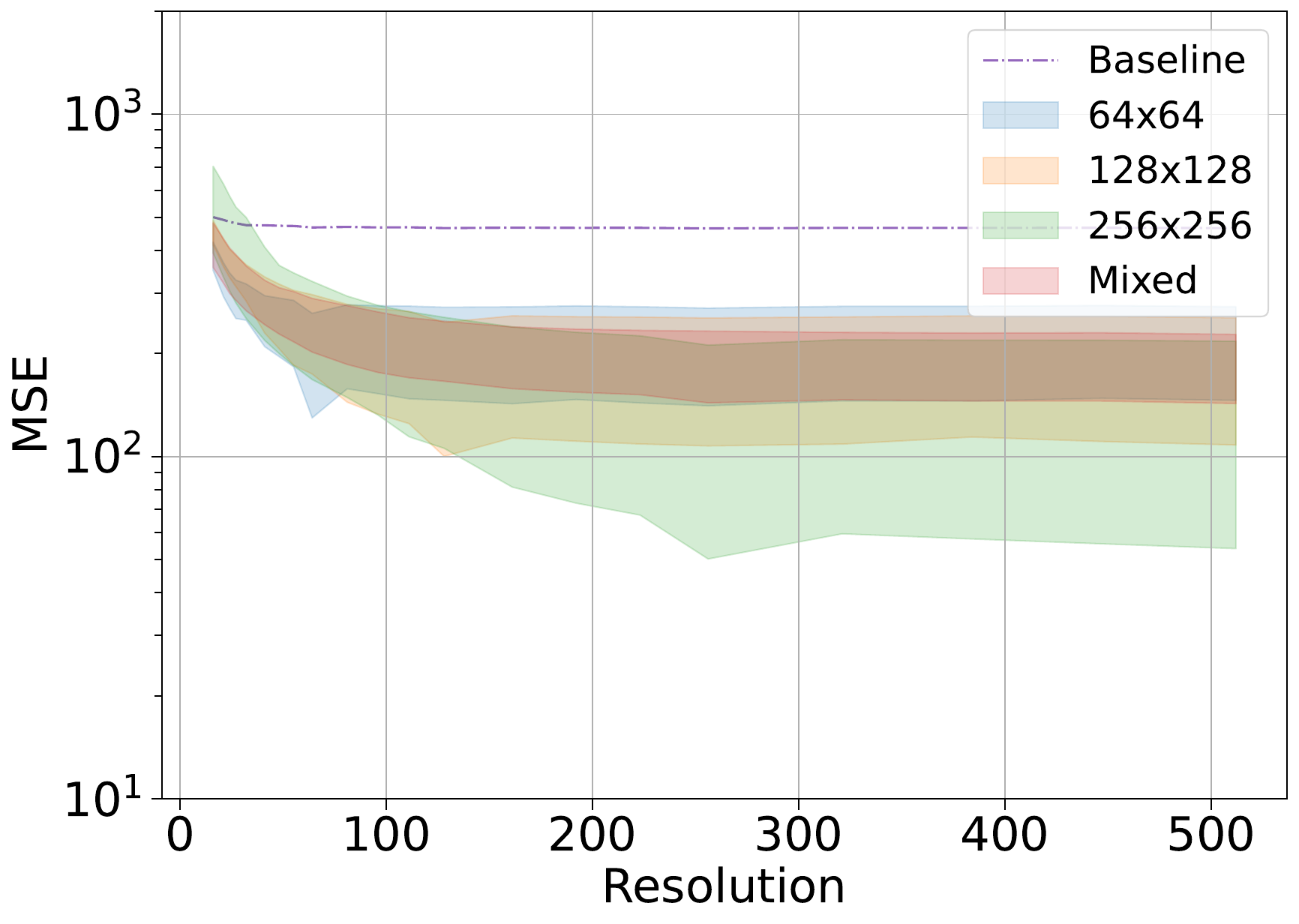}
         \subcaption{Differential U-Net\\\quad}
    \end{subfigure}
    \begin{subfigure}{0.47\linewidth}
        \includegraphics[width=\linewidth]{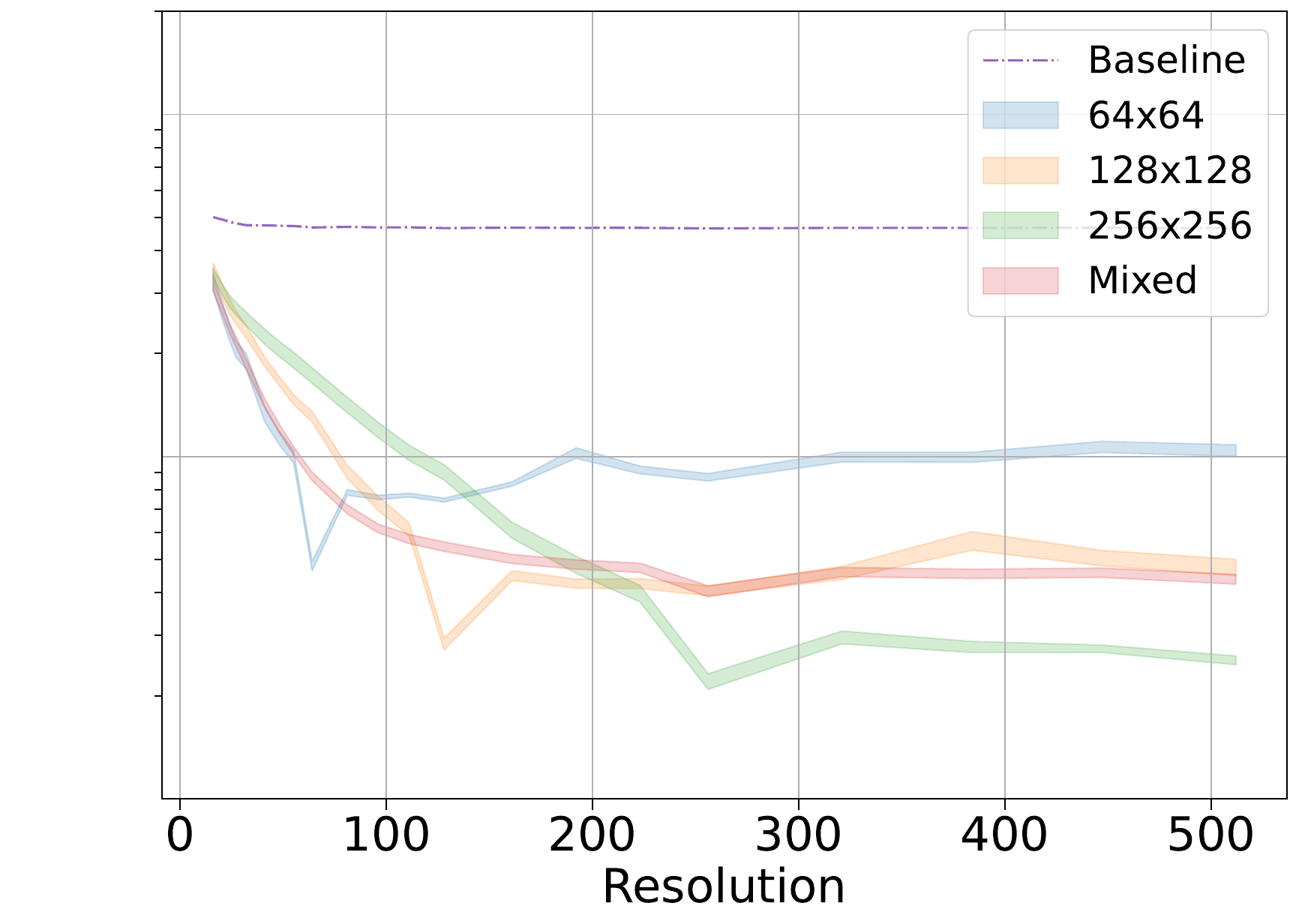}
        \subcaption{Spectral resizing U-Net\\\quad}
    \end{subfigure}
    \begin{subfigure}{0.47\linewidth}
         \includegraphics[width=\linewidth]{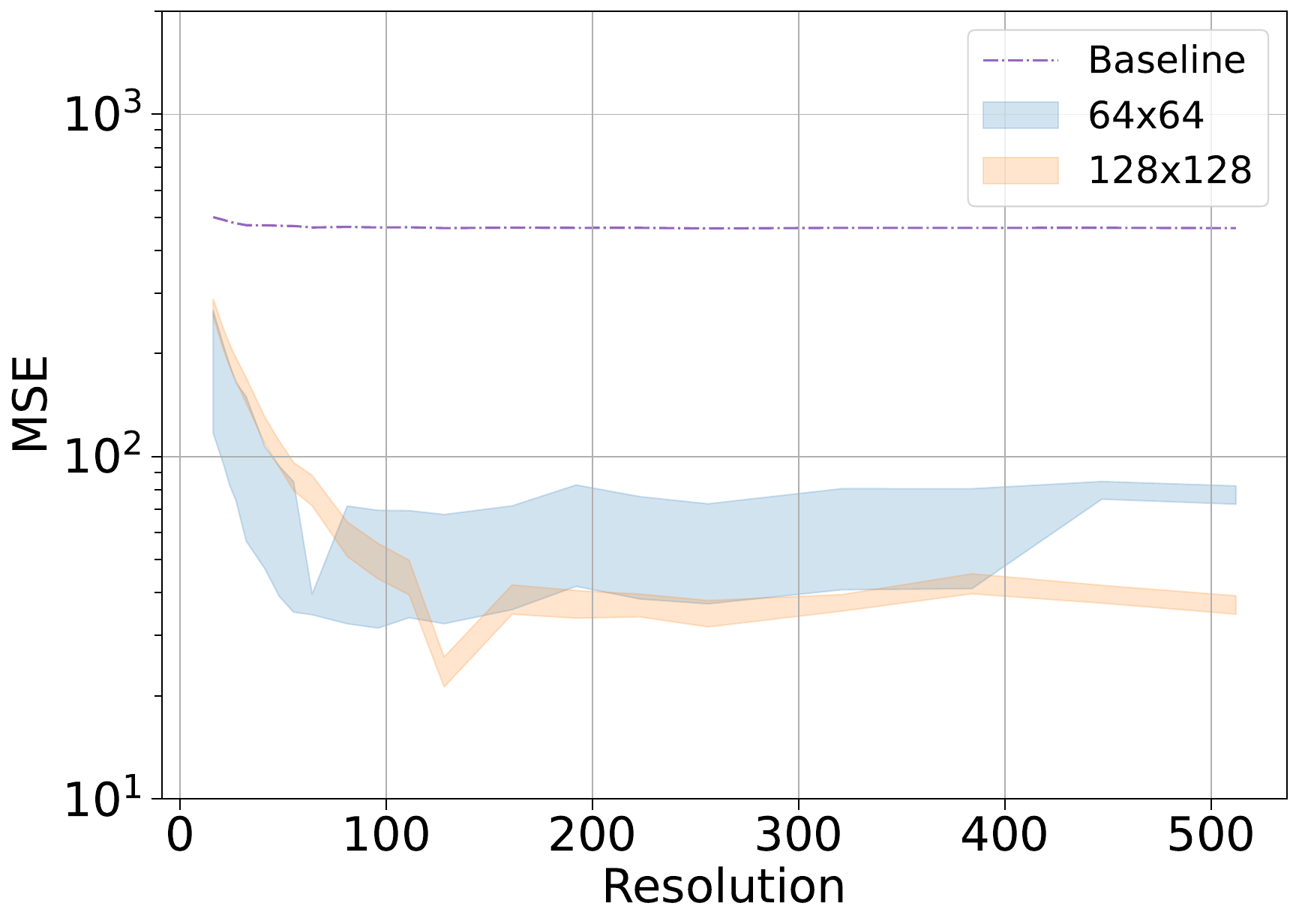}
         \subcaption{CNO}
    \end{subfigure}
    \begin{subfigure}{0.47\linewidth}
         \includegraphics[width=\linewidth]{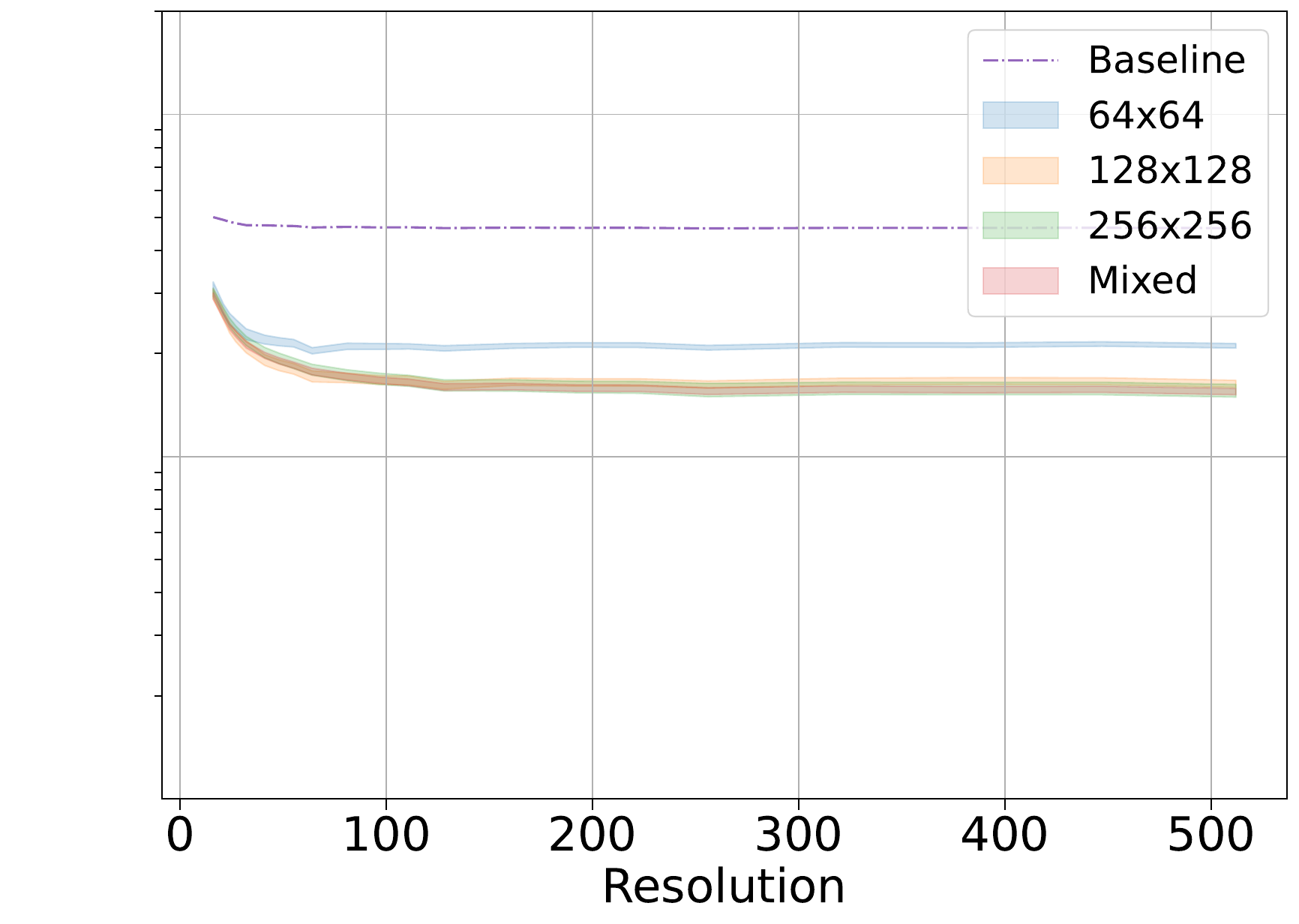}
         \subcaption{U-NO}
    \end{subfigure}
    \caption{Ranges in performance of different architectures on different resolutions resized to their corresponding training resolution. ``Baseline'' indicates the MSE of the initial reconstruction.}
    \label{aa_fig:results2dVarsR}
\end{figure}
\begin{figure}
    \centering
    \includegraphics[width=0.5\linewidth]{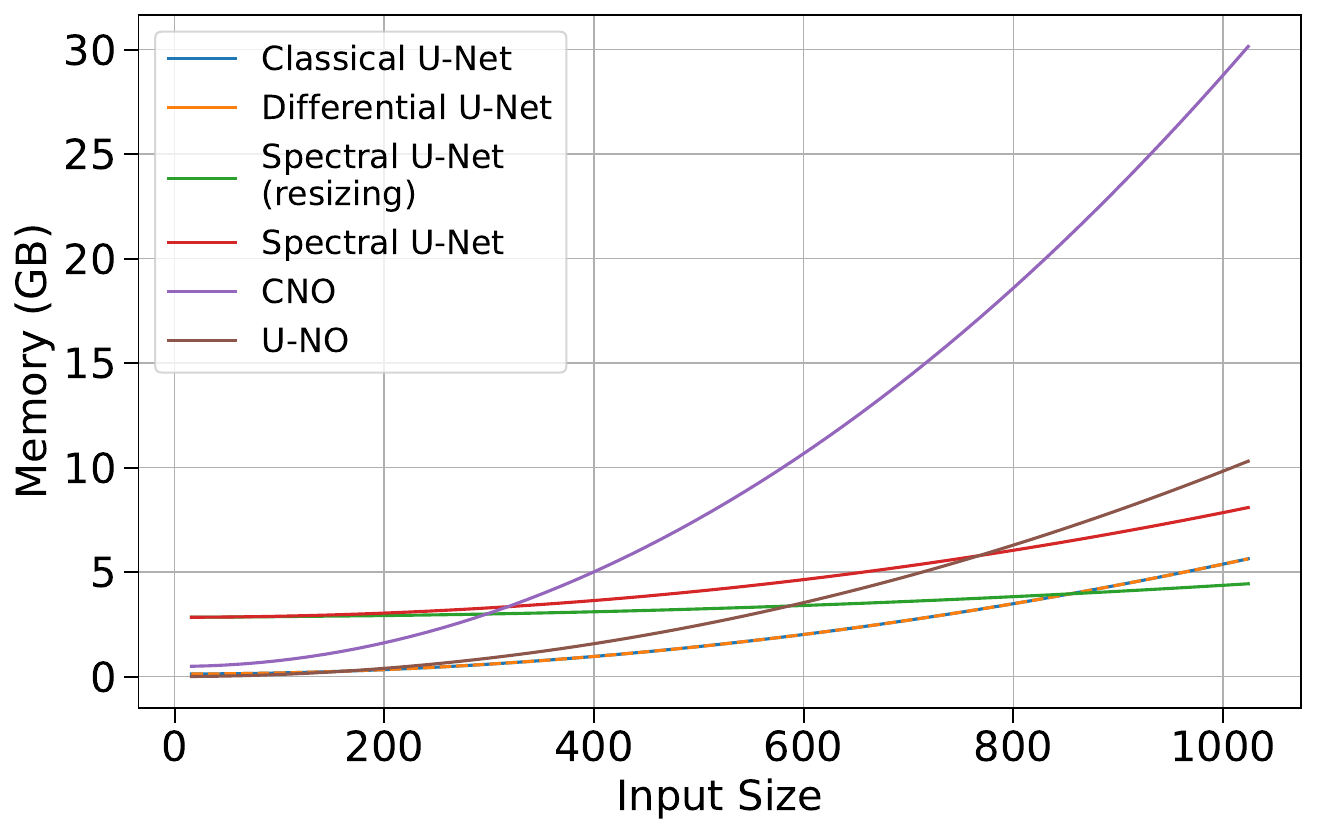}
    \caption{Comparison of the memory requirements for different models on different resolutions.}
    \label{aa_fig:memReq}
\end{figure}
\begin{table}
    \centering
    \begin{tabular}{c|c|c|c|c|c|c}
         & \makecell{Classical\\U-Net} & \makecell{Differential\\U-Net} & \makecell{Spectral resizing\\U-Net} & \makecell{Spectral\\U-Net} & CNO & U-NO  \\ \hline
       \makecell{\# Params\\(Mio.)}& 31.76 & 31.76 & 363.15 & 363.15 & 124.23 & 541.46
    \end{tabular}
    \caption{Comparison of the parameter counts of different models in the configurations used in our experiments.}
    \label{aa_tab:params}
\end{table}

Lastly we would like to highlight once again the vast differences in compute requirements of the different models. Table \ref{aa_tab:params} shows the amount of parameters the models in our experiments utilized. These are independent from training and test resolutions and do not change. Not only does the amount of parameters partially explain the difference in performance between the spectral and more classic versions, but also makes a strong case for the use of common, discretized models, which either generalize well enough by themselves or can be made nearly resolution independent using resized inputs.
Figure \ref{aa_fig:memReq} additionally shows the amount of required memory for a single (gradient-tracking) forward-pass through a certain architecture, making again a strong point for discretized, multiscale architectures.

\section{Conclusion and Outlook}\label{aa_sec:conclusion} The numerical experiments showed that the spectral U-shaped architectures generalize well to unseen data resolutions. To make them practically competitive a more efficient parametrization seems inevitable. Here, we see potential in a more local, spatial parametrization in combination with interpolation.  Another interesting direction is the robustness classical U-Nets showed for a certain range of unseen resolutions. It is open if this observation can be linked to the differential character of (unscaled) convolutional layers from the perspective of finite difference schemes. Towards the approximation of differential operators, our results strongly indicate that it is not enough to simply substitute convolutional layers by differential layers to achieve resolution-invariance. The option to integrate differential layers into an architecture for image processing thus requires further investigations. One perspective could be to combine integral layers and differential layers like it was done in \cite{liuschiaffini2024localno_AA}. A practical shortcoming of the classical U-Net was the large variability of performance we observed in the different training runs. In this regard, further numerical experiments are necessary to see if this issue can be fixed with an optimized training process. 

\section*{Acknowledgement}
This research was supported in part through the Maxwell computational resources operated at Deutsches Elektronen-Synchrotron DESY, Hamburg, Germany. Additional support was received by the OMNI-Cluster compute resources at the University of Siegen. MB and SK acknowledge support from DESY (Hamburg, Germany), a member of the Helmholtz Association HGF. 
AA, MB, SK and MM acknowledge support by the German Research Foundation, projects BU 2327/19-1 and MO 2962/7-1. 
SK would like to thank Tim Roith and Lukas Weigand for the fruitful discussions.

\printbibliography

@InProceedings{ronneberger2015unet_AA,
author="Ronneberger, Olaf
and Fischer, Philipp
and Brox, Thomas",
editor="Navab, Nassir
and Hornegger, Joachim
and Wells, William M.
and Frangi, Alejandro F.",
title="U-Net: Convolutional Networks for Biomedical Image Segmentation",
booktitle="Medical Image Computing and Computer-Assisted Intervention -- MICCAI 2015",
year="2015",
publisher="Springer International Publishing",
address="Cham",
pages="234--241",
isbn="978-3-319-24574-4"
}

@Article{Zhang2022dpirdrunet_AA,
  author    = {Zhang, Kai and Li, Yawei and Zuo, Wangmeng and Zhang, Lei and Van Gool, Luc and Timofte, Radu},
  journal   = {IEEE Transactions on Pattern Analysis and Machine Intelligence},
  title     = {Plug-and-Play Image Restoration With Deep Denoiser Prior},
  year      = {2022},
  issn      = {1939-3539},
  month     = oct,
  number    = {10},
  pages     = {6360--6376},
  volume    = {44},
  doi       = {10.1109/tpami.2021.3088914},
  publisher = {Institute of Electrical and Electronics Engineers (IEEE)},
}

@article{kovachki2023neural_AA,
  title={Neural operator: Learning maps between function spaces with applications to pdes},
  author={Kovachki, Nikola and Li, Zongyi and Liu, Burigede and Azizzadenesheli, Kamyar and Bhattacharya, Kaushik and Stuart, Andrew and Anandkumar, Anima},
  journal={Journal of Machine Learning Research},
  volume={24},
  number={89},
  pages={1--97},
  year={2023}
}

@inproceedings{
li2021fourier_AA,
title={Fourier Neural Operator for Parametric Partial Differential Equations},
author={Zongyi Li and Nikola Borislavov Kovachki and Kamyar Azizzadenesheli and Burigede Liu and Kaushik Bhattacharya and Andrew Stuart and Anima Anandkumar},
booktitle={International Conference on Learning Representations},
year={2021},
}

@article{wen2022ufno_AA,
  title={U-FNO—An enhanced Fourier neural operator-based deep-learning model for multiphase flow},
  author={Wen, Gege and Li, Zongyi and Azizzadenesheli, Kamyar and Anandkumar, Anima and Benson, Sally M},
  journal={Advances in Water Resources},
  volume={163},
  pages={104180},
  year={2022},
  publisher={Elsevier}
}

@article{
rahman2023uno_AA,
title={U-{NO}: U-shaped Neural Operators},
author={Md Ashiqur Rahman and Zachary E Ross and Kamyar Azizzadenesheli},
journal={Transactions on Machine Learning Research},
issn={2835-8856},
year={2023},
}

@Article{Bonneville2025uafno_AA,
  author    = {Bonneville, Christophe and Bieberdorf, Nathan and Hegde, Arun and Asta, Mark and Najm, Habib N. and Capolungo, Laurent and Safta, Cosmin},
  journal   = {npj Computational Materials},
  title     = {Accelerating phase field simulations through a hybrid adaptive Fourier neural operator with U-net backbone},
  year      = {2025},
  issn      = {2057-3960},
  month     = jan,
  number    = {1},
  volume    = {11},
  doi       = {10.1038/s41524-024-01488-z},
  publisher = {Springer Science and Business Media LLC},
}

@inproceedings{liuschiaffini2024localno_AA,
    author = {Liu-Schiaffini, Miguel and Berner, Julius and Bonev, Boris and Kurth, Thorsten and Azizzadenesheli, Kamyar and Anandkumar, Anima},
    title = {Neural operators with localized integral and differential kernels},
    year = {2024},
    publisher = {JMLR.org},
    booktitle = {Proceedings of the 41st International Conference on Machine Learning},
    articleno = {1321},
    numpages = {19},
    location = {Vienna, Austria},
    series = {ICML'24}
    }

@article{berner2025principled_AA,
  title={Principled Approaches for Extending Neural Architectures to Function Spaces for Operator Learning},
  author={Berner, Julius and Liu-Schiaffini, Miguel and Kossaifi, Jean and Duruisseaux, Valentin and Bonev, Boris and Azizzadenesheli, Kamyar and Anandkumar, Anima},
  journal={arXiv:2506.10973},
  year={2025}
}

@article{shelhamer2017fully_AA,
  author={Shelhamer, Evan and Long, Jonathan and Darrell, Trevor},
  journal={IEEE Transactions on Pattern Analysis and Machine Intelligence}, 
  title={Fully Convolutional Networks for Semantic Segmentation}, 
  year={2017},
  volume={39},
  number={4},
  pages={640-651},
  doi={10.1109/TPAMI.2016.2572683}
}

@inproceedings{raonic20cno_AA,
    author = {Raoni\'{c}, Bogdan and Molinaro, Roberto and De Ryck, Tim and Rohner, Tobias and Bartolucci, Francesca and Alaifari, Rima and Mishra, Siddhartha and de B\'{e}zenac, Emmanuel},
    title = {Convolutional neural operators for robust and accurate learning of PDEs},
    year = {2023},
    publisher = {Curran Associates Inc.},
    address = {Red Hook, NY, USA},
    booktitle = {Proceedings of the 37th International Conference on Neural Information Processing Systems},
    articleno = {3376},
    numpages = {14},
    location = {New Orleans, LA, USA},
    series = {NIPS '23}
    }

@article{bartolucci2023representation_AA,
  title={Representation equivalent neural operators: a framework for alias-free operator learning},
  author={Bartolucci, Francesca and de Bezenac, Emmanuel and Raonic, Bogdan and Molinaro, Roberto and Mishra, Siddhartha and Alaifari, Rima},
  journal={Advances in Neural Information Processing Systems},
  volume={36},
  pages={69661--69672},
  year={2023}
}

@Book{Grafakos2014Fourier_AA,
  author    = {Grafakos, Loukas},
  publisher = {Springer New York},
  title     = {Classical Fourier Analysis},
  year      = {2014},
  isbn      = {9781493911943},
  doi       = {10.1007/978-1-4939-1194-3},
  issn      = {2197-5612},
  journal   = {Graduate Texts in Mathematics},
}

@InBook{Kabri2023FNO_AA,
  author    = {Kabri, Samira and Roith, Tim and Tenbrinck, Daniel and Burger, Martin},
  pages     = {236--249},
  publisher = {Springer International Publishing},
  title     = {Resolution-Invariant Image Classification Based on Fourier Neural Operators},
  year      = {2023},
  isbn      = {9783031319754},
  booktitle = {Scale Space and Variational Methods in Computer Vision},
  doi       = {10.1007/978-3-031-31975-4_18},
  issn      = {1611-3349},
}

@Article{Fanaskov2023specneurop_AA,
  author    = {Fanaskov, V. S. and Oseledets, I. V.},
  journal   = {Doklady Mathematics},
  title     = {Spectral Neural Operators},
  year      = {2023},
  issn      = {1531-8362},
  month     = dec,
  number    = {S2},
  pages     = {S226--S232},
  volume    = {108},
  doi       = {10.1134/s1064562423701107},
  publisher = {Pleiades Publishing Ltd},
}

@Book{Troeltzsch2010_AA,
  author    = {Tröltzsch, Fredi},
  publisher = {American Mathematical Society},
  title     = {Optimal Control of Partial Differential Equations},
  year      = {2010},
  isbn      = {9781470411749},
  month     = apr,
  doi       = {10.1090/gsm/112},
  issn      = {1065-7339},
  journal   = {Graduate Studies in Mathematics},
}

@article{ruthotto2020deep_AA,
  title={Deep neural networks motivated by partial differential equations},
  author={Ruthotto, Lars and Haber, Eldad},
  journal={Journal of Mathematical Imaging and Vision},
  volume={62},
  number={3},
  pages={352--364},
  year={2020},
  publisher={Springer}
}

@article{hagemann2023multilevel_AA,
author = {Hagemann, Paul and Mildenberger, Sophie and Ruthotto, Lars and Steidl, Gabriele and Yang, Nicole Tianjiao},
title = {Multilevel Diffusion: Infinite Dimensional Score-Based Diffusion Models for Image Generation},
journal = {SIAM Journal on Mathematics of Data Science},
volume = {7},
number = {3},
pages = {1337-1366},
year = {2025},
doi = {10.1137/23M1614092},
URL = { 
        https://doi.org/10.1137/23M1614092
},
eprint = {
        https://doi.org/10.1137/23M1614092
}
}

@Book{Strikwerda2004_AA,
  author    = {Strikwerda, John C.},
  publisher = {Society for Industrial and Applied Mathematics},
  title     = {Finite Difference Schemes and Partial Differential Equations, Second Edition},
  year      = {2004},
  isbn      = {9780898717938},
  month     = jan,
  doi       = {10.1137/1.9780898717938},
}

@book{Ambrosio2000_AA,
    author = {Ambrosio, Luigi and Fusco, Nicola and Pallara, Diego},
    title = {Functions of Bounded Variation and Free Discontinuity Problems},
    publisher = {Oxford University Press},
    year = {2000},
    month = {03},
    isbn = {9780198502456},
    doi = {10.1093/oso/9780198502456.001.0001},
    url = {https://doi.org/10.1093/oso/9780198502456.001.0001},
}

@Book{Nixon2020_AA,
  author    = {Nixon, Mark S. and Aguado, Alberto S. },
  publisher = {Academic Press},
  title     = {Feature extraction and image processing for computer vision},
  year      = {2020},
  address   = {[Place of publication not identified]},
  edition   = {4th ed.},
  isbn      = {0128149779},
  note      = {Includes bibliographical references and index},
  pagetotal = {1},
  ppn_gvk   = {189276475X},
}

@InProceedings{continuous_conv_AA,
author = {Wang, Shenlong and Suo, Simon and Ma, Wei-Chiu and Pokrovsky, Andrei and Urtasun, Raquel},
title = {Deep Parametric Continuous Convolutional Neural Networks},
booktitle = {Proceedings of the IEEE Conference on Computer Vision and Pattern Recognition (CVPR)},
month = {June},
year = {2018}
}

@Book{Alt2016_AA,
  author    = {Alt, Hans Wilhelm},
  publisher = {Springer London},
  title     = {Linear Functional Analysis},
  year      = {2016},
  isbn      = {9781447172802},
  doi       = {10.1007/978-1-4471-7280-2},
  issn      = {2191-6675},
  journal   = {Universitext},
}

@inproceedings{erisen2024_AA,

  author={Erişen, Serdar},

  booktitle={2024 IEEE International Conference on Computer Vision and Machine Intelligence (CVMI)}, 

  title={SERNet-Former: Segmentation by Efficient-ResNet with Attention-Boosting Gates and Attention-Fusion Networks}, 

  year={2024},

  volume={},

  number={},

  pages={1-6},

  doi={10.1109/CVMI61877.2024.10782648}}

@article{song2019generative_AA,
  title={Generative modeling by estimating gradients of the data distribution},
  author={Song, Yang and Ermon, Stefano},
  journal={Advances in neural information processing systems},
  volume={32},
  year={2019}
}

@article{Feng20Endtoend_AA,
author = {Jinchao Feng and Jianguang Deng and Zhe Li and Zhonghua Sun and Huijing Dou and Kebin Jia},
journal = {Biomed. Opt. Express},
number = {9},
pages = {5321--5340},
publisher = {Optica Publishing Group},
title = {End-to-end Res-Unet based reconstruction algorithm for photoacoustic imaging},
volume = {11},
month = {Sep},
year = {2020},
url = {https://opg.optica.org/boe/abstract.cfm?URI=boe-11-9-5321},
doi = {10.1364/BOE.396598},
}

@article{li2023fourier_AA,
  title={Fourier neural operator with learned deformations for pdes on general geometries},
  author={Li, Zongyi and Huang, Daniel Zhengyu and Liu, Burigede and Anandkumar, Anima},
  journal={Journal of Machine Learning Research},
  volume={24},
  number={388},
  pages={1--26},
  year={2023}
}

@inproceedings{tranfactorized_AA,
  title={Factorized Fourier Neural Operators},
  author={Tran, Alasdair and Mathews, Alexander and Xie, Lexing and Ong, Cheng Soon},
  booktitle={The Eleventh International Conference on Learning Representations},
  year = {2023},
}

@inproceedings{ocampo2023scalable_AA,
  title={Scalable and Equivariant Spherical CNNs by Discrete-Continuous (DISCO) Convolutions},
  author={Ocampo, Jeremy and Price, Matthew Alexander and McEwen, Jason},
  booktitle={The Eleventh International Conference on Learning Representations},
  year = {2023},
}

@Book{Aubert2006_AA,
  author    = {Aubert, Gilles and Kornprobst, Pierre},
  publisher = {Springer New York},
  title     = {Mathematical Problems in Image Processing: Partial Differential Equations and the Calculus of Variations},
  year      = {2006},
  isbn      = {9780387445885},
  doi       = {10.1007/978-0-387-44588-5},
  issn      = {0066-5452},
  journal   = {Applied Mathematical Sciences},
}

@Book{Schoenlieb2015_AA,
  author    = {Schönlieb, Carola-Bibiane},
  publisher = {Cambridge University Press},
  title     = {Partial differential equation methods for image inpainting},
  year      = {2015},
  address   = {New York, NY},
  isbn      = {9781107001008},
  number    = {29},
  series    = {Cambridge monographs on applied and computational mathematics},
  pagetotal = {254},
  ppn_gvk   = {1620434024},
}

@InProceedings{Wei2023superres_AA,
    author    = {Wei, Min and Zhang, Xuesong},
    title     = {Super-Resolution Neural Operator},
    booktitle = {Proceedings of the IEEE/CVF Conference on Computer Vision and Pattern Recognition (CVPR)},
    month     = {June},
    year      = {2023},
    pages     = {18247-18256}
}

@INPROCEEDINGS{Johnny2022fno_AA,
  author={Johnny, Williamson and Brigido, Hatzinakis and Ladeira, Marcelo and Souza, Joao Carlos Felix},
  booktitle={2022 17th Iberian Conference on Information Systems and Technologies (CISTI)}, 
  title={Fourier Neural Operator for Image Classification}, 
  year={2022},
  volume={},
  number={},
  pages={1-6},
  doi={10.23919/CISTI54924.2022.9820128}}

@article{liu2025uffno_AA,
title = {U-shaped factorized Fourier neural operator for solving partial differential equations},
journal = {Computers \& Mathematics with Applications},
volume = {196},
pages = {233-245},
year = {2025},
issn = {0898-1221},
doi = {https://doi.org/10.1016/j.camwa.2025.07.013},
author = {Hui Liu and Peizhi Zhao and Tao Song},
}

@article{chen2018neural_AA,
  title={Neural ordinary differential equations},
  author={Chen, Ricky TQ and Rubanova, Yulia and Bettencourt, Jesse and Duvenaud, David K},
  journal={Advances in neural information processing systems},
  volume={31},
  year={2018}
}

@article{e2017dynamical_AA,
title = {A Proposal on Machine Learning via Dynamical
Systems},
  author={Weinan E},
  journal={Communications in Mathematics and Statistics},
  volume={5},
  year={2017}
}

@article{haber2017stable_AA,
  title={Stable architectures for deep neural networks},
  author={Haber, Eldad and Ruthotto, Lars},
  journal={Inverse problems},
  volume={34},
  number={1},
  pages={014004},
  year={2017},
  publisher={IOP Publishing}
}

@article{thorpe2023deep_AA,
  title={Deep limits of residual neural networks},
  author={Thorpe, Matthew and van Gennip, Yves},
  journal={Research in the Mathematical Sciences},
  volume={10},
  number={1},
  pages={6},
  year={2023},
  publisher={Springer}
}

@inproceedings{venkatakrishnan2013pnp_AA,
  author={Venkatakrishnan, Singanallur V. and Bouman, Charles A. and Wohlberg, Brendt},
  booktitle={2013 IEEE Global Conference on Signal and Information Processing}, 
  title={Plug-and-Play priors for model based reconstruction}, 
  year={2013},
  volume={},
  number={},
  pages={945-948},
  doi={10.1109/GlobalSIP.2013.6737048}}

@inproceedings{hurault2023gradient_AA,
  title={Gradient Step Denoiser for convergent Plug-and-Play},
  author={Hurault, Samuel and Leclaire, Arthur and Papadakis, Nicolas},
  booktitle={International Conference on Learning Representations},
  year = {2023},
}

@inproceedings{feng2023score_AA,
  title={Score-based diffusion models as principled priors for inverse imaging},
  author={Feng, Berthy T and Smith, Jamie and Rubinstein, Michael and Chang, Huiwen and Bouman, Katherine L and Freeman, William T},
  booktitle={Proceedings of the IEEE/CVF International Conference on Computer Vision},
  pages={10520--10531},
  year={2023}
}

@misc{Welker2025Ptycho_AA, 
    author = {Simon Welker and Lorenz Kuger and Tim Roith and Berthy Feng and Martin Burger and Timo Gerkmann and Henry Chapman},
    title = {Position-Blind Ptychography: Viability of image reconstruction via data-driven variational inference (in preparation)},
    year = {2025} }

@book{engl1996regularization_AA,
  title={Regularization of inverse problems},
  author={Engl, Heinz Werner and Hanke, Martin and Neubauer, Andreas},
  volume={375},
  year={1996},
  publisher={Kluwer},
  address ={Dordrecht}
}

@article{auras2024overview,
author = {Auras, Alexander and Gandikota, Kanchana Vaishnavi and Droege, Hannah and Moeller, Michael},
title = {Robustness and exploration of variational and machine learning approaches to inverse problems: An overview},
journal = {GAMM-Mitteilungen},
volume = {47},
number = {4},
pages = {e202470003},
doi = {https://doi.org/10.1002/gamm.202470003},
year = {2024}
}

@article{kovachki2021fnoapprox_AA,
  author  = {Nikola Kovachki and Samuel Lanthaler and Siddhartha Mishra},
  title   = {On Universal Approximation and Error Bounds for Fourier Neural Operators},
  journal = {Journal of Machine Learning Research},
  year    = {2021},
  volume  = {22},
  number  = {290},
  pages   = {1-76},
  url     = {http://jmlr.org/papers/v22/21-0806.html}
}

@article{deHoop2023operatorlearning_AA,
author = {de Hoop, Maarten V. and Kovachki, Nikola B. and Nelsen, Nicholas H. and Stuart, Andrew M.},
title = {Convergence Rates for Learning Linear Operators from Noisy Data},
journal = {SIAM/ASA Journal on Uncertainty Quantification},
volume = {11},
number = {2},
pages = {480-513},
year = {2023},
doi = {10.1137/21M1442942},
}

@misc{Reinhardt2024Operator_AA,
    Author = {Reinhardt, Niklas and Wang, Sven and Zech, Jakob},
    Title = {Statistical Learning Theory for Neural Operators},
    Year = {2024},
    Journal = {arXiv preprint},
    Url = {https://arxiv.org/abs/2412.17582}
}

@inproceedings{koshizuka2024express_AA,
 author = {Koshizuka, Takeshi and Fujisawa, Masahiro and Tanaka, Yusuke and Sato, Issei},
 booktitle = {Advances in Neural Information Processing Systems},
 editor = {A. Globerson and L. Mackey and D. Belgrave and A. Fan and U. Paquet and J. Tomczak and C. Zhang},
 pages = {11021--11060},
 publisher = {Curran Associates, Inc.},
 title = {Understanding the Expressivity and Trainability of Fourier Neural Operator: A Mean-Field Perspective},
 url = {https://proceedings.neurips.cc/paper_files/paper/2024/file/14da7aea05debb963b3d8d46449d51a0-Paper-Conference.pdf},
 volume = {37},
 year = {2024}
}

@article{rowbottom2025multi_AA,
  title={Multi-level monte carlo training of neural operators},
  author={Rowbottom, James and Fresca, Stefania and Lio, Pietro and Sch{\"o}nlieb, Carola-Bibiane and Boull{\'e}, Nicolas},
  journal={arXiv preprint arXiv:2505.12940},
  year={2025}
}

@Book{Quarteroni2007_AA,
  author    = {Quarteroni, Alfio and Sacco, Riccardo and Saleri, Fausto},
  publisher = {Springer New York},
  title     = {Numerical Mathematics},
  year      = {2007},
  isbn      = {9780387227504},
  doi       = {10.1007/b98885},
  issn      = {2196-9949},
  journal   = {Texts in Applied Mathematics},
}

@inproceedings{Xu24Provably_AA,
  author       = {Xingyu Xu and
                  Yuejie Chi},
  editor       = {Amir Globersons and
                  Lester Mackey and
                  Danielle Belgrave and
                  Angela Fan and
                  Ulrich Paquet and
                  Jakub M. Tomczak and
                  Cheng Zhang},
  title        = {Provably Robust Score-Based Diffusion Posterior Sampling for Plug-and-Play
                  Image Reconstruction},
  booktitle    = {Advances in Neural Information Processing Systems 38: Annual Conference
                  on Neural Information Processing Systems 2024, NeurIPS 2024, Vancouver,
                  BC, Canada, December 10 - 15, 2024},
  year         = {2024},
  url          = {http://papers.nips.cc/paper\_files/paper/2024/hash/3fa2d2b637122007845a2fbb7c21453b-Abstract-Conference.html},
  bibsource    = {dblp computer science bibliography, https://dblp.org}
}

@article{Lim25ScoreBased_AA,
  author  = {Jae Hyun Lim and Nikola B. Kovachki and Ricardo Baptista and Christopher Beckham and Kamyar Azizzadenesheli and Jean Kossaifi and Vikram Voleti and Jiaming Song and Karsten Kreis and Jan Kautz and Christopher Pal and Arash Vahdat and Anima Anandkumar},
  title   = {Score-Based Diffusion Models in Function Space},
  journal = {Journal of Machine Learning Research},
  year    = {2025},
  volume  = {26},
  number  = {158},
  pages   = {1--62},
  url     = {http://jmlr.org/papers/v26/23-1472.html}
}

@incollection{steidl2025flowmatching_AA,
  title={{Flow Matching: Markov kernels, stochastic processes and transport plans}},
  author={Wald, C. and Steidl, G.},
  booktitle={Variational and Information Flows in Machine Learning and Optimal Transport, Oberwolfach Seminars. Vol. 56},
editors = {Wuchen Li and
Bernhard Schmitzer and
Gabriele Steidl and
Francois-Xavier Vialard and
Christian Wald},
publisher = {Birkh\"auser},
  pages={185--254},
  year={2025}
}
\end{document}